\documentclass[final,times]{elsarticle}

\usepackage{epsfig}
\usepackage{amssymb}
\usepackage{amsthm}

\usepackage{mathtools}
\usepackage{enumitem}
\usepackage{caption}
\usepackage{array,booktabs}
\usepackage[hyphens]{url}
\usepackage{psfrag,subfigure}
\usepackage[hidelinks]{hyperref}
\hypersetup{breaklinks=true}
\usepackage[english]{babel}
\usepackage{amsfonts, verbatim}
\usepackage{amsmath}

\newcommand{\R}{{\mathbb{R}}}
\newcommand{\integers}{{\mathbb{Z}}}
\newcommand{\naturals}{{\mathbb{N}}}
\renewcommand{\d}{{\rm{d}}}

\newcommand{\modulo}[1]{{\left|#1\right|}}
\newcommand{\norma}[1]{{\left\|#1\right\|}}
\newcommand{\Cc}[1]{\mathbf{C_{\rm c}^{#1}}}

\newcommand{\C}[1]{\mathbf{C^{#1}}}
\newtheorem{thm}{Theorem}[section]
\newtheorem{lemma}[thm]{Lemma}
\newtheorem{proposition}[thm]{Proposition}
\newtheorem{rem}[thm]{Remark}

\newtheorem{definition}[thm]{Definition}
\theoremstyle{definition}
\newtheorem{example}[thm]{Example}
 \renewcommand{\L}[1]{\mathbf{L^{#1}}}
\newcommand{\Lloc}[1]{\mathbf{L^{#1}_{loc}}}
\newcommand{\sign}{\mathrm{sign}}
\newcommand{\BV}{\mathbf{BV}}
\newcommand{\tv}{\mathrm{TV}}
\usepackage{psfrag}
\newcommand{\PC}{\mathbf{PC}}
\newcommand{\caratt}[1]{{\displaystyle\chi_{\strut{\textstyle #1}}}}

\journal{}

\usepackage{calc}
\usepackage{lipsum}
\makeatletter
\def\ps@pprintTitle{%
 \let\@oddhead\@empty
 \let\@evenhead\@empty
 \def\@oddfoot{}%
 \let\@evenfoot\@oddfoot}
\makeatother

\begin{document}

\begin{frontmatter}



\title{Entropy solutions for a traffic model with phase transitions}


\author[L'Aquila]{Mohamed Benyahia}
\ead{benyahia.ramiz@gmail.com}

\author[Warsaw]{Massimiliano D.~Rosini\corref{mycorrespondingauthor}}
\cortext[mycorrespondingauthor]{Corresponding author}
\ead{mrosini@umcs.lublin.pl}

\address[L'Aquila]{Gran Sasso Science Institute\\
Viale F. Crispi 7, 67100 L'Aquila, Italy}

\address[Warsaw]{
Instytut Matematyki, Uniwersytet Marii Curie-Sk\l odowskiej\\
Plac Marii Curie-Sk\l odowskiej 1, 20-031 Lublin, Poland}

\begin{abstract}
In this paper, we consider the two phases macroscopic traffic model introduced in [P.~Goatin, The Aw-Rascle vehicular traffic flow with phase transitions, Mathematical and Computer Modeling 44 (2006) 287-303].
We first apply the wave-front tracking method to prove existence and a priori bounds for weak solutions.
Then, in the case the characteristic field corresponding to the free phase is linearly degenerate, we prove that the obtained weak solutions are in fact entropy solutions \emph{\`a la} Kruzhkov.
The case of solutions attaining values at the vacuum is considered.
We also present an explicit numerical example to describe some qualitative features of the solutions.
\end{abstract}


\end{frontmatter}



\section{Introduction}

This paper deals with macroscopic modelling of traffic flows.
The existing literature on macroscopic models for traffic flows is already vast and characterized by contributions motivated by their real life applications, as the surveys~\cite{Bellomo2002, bellomo2011modeling, survey2013, piccolitosinreview, survey2014} and the books~\cite{garavello2006traffic, Rosinibook} demonstrate.

The macroscopic variables that translate the discrete nature of traffic into 
continuous variables are the velocity $v$, namely the space covered per unit time by the vehicles, the density $\rho$, namely the number of vehicles per unit length of the road, and the flow $f$, namely the number of vehicles per unit time.
By definition we have that
\begin{equation}\label{eq:macro1}
f=\rho \, v.
\end{equation}
Clearly, the macroscopic variables are in general functions of time $t>0$ and space $x \in \R$.
By imposing the conservation of the number of vehicles along a road with no entrances or exits we deduce the scalar conservation law
\begin{equation}\label{eq:macro2}
\rho_t + f_x = 0.
\end{equation}
Since the system~\eqref{eq:macro1}, \eqref{eq:macro2} has three unknown variables, a further condition has to be imposed.
There are two main approaches to do it.
First order macroscopic models close the system~\eqref{eq:macro1}, \eqref{eq:macro2} by giving beforehand an explicit expression of one of the three unknown variables in terms of the remaining two. The prototype of the first order models is the Lighthill, Whitham~\cite{LWR1} and Richards~\cite{LWR2} model (LWR).
The basic assumption of LWR is that the velocity of any driver depends on the density alone, namely
\[v=V(\rho).\]
The function $V \colon [0,\rho_{\max}]\to \left[0,v_{\max}\right]$ is given beforehand and is assumed to be $\C1$, non-increasing, with $V(0) = v_{\max}$ and $V(\rho_{\max})=0$, where $\rho_{\max}$ is the maximal density corresponding to the situation in which the vehicles are bumper to bumper, and $v_{\max}$ is the maximal speed corresponding to the free road.
As a result, LWR is given by the scalar conservation law
\[
\rho_t+[\rho\,V(\rho)]_x=0.
\]
Second order macroscopic models close the system~\eqref{eq:macro1}, \eqref{eq:macro2} by adding a further conservation law.
The most celebrated second order macroscopic model is the Aw, Rascle~\cite{ARZ1} and Zhang~\cite{ARZ2} model (ARZ). Away from the vacuum, ARZ writes
\begin{align*}
&\rho_t+[\rho\,v]_x=0,&
&\left[\rho\left(v+P(\rho)\right)\right]_t + \left[\rho\left(v+P(\rho)\right)v\right]_x=0,
\end{align*}
where the ``pressure'' function $P(\rho)$ plays the role of an anticipation factor, taking into account drivers' reactions to the state of traffic in front of them.

The main drawback of LWR is the unrealistic behaviour of the drivers, who take into account the slightest change in the density and adjust instantaneously their velocities according to the densities they are experiencing (which implies infinite acceleration of the vehicles).
Moreover, experimental data show that the fundamental diagram $(\rho,f)$ is given by a cloud of points rather than being the support of a map $\rho \mapsto \left[\rho \, v(\rho)\right]$.
ARZ can be interpreted as a generalization of LWR, possessing a family of fundamental diagram curves, rather than a single one.
For this reason ARZ avoids the drawbacks of LWR listed above.
Moreover, traffic hysteresis, which means that for the same distance headway drivers choose a different speed during acceleration from that chosen during deceleration, can be reproduced with ARZ but not with LWR.

On the other hand, the system describing ARZ degenerates into just one equation at the vacuum $\rho = 0$.
In particular, as pointed out in~\cite{ARZ1}, the solutions to ARZ fail to depend continuously on the initial data in any neighbourhood of $ \rho=0$; moreover, as observed in~\cite{godvik}, the solutions may experience a sudden increase of the total variation as the vacuum appears.

For the above reasons, Goatin~\cite{goatin2006aw} proposes to couple ARZ with LWR by introducing a two phase transition model.
More precisely, the phase transition model proposed in~\cite{goatin2006aw} describes the dynamics in the free flow and those in the congested flow respectively with LWR and ARZ.
In fact, this allows to better fit the experimental data, and has also the advantage of correcting the exposed drawbacks of LWR in the congested traffic and of ARZ at the vacuum.

In~\cite{GaravelloPiccoli2009} the authors point out that the model proposed in~\cite{goatin2006aw} doesn't satisfy properties that they consider necessary to model appropriately urban road networks, namely:
\begin{itemize}[leftmargin=*]
\item
Vehicles stop only at maximum density, i.e.~the velocity $v$ is zero if and only if the density $\rho$ is equal to the maximum density possible $\rho_{\max}$.
\item
The density at a red traffic light is the maximum possible, i.e.~$\rho_{\max}$.
\end{itemize}
However, ARZ can be interpreted as multi-population traffic model, see for instance~\cite{BCM2order, BCJMU-order2} and~\cite{MarcoSimoneMax-ARZ1} for a microscopic interpretation.
In particular, the vehicles are allowed to have different lengths.
On one hand this is supported by the real life experience, on the other hand this necessarily implies that the above two conditions are not satisfied.
Finally, Laval shows in~\cite{hysteresis} that traffic hysteresis is better explained in terms of heterogeneous drivers rather than acceleration and deceleration phases.
This suggests that the ability of ARZ to reproduce traffic hysteresis relies also on its ability to consider different driver behaviours.

Aim of the present paper is to generalize the model introduced in~\cite{goatin2006aw} and to prove an existence result by exploiting the recent achievements obtained in~\cite{BCM2order, BCJMU-order2} for ARZ.
We introduce a definition of entropy solution \emph{\`a la} Kruzhkov~\cite{kruzkov} and prove the corresponding existence results.
To the best of the authors' knowledge, there are no references in the literature to a rigorous definition of entropy solution to~\eqref{eq:model} or to other phase transition models for vehicular traffic, see for instance~\cite{BlandinWorkGoatinPiccoliBayen, BorscheKimathiKlar, Colombophasetransitions, ColomboGoatinPriuli, ColomboMarcelliniRascle, goatin2006aw, Marcellini2014}.
The key tools used to prove the existence of an entropy solution are the wave-front tracking method~\cite{DafermosWFT} and the estimates that permit to exploit the wave-front tracking method in the $\BV$ functional setting.
A Temple like functional is proposed in order to compensate, via a kind of potential, the possible increase of the wave fronts.
We choose to apply the wave-front tracking scheme because it is able to operate also in the case with point constraints on the flow, when non-classical shocks~\cite{LeflochBook} at the constraint locations have to be taken into account.
We recall that the concept of point constraints was introduced in the framework of vehicular traffic in~\cite{ColomboGoatinConstraint} and in the framework of crowd dynamics in~\cite{ColomboRosiniPedestrain1}.
We defer to~\cite{AndreianovDonadelloRosiniMBE, BCM2order, AndreianovDonadelloRosini1, scontrainte, BCJMU-order2, BCMUqualitative, Cances20123036, chalonsgoatinseguin, ColomboGoatinRosini1, ColomboGoatinRosini3, ColomboGoatinRosini2, DelleMonache2014435, delle2014scalar, Lattanzio201150, Rosini201363} for further developments and applications also to crowd dynamics.

In order to properly describe phase transition model, we use the coordinates given by the extension of the Riemann invariants of ARZ, rather than the conserved variables of ARZ, see~\eqref{eq:change} for the definition of the change of coordinates.
The choice of the extended Riemann invariants as independent variables is in fact convenient to describe the Riemann solver $\mathcal{R}$ and ease the forthcoming analysis, as the total variation of the solutions in these coordinates does not increase, see~\cite{BCM2order, BCJMU-order2, FerreiraKondo2010,godvik,Lu20112797} where this property is exploited to prove existence results for ARZ.

The outline of the paper is as follows.
In Section~\ref{sec:notations} we recall the phase transition model introduced in~\cite{goatin2006aw} together with its main properties.
More precisely, in Section~\ref{sec:assandnot} we introduce the notations and the assumptions needed to state in Section~\ref{sec:PT} the two phase model~\eqref{eq:model}; then in Section~\ref{sec:we} we introduce the concepts of weak and entropy solutions to the Cauchy problem for the two phase model and expose Theorem~\ref{thm:1}, that is the main result of the paper.
In Section~\ref{sec:ex} we apply the model to compute an explicit example reproducing the effects of a traffic light on the traffic along a road.
In Section~\ref{sec:wft} we describe the wave-front tracking algorithm used to construct approximate solutions and prove their convergence in $\Lloc1$.
Finally, in Section~\ref{sec:ts} we collect the technical proofs.

\section{Assumptions and main result}\label{sec:notations}

In this section we introduce a two phase transition model that combines LWR and ARZ to describe respectively the free and the congested flow based on that one proposed in~\cite{goatin2006aw}.
We conclude this section by giving our main result in Theorem~\ref{thm:1}.

\subsection{Assumptions and notations}\label{sec:assandnot}

Before writing the two phase model~\eqref{eq:model}, we need to introduce some notations.
Fix $R_{\rm f}'' > 0$ and consider a map $v_{\rm f} \colon [0,R_{\rm f}''] \to \R_+$ such that
\begin{equation}\tag{{\rm\bf H1}}\label{H1}
\begin{array}{c}
v_{\rm f} \in \C2([0,R_{\rm f}''];\R_+),
\quad
v_{\rm f}(R_{\rm f}'') > 0,
\\[5pt]
v_{\rm f}'(\rho) \leq 0,
~
v_{\rm f}(\rho) + \rho \, v_{\rm f}'(\rho) > 0,
~
2v_{\rm f}'(\rho) + \rho \, v_{\rm f}''(\rho) \leq 0
~
\text{ for every }\rho \in [0,R_{\rm f}''].
\end{array}
\end{equation}
Fix $R_{\rm f}' \in ~]0,R_{\rm f}''[$ and consider $p \colon [R_{\rm f}',+\infty[~ \to \R$ such that
\begin{align*}\tag{{\rm\bf H2}}\label{H2}
    &p \in \C2([R_{\rm f}',+\infty[;\R),&
    &p'(\rho)>0,&
    &2p'(\rho) + \rho \, p''(\rho)>0&
    \text{for every }\rho\geq R_{\rm f}'.
\end{align*}
Recall that~\eqref{H1} and~\eqref{H2} are the basic assumptions of respectively LWR and ARZ.

We also require the following compatibility condition between $v_{\rm f}$ and $p$ in order to ensure the \emph{capacity drop}~\cite{kerner2004physics} in the passage from the free phase to the congested phase
\begin{align}\tag{{\rm\bf H3}}\label{H3}
&\text{the map $\rho \mapsto v_{\rm f}(\rho)+p(\rho)$ is increasing in }[R_{\rm f}',R_{\rm f}'']&
&\text{ and }&
&v_{\rm f}(\rho) < \rho \, p'(\rho) \text{ for every }\rho \in [R_{\rm f}',R_{\rm f}''].
\end{align}

\begin{example}
The simplest and typical expression for the velocity in a free flow is the linear one~\cite{Greenshields}
\begin{equation}\label{eq:ex1v}
v_{\rm f}(\rho) \doteq v_{\max} \left[1-\frac{\rho}{R}\right],
\end{equation}
where $R>0$ is a parameter and $v_{\max}>0$ is the maximal velocity.
Clearly, the above velocity satisfies the condition~\eqref{H1} if and only if $0 < 2 R_{\rm f}''< R$.

In~\cite{AKMR2002} the authors consider as pressure function
\begin{equation}\label{eq:ex1p}
p(\rho) \doteq
\begin{cases}
\dfrac{v_{\rm ref}}{\gamma} \left[\dfrac{\rho}{\rho_{\max}}\right]^\gamma,& \gamma>0,\\[7pt]
v_{\rm ref} \, \log\left[\dfrac{\rho}{\rho_{\max}}\right],& \gamma=0,
\end{cases}
\end{equation}
where $v_{\rm ref}>0$ is a reference velocity.
The above choice reduces to the original one proposed in~\cite{ARZ1} when $v_{\rm ref}/(\gamma\,\rho_{\max}^\gamma) = 1$, and to that one proposed in~\cite{BagneriniColomboCorli,goatin2006aw} when $\gamma=0$.
The condition~\eqref{H2} is satisfied by the above expression of $p$ for any $\gamma\ge0$.
Moreover, the above choice for $v_{\rm f}$ and $p$ satisfy also~\eqref{H3} if and only if
\begin{align*}
&\frac{v_{\rm ref}}{v_{\max}}
>
\begin{cases}
\max\left\{
\left[1-\dfrac{R_{\rm f}'}{R}\right] \left[\dfrac{\rho_{\max}}{R_{\rm f}'}\right]^\gamma
,
\dfrac{R_{\rm f}'}{R} \left[\dfrac{\rho_{\max}}{R_{\rm f}'}\right]^\gamma
,
\dfrac{R_{\rm f}''}{R} \left[\dfrac{\rho_{\max}}{R_{\rm f}''}\right]^\gamma
\right\}
&\text{ if }
\gamma > 0,
\\[10pt]
1-\dfrac{R_{\rm f}''}{R}
&\text{ if }
\gamma = 0.
\end{cases}
\end{align*}
For completeness, we finally recall that in~\cite{BerthelinDegondDelitalaRascle} the authors consider
\[
p(\rho) \doteq \left(\frac{1}{\rho} - \frac{1}{\rho_{\max}}\right)^{-\gamma},
\]
that satisfies~\eqref{H2}; and in~\cite{PanHan} the authors consider $p(\rho) \doteq  -\varepsilon/\rho$, that does not satisfy~\eqref{H2}.
\end{example}

\begin{figure}[ht]
\centering{
\begin{psfrags}
      \psfrag{a}[c,B]{$\rho$}
      \psfrag{b}[c,B]{$R_{\max}$}
      \psfrag{c}[c,B]{$~~R_{\rm f}''$}
      \psfrag{d}[c,B]{$R_{\rm c}~$}
      \psfrag{e}[c,B]{$R_{\rm f}'$}
      \psfrag{f}[c,B]{$R_{\rm f}'' \, V_{\rm f}\quad$}
      \psfrag{g}[c,B]{$\rho\,v$}
      \psfrag{h}[c,B]{$V_{\max}$}
      \psfrag{i}[c,B]{$V_{\rm f}$}
      \psfrag{l}[c,B]{$V_{\rm c}$}
      \psfrag{m}[c,c]{$\Omega_{\rm c}$}
      \psfrag{n}[c,B]{$\Omega_{\rm f}''$}
      \psfrag{o}[c,B]{$\Omega_{\rm f}'$}
\includegraphics[width=.3\textwidth]{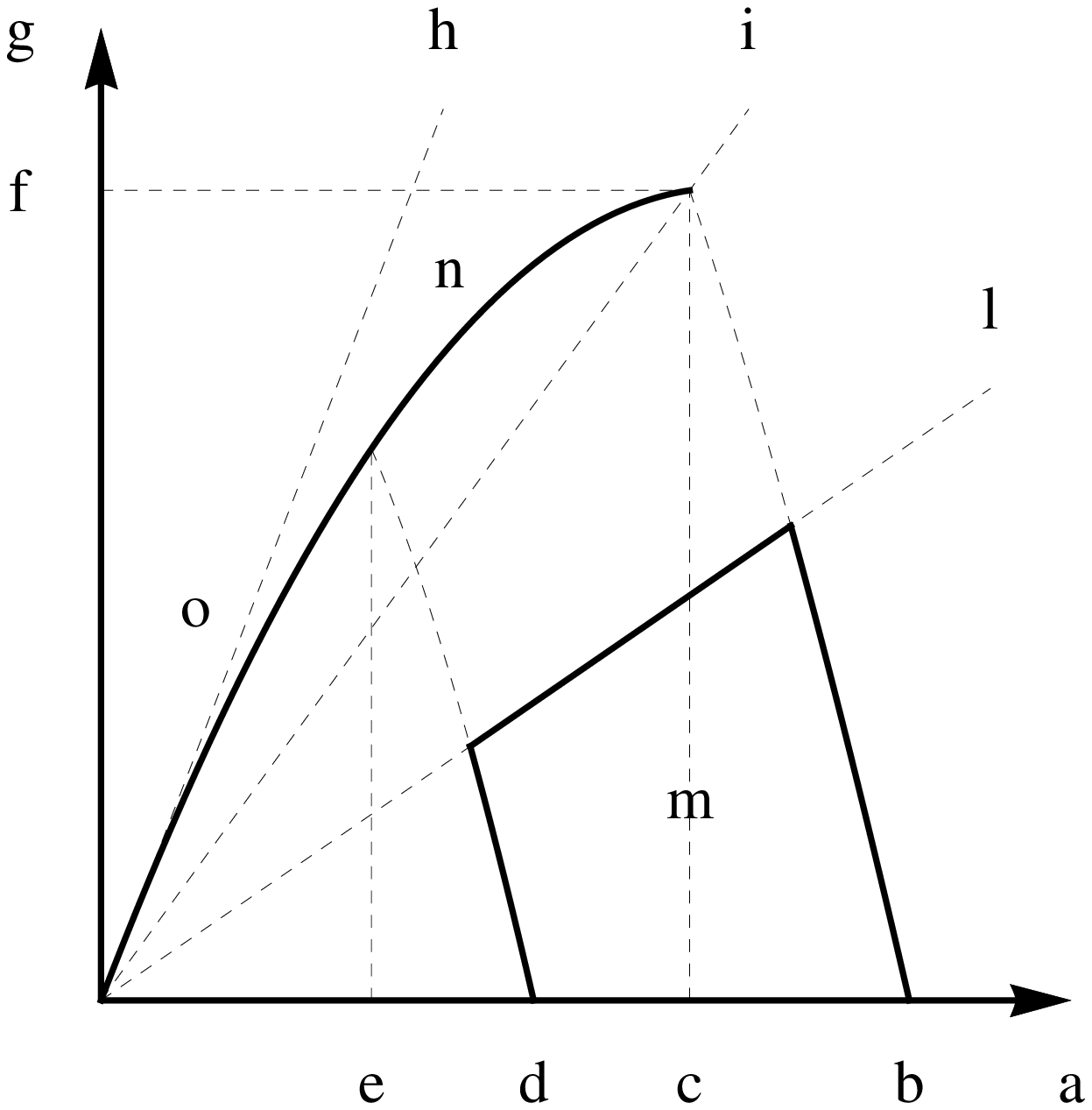}\qquad\qquad
      \psfrag{b}[c,B]{$~~~R_{\max}$}
      \psfrag{c}[c,B]{$~~R_{\rm f}''$}
      \psfrag{d}[c,B]{$R_{\rm c}~$}
      \psfrag{m}[c,c]{$~\Omega_{\rm c}$}
      \psfrag{n}[c,c]{$\Omega_{\rm f}''$}
\includegraphics[width=.3\textwidth]{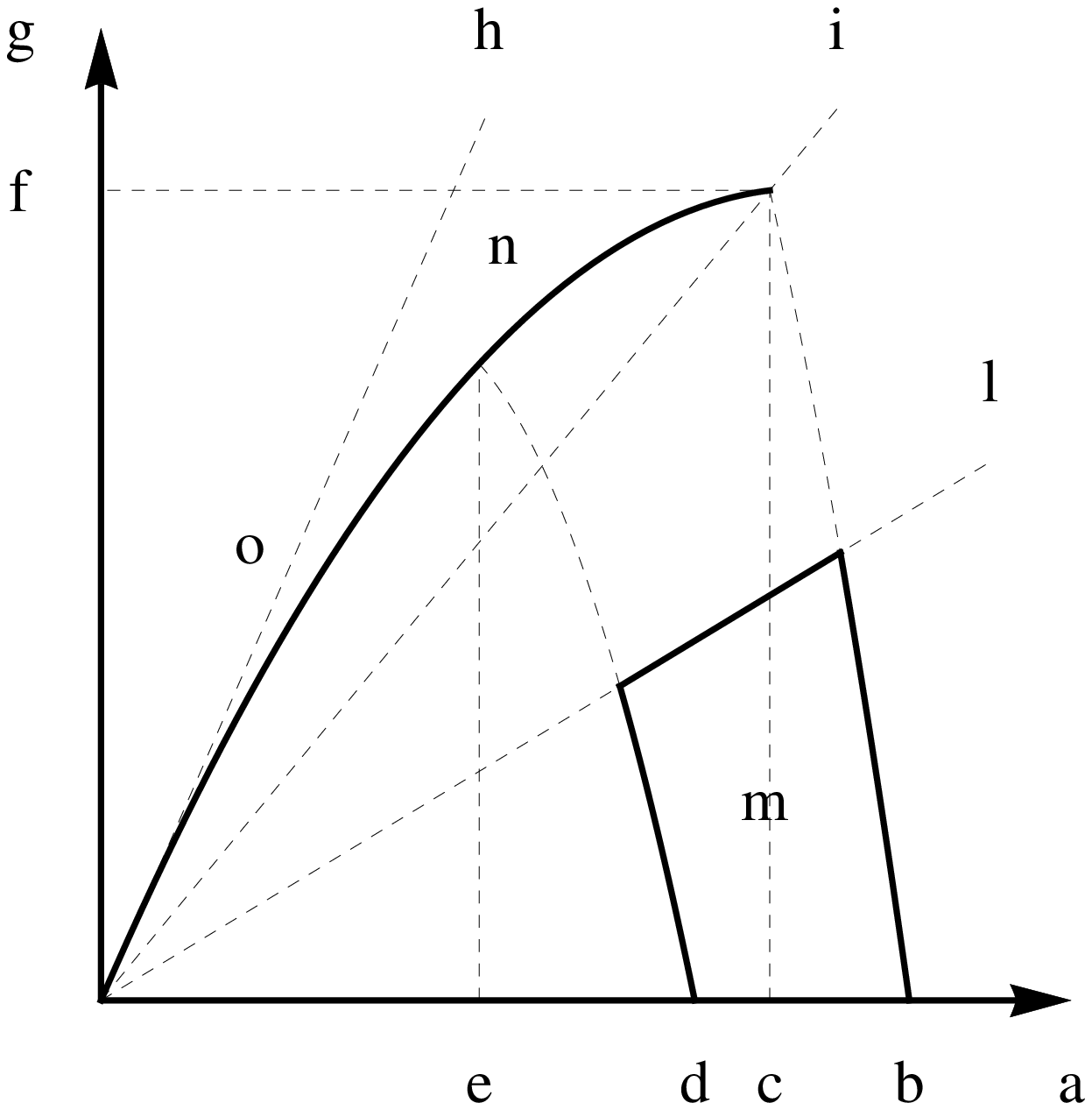}\\[5pt]
      \psfrag{n}[c,c]{$~~\Omega_{\rm f}''$}
      \psfrag{o}[c,c]{$~~\Omega_{\rm f}'$}
      \psfrag{v}[c,B]{$v$}
      \psfrag{w}[c,B]{$w$}
      \psfrag{p}[c,B]{$W_{\rm c}$}
      \psfrag{q}[c,c]{$W_{\min}~~$}
      \psfrag{r}[c,c]{$W_{\max}~~$}
      \psfrag{i}[c,c]{$V_{\rm f}$}
      \psfrag{l}[c,c]{$V_{\rm c}$}
\includegraphics[width=.3\textwidth]{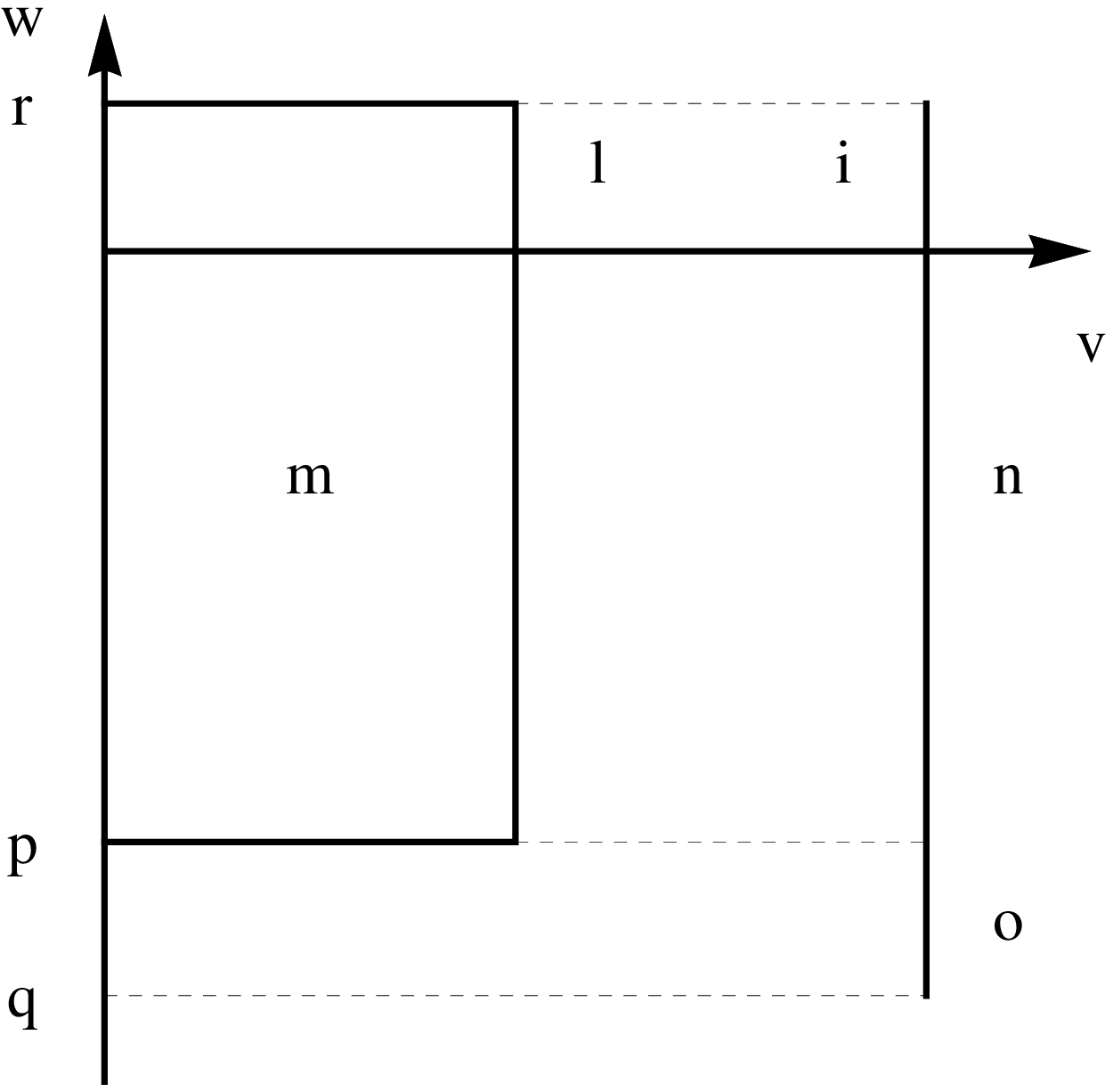}\qquad\qquad
\includegraphics[width=.3\textwidth]{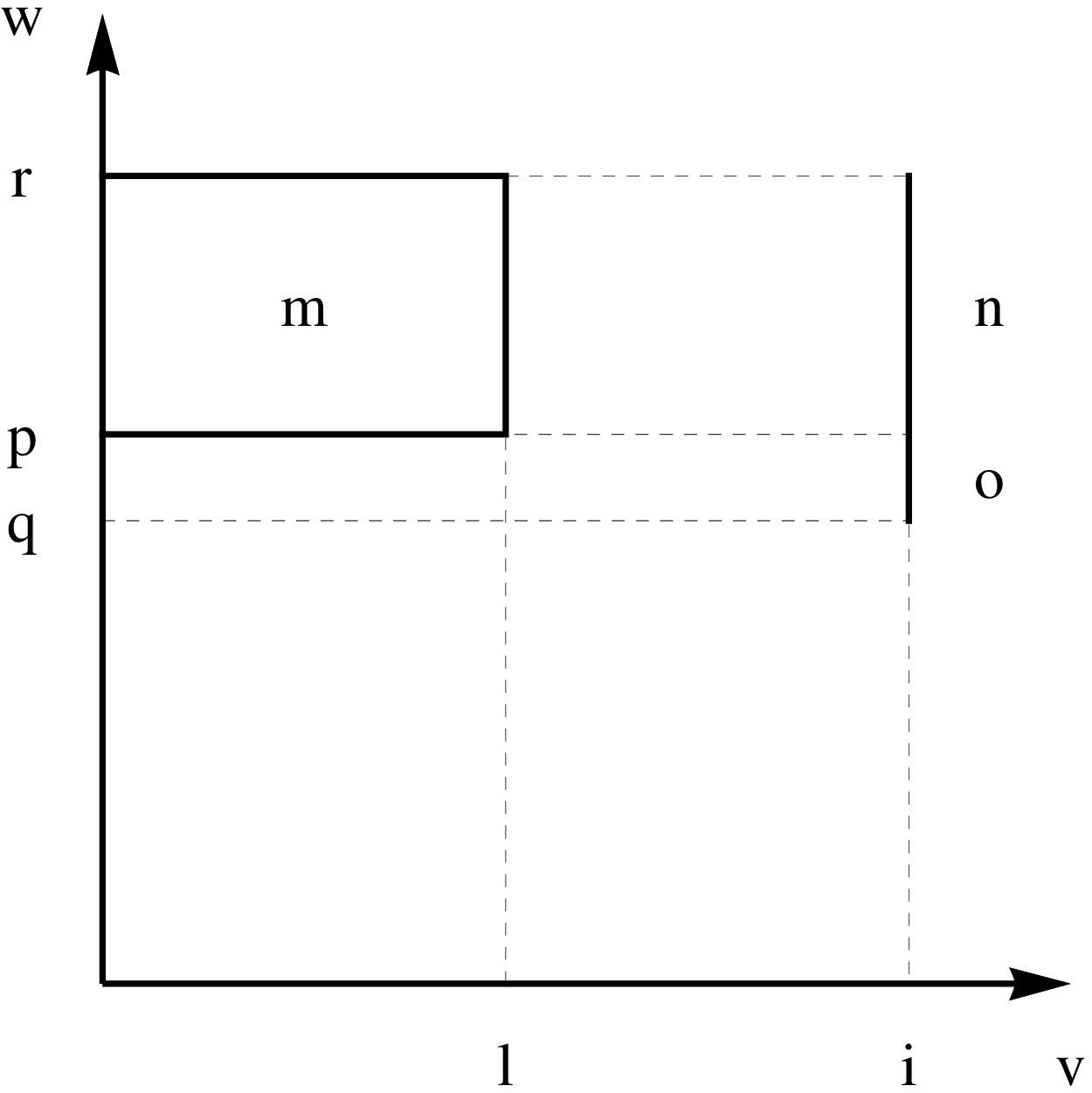}
\end{psfrags}}
\caption{The domains $\Omega_{\rm f}'$, $\Omega_{\rm f}''$ and $\Omega_{\rm c}$ corresponding to $v_{\rm f}$ and $p$ respectively given by~\eqref{eq:ex1v} and~\eqref{eq:ex1p} with $\gamma=0$ on the left, and $\gamma>0$ on the right.}
\label{fig:domains}
\end{figure}

Introduce the following notation, see \figurename~\ref{fig:domains},
\begin{align*}
&V_{\max} \doteq v_{\rm f}(0),&
&V_{\rm f} \doteq v_{\rm f}(R_{\rm f}''),
\\
&W_{\max} \doteq p(R_{\rm f}'') + V_{\rm f},&
&W_{\rm c} \doteq p(R_{\rm f}') + v_{\rm f}(R_{\rm f}'),&
&W_{\min} \doteq W_{\rm c}+v_{\rm f}(R_{\rm f}')-V_{\max},\\
&R_{\max} \doteq p^{-1}(W_{\max}),&
&R_{\rm c} \doteq p^{-1}(W_{\rm c}).
\end{align*}
By definition we have $R_{\max}>R_{\rm f}''>0$, $R_{\rm c}>R_{\rm f}'>0$, $W_{\max}>W_{\rm c}>W_{\min}$ and by~\eqref{H2} we have that $p^{-1} \colon [W_{\rm c} - v_{\rm f}(R_{\rm f}'),W_{\max}] \to [R_{\rm f}',R_{\max}]$ is increasing.

\begin{rem}
We can relax the assumption~\eqref{H2} on $p$ by requiring it on $[R_{\rm f}',R_{\max}]$.
\end{rem}

For notational simplicity, we let
\[
u \doteq (\rho,v) .
\]
Fix $V_{\rm c} \in \left]0, V_{\rm f}\right[$ and let the domains of free phases and congested phases be respectively
\begin{align*}
&\Omega_{\rm f} \doteq \left\{u \in [0,R_{\rm f}''] \times [V_{\rm f},V_{\max}] \colon v = v_{\rm f}(\rho)\right\},&
&\Omega_{\rm c} \doteq \left\{\vphantom{R_{\rm f}''} u \in [0,R_{\max}] \times [0,V_{\rm c}] \colon  W_{\rm c} \leq v+p(\rho) \leq W_{\max}\right\}.
\end{align*}
Observe that $\Omega_{\rm f}$ and $\Omega_{\rm c}$ are invariant domains for respectively LWR and ARZ.
For later use, introduce also the following subsets of $\Omega_{\rm f}$
\begin{align*}
&\Omega_{\rm f}' \doteq \left\{u \in \Omega_{\rm f} \colon \rho \in [0,R_{\rm f}'[ \right\},
&\Omega_{\rm f}'' \doteq \left\{u \in \Omega_{\rm f} \colon \rho \in [R_{\rm f}',R_{\rm f}''] \right\},
\end{align*}
and denote by $\Omega$ the domain of free and congested phases, namely
\[
\Omega \doteq \Omega_{\rm f} \cup \Omega_{\rm c}.
\]

\begin{lemma}[Definition of $\rho_{\rm f}$]
For any $w \in \left[ W_{\rm c}, W_{\max}\right]$ the graphs of the maps
\begin{align*}
&[0, R_{\rm f}''] \ni \rho \mapsto \rho \, v_{\rm f}(\rho)&
\text{and}&
&[R_{\rm f}', R_{\max}] \ni \rho \mapsto \rho \left[w-p(\rho)\right]&
\end{align*}
intersect in $\rho \doteq \rho_{\rm f}(w) > 0$.
Moreover the second map is strictly decreasing in $\rho_{\rm f}(w)$ and the capacity drop holds true.
\end{lemma}

\begin{proof}
We first observe that $R_{\rm f}' \, [w-p(R_{\rm f}')] \ge R_{\rm f}' \, v_{\rm f}(R_{\rm f}')$ and $R_{\rm f}'' \, [w-p(R_{\rm f}'')] \le R_{\rm f}'' \, v_{\rm f}(R_{\rm f}'')$, because $p(R_{\rm f}') + v_{\rm f}(R_{\rm f}') = W_{\rm c} \le w \le W_{\max} = p(R_{\rm f}'') + v_{\rm f}(R_{\rm f}'')$.
Moreover, by~\eqref{H1} the map $\rho \mapsto \rho \, v_{\rm f}(\rho)$ is strictly increasing in $[0,R_{\rm f}'']$, and by~\eqref{H2} the map $\rho \mapsto \rho \left[w-p(\rho)\right]$ is strictly concave in $[R_{\rm f}',+\infty[$.
Therefore there exists a unique $\rho = \rho(w)$ in $[R_{\rm f}' , R_{\rm f}'']$ such that $\rho \, [w-p(\rho)] = \rho \, v_{\rm f}(\rho)$, namely $w = p(\rho) + v_{\rm f}(\rho)$, and by~\eqref{H3} it satisfies $w - p(\rho) - \rho \, p'(\rho)
= v_{\rm f}(\rho)- \rho \, p'(\rho) < 0$.
\end{proof}
\noindent
Let us underline that by definition $R_{\rm f}' = \rho_{\rm f}(W_{\rm c})$ and $R_{\rm f}'' = \rho_{\rm f}(W_{\max})$.

\subsection{The two phase model}\label{sec:PT}

We are now in a position to write our two phase model
\begin{align}\label{eq:model}
&\begin{array}{l}
\text{\textbf{Free flow} (LWR)}\\[2pt]
\begin{cases}
(\rho,v) \in \Omega_{\rm f},\\
\rho_t+[\rho\,v]_x=0,\\
v=v_{\rm f}(\rho),
\end{cases}
\end{array}
&
\begin{array}{l}
\text{\textbf{Congested flow} (ARZ)}\\[2pt]
\begin{cases}
(\rho,v) \in \Omega_{\rm c},\\
\rho_t+[\rho\,v]_x=0,\\
\left[\rho\left(v+p(\rho)\right)\right]_t + \left[\rho\left(v+p(\rho)\right)v\right]_x=0.
\end{cases}
\end{array}
\end{align}
Above $\rho$ and $v$ denote respectively the density and the average speed of the vehicles, while $v_{\rm f}$ and $p$ are given functions satisfying~\eqref{H1}, \eqref{H2} and~\eqref{H3} and denote respectively the speed of the vehicles in a free flow and the ``pressure'' of the vehicles in a congested flow.
Recall that $p$ takes into account the drivers' reactions to the state of traffic in front of them.
Finally, $\Omega_{\rm f}$ and $\Omega_{\rm c}$ are respectively the domains of free and congested phases.
Observe that in $\Omega_{\rm f}$ the density $\rho$ is the unique independent variable, while in $\Omega_{\rm c}$ the independent variables are two, both the density $\rho$ and the velocity $v$.

The aim of this article is to prove Theorem~\ref{thm:1} given at the end of this section, that states the global existence of solutions of Cauchy problems for~\eqref{eq:model} with $\BV$-initial data
\begin{align}\label{eq:initial}
&\rho(0,x) = \bar{\rho}(x),&
&v(0,x) = \bar{v}(x).
\end{align}

Let $u$ be a solution of~\eqref{eq:model}, \eqref{eq:initial} in a sense that we specify in the next subsection.
Any discontinuity performed by $u$ separating a state in $\Omega_{\rm f}$ from a state in $\Omega_{\rm c}$ is called a \emph{phase transition}.
In this paper the number of phase transitions performed by the initial datum and by the solution is not fixed a priori but are imposed to be finite.

We recall the main features of the two phase model~\eqref{eq:model}.
In the free phase the characteristic speed is $\lambda_{\rm f}(\rho) \doteq v_{\rm f}(\rho) + \rho \, v_{\rm f}'(\rho)$, while the informations on the system modelling the congested phase are collected in the following table (see~\citep{ARZ1} for more details):
\begin{align*}
&r_1(u) \doteq (1,-p'(\rho)),&
&r_2(u) \doteq (1,0),\\
&\lambda_1(u) \doteq v-\rho\,p'(\rho),&
&\lambda_2(u) \doteq v,\\
&\nabla\lambda_1\cdot r_1(u) = -2p'(\rho)-\rho\,p''(\rho),&
&\nabla\lambda_2\cdot r_2(u) = 0,\\
&\mathcal{L}_1(\rho;u_0) \doteq v_0+p(\rho_0)-p(\rho),&
&\mathcal{L}_2(\rho;u_0) \doteq v_0,\\
&w_1(u) \doteq v,&
&w_2(u) \doteq v+p(\rho),
\end{align*}
where $r_i$ is the $i$-th right eigenvector, $\lambda_i$ is the corresponding eigenvalue, $\{ u \in \Omega_{\rm c} \colon v = \mathcal{L}_i(\rho;u_0)\}$ is the $i$-Lax curve passing through $u_0 \in \Omega_{\rm c}$ and $w_i$ is the $i$-th Riemann invariant.
In particular, the assumptions~\eqref{H1} and~\eqref{H2} ensure that the characteristic speeds are bounded by the velocity, $\lambda_{\rm f}(u) \le v$, $\lambda_1(u) \leq \lambda_2(u) = v$, and that $\lambda_1$ is genuinely non-linear, $\nabla\lambda_1 \cdot r_1(u) \neq0$.

By using the Riemann invariant coordinates, the domain of congested phases writes $\Omega_{\rm c} = [0,V_{\rm c}] \times [W_{\rm c},W_{\max}]$, see \figurename~\ref{fig:domains}.
We extend the Riemann invariants to the whole $\Omega$ and define for any $u \in \Omega_{\rm f}$
\begin{align*}
&w_1(u) \doteq V_{\rm f},&
&w_2(u) \doteq \begin{cases}
v_{\rm f}(\rho)+p(\rho)&\text{if }u \in \Omega_{\rm f}'',\\
W_{\rm c}+v_{\rm f}(R_{\rm f}')-v_{\rm f}(\rho)&\text{if }u \in \Omega_{\rm f}',
\end{cases}
\end{align*}
so that the domain of free phases writes in the extended Riemann invariant coordinates $\Omega_{\rm f} = \{V_{\rm f}\} \times [W_{\min},W_{\max}]$, see \figurename~\ref{fig:domains}.
In conclusion, we let
\begin{align}\label{eq:change}
&w_1(u) \doteq 
\begin{cases}
v&\text{if }u \in \Omega_{\rm c},\\
V_{\rm f}&\text{if }u \in \Omega_{\rm f},
\end{cases}
&w_2(u) \doteq \begin{cases}
v+p(\rho)&\text{if }u\in \Omega_{\rm c} \cup \Omega_{\rm f}'',\\
W_{\rm c}+v_{\rm f}(R_{\rm f}')-v_{\rm f}(\rho)&\text{if }u \in \Omega_{\rm f}'.
\end{cases}
\end{align}
In $\Omega$ we will consider the norm corresponding to the above extended Riemannn invariant coordinates
\[
\norma{u} \doteq \modulo{w_1(u)} + \modulo{w_2(u)},
\]
as well as the corresponding distance.

\subsection{Weak and entropy solutions}\label{sec:we}

In this section we introduce the definitions of weak and entropy solutions to~\eqref{eq:model}, \eqref{eq:initial} and prove an existence result in Theorem~\ref{thm:1}.
Let us first recall the corresponding definitions related to LWR and ARZ.

\begin{definition}[Solutions of LWR~\cite{kruzkov}]\label{def:solLWR}
Fix an initial datum $\bar{u}$ in $\L\infty(\R;\Omega_{\rm f})$.
Let $u$ be a function in $\L\infty(\R_+\times\R;\Omega_{\rm f}) \cap \C0(\R_+;\Lloc1(\R;\Omega_{\rm f}))$.
\begin{enumerate}[label={(\arabic*)},leftmargin=*]
\item
We say that $u$ is a weak solution to LWR~\eqref{eq:model}-left, \eqref{eq:initial} if $u(0,x) = \bar{u}(x)$ for a.e.~$x\in\R$ and for any test function $\varphi$ in $\Cc\infty(]0,+\infty[\times\R;\R)$
\begin{equation}\label{eq:weakLWR}
\iint_{\R_+\times\R} \rho \left[\vphantom{\sum} \varphi_t+v_{\rm f}(\rho) \, \varphi_x\right] \d x \, \d t = 0.
\end{equation}
\item
We say that $u$ is an entropy solution to LWR~\eqref{eq:model}-left, \eqref{eq:initial} if $u(0,x) = \bar{u}(x)$ for a.e.~$x\in\R$ and for any non-negative test function $\varphi$ in $\Cc\infty(]0,+\infty[\times\R;\R)$ and for any constant $h$ in $[0,R_{\rm f}'']$
\begin{equation}\label{eq:entropyLWR}
\iint_{\R_+\times\R} \left[\mathcal{E}_{\rm LWR}^h(u) \, \varphi_t + \mathcal{Q}_{\rm LWR}^h(u) \, \varphi_x\right] \d x \, \d t \ge 0,
\end{equation}
where
\begin{align*}
&\mathcal{E}_{\rm LWR}^h(u) \doteq \modulo{\rho-h},&
&\mathcal{Q}_{\rm LWR}^h(u) \doteq \sign(\rho-h) \left(\vphantom{\sum} \rho \, v_{\rm f}(\rho) - h \, v_{\rm f}(h)\right).
\end{align*}
\end{enumerate}
\end{definition}

\begin{definition}[Solutions of ARZ~\cite{BCM2order, BCJMU-order2}]\label{def:solARZ}
Fix an initial datum $\bar{u}$ in $\L\infty(\R;\Omega_{\rm c})$.
Let $u$ be a function in $\L\infty(\R_+\times\R;\Omega_{\rm c}) \cap \C0(\R_+;\Lloc1(\R;\Omega_{\rm c}))$.
\begin{enumerate}[label={(\arabic*)},leftmargin=*]
\item
We say that $u$ is a weak solution to ARZ~\eqref{eq:model}-right, \eqref{eq:initial} if $u(0,x) = \bar{u}(x)$ for a.e.~$x\in\R$ and for any test function $\varphi$ in $\Cc\infty(]0,+\infty[\times\R;\R)$
\begin{align}\label{eq:weakARZ}
&\iint_{\R_+\times\R} \rho \left[\vphantom{\sum} \varphi_t+v \, \varphi_x\right] \d x \, \d t = 0,&
&\iint_{\R_+\times\R} \rho \, w_2(u) \left[\vphantom{\sum} \varphi_t+v \, \varphi_x\right] \d x \, \d t = 0.
\end{align}
\item
We say that $u$ is an entropy solution to ARZ~\eqref{eq:model}-right, \eqref{eq:initial} if it is a weak solution and for any non-negative test function $\varphi$ in $\Cc\infty(]0,+\infty[\times\R;\R)$ and for any constant $k$ in $[0,V_{\rm c}]$
\begin{equation}\label{eq:entropyARZ}
\iint_{\R_+\times\R} \left[\mathcal{E}_{\rm ARZ}^k(u) \, \varphi_t + \mathcal{Q}_{\rm ARZ}^k(u) \, \varphi_x\right] \d x \, \d t \ge 0,
\end{equation}
where
\begin{align*}
&\mathcal{E}_{\rm ARZ}^k(u) \doteq
\begin{cases}
0&\text{if }v\le k,
\\
1-\dfrac{\rho}{p^{-1}(w_2(u)-k)}&\text{if }v>k,
\end{cases}&
&\mathcal{Q}_{\rm ARZ}^k(u) \doteq
\begin{cases}
0&\text{if }v\le k,
\\
k-\dfrac{\rho \, v}{p^{-1}(w_2(u)-k)}&\text{if }v>k.
\end{cases}
\end{align*}
\end{enumerate}
\end{definition}

The definitions of weak and entropy solutions to the two phase model~\eqref{eq:model} can not be obtained by just imposing the Definition~\ref{def:solLWR} in the free phase and the Definition~\ref{def:solARZ} in the congested phase.
In fact, the condition~\eqref{eq:weakLWR} states the conservation of the number of vehicles, while~\eqref{eq:weakARZ} states the conservation of both the number of vehicles and of the generalized momentum $\rho \, w_2(u)$.
For this reason, in a phase transition from $\Omega_{\rm c}$ to $\Omega_{\rm f}$ we ``loose'' the conservation of the generalized momentum, while in a phase transition from $\Omega_{\rm f}$ to $\Omega_{\rm c}$ we ``introduce'' a generalized momentum, which is then conserved as long as the traffic remains in the congested phase.

Moreover, it is well known that ARZ can be interpreted as a generalization of LWR to the case of a multi-population traffic, see for instance~\cite{BCM2order, BCJMU-order2, MarcoSimoneMax-ARZ1, FanHertySeibold}.
More specifically, each vehicle is characterized by a constant value of $w_2$, in other words, for any trajectory of a vehicle $t \mapsto x(t)$, the map $t \mapsto w_2(u(t,x(t)))$ is constant.
For this reason $w_2$ is a Lagrangian marker.
Furthermore, the vehicle initially in $\bar{x}$ has at any time Lagrangian marker $\bar{w} \doteq w_2(\bar{u}(\bar{x})))$ and, consequently, it has speed law $\rho \mapsto [\bar{w}-p(\rho)]$, length $[1/p^{-1}(\bar{w})]$ and maximal velocity $\bar{w}$.
For this reason, in a phase transition from $\Omega_{\rm c}$ to $\Omega_{\rm f}$ we ``loose'' the characterization of the vehicles, while in a phase transition from $\Omega_{\rm f}$ to $\Omega_{\rm c}$ we ``introduce'' a characterization of the vehicles.

We first state general definitions of weak and entropy solutions to the Cauchy problem for the two phase model~\eqref{eq:model}, \eqref{eq:initial} based on that ones given in~\cite{ColomboMarcellini2015}, where the LWR model is coupled with a microscopic follow-the leader model, see also~\cite[Definition~1.2 on page~243]{LeflochBook}.
Introduce the following notations
\begin{align}\label{eq:sigma}
\sigma(u_-,u_+) \doteq \dfrac{\rho_+ \, v_+ - \rho_- \, v_-}{\rho_+ - \rho_-},
\end{align}
$\mathcal{G}_{\rm w} \doteq \mathcal{G}_1 \cup \mathcal{G}_2 \cup \mathcal{G}_1^T \cup \mathcal{G}_2^T$ and $\mathcal{G}_{\rm e} \doteq \mathcal{G}_1 \cup \mathcal{G}_2 \cup \mathcal{G}_3$, see~\cite{scontrainte}, where
\begin{align*}
&\mathcal{G}_1 \doteq \left\{ (u_-,u_+) \in \Omega_{\rm f}' \times \Omega_{\rm c} \colon \left[w_2(u_+) - W_{\rm c}\right]\rho_- = 0\right\},&
&\mathcal{G}_2 \doteq \left\{ (u_-,u_+) \in \Omega_{\rm f}'' \times \Omega_{\rm c} \colon w_2(u_-) = w_2(u_+)\right\},\\
&\mathcal{G}_3 \doteq \left\{ (u_-,u_+) \in \Omega_{\rm c} \times \Omega_{\rm f}'' \colon w_2(u_-) = w_2(u_+) \text{ and }v_- = V_{\rm c}\right\},&
&\mathcal{G}_i^T \doteq \left\{\vphantom{\Omega_{\rm f}'} (u_-,u_+) \in \Omega \times \Omega \colon (u_+,u_-) \in \mathcal{G}_i \right\}.
\end{align*}

\begin{definition}[General concept of weak solution of~\eqref{eq:model}, \eqref{eq:initial}]\label{def:weakPT}
Fix an initial datum $\bar{u}$ in $\BV(\R;\Omega)$.
A function $u$ in $\L\infty(]0,+\infty[;\BV(\R;\Omega)) \cap \C0(\R_+;\Lloc1(\R;\Omega))$ is a weak solution to~\eqref{eq:model}, \eqref{eq:initial} if $u(0,x) = \bar{u}(x)$ for a.e.~$x\in\R$ and it satisfies the following conditions:
\begin{enumerate}[label={(W.\arabic*)},leftmargin=*]
\item\label{weakPT1}
The equality~\eqref{eq:weakLWR} holds for any test function $\varphi$ in $\Cc\infty(]0,+\infty[\times\R;\R)$ such that $u(t,x) \in \Omega_{\rm f}$ for a.e.~$(t,x)$ in the support of $\varphi$.
\item\label{weakPT2}
The equality~\eqref{eq:weakARZ} holds for any test function $\varphi$ in $\Cc\infty(]0,+\infty[\times\R;\R)$ such that $u(t,x) \in \Omega_{\rm c}$ for a.e.~$(t,x)$ in the support of $\varphi$.
\item\label{weakPT3}
There exist finitely many Lipschitz-continuous curves $x=x_i(t)$ across which $u$ may perform a phase transition.
Moreover, if $x \mapsto u(t,x)$ performs across $x=x_i(t)$, $t>0$, a phase transition from $u_-(t) \doteq \lim_{x\to x_i(t)^-} u(t,x)$ to $u_+(t) \doteq \lim_{x\to x_i(t)^+} u(t,x)$, then the speed of propagation of the phase transition $\dot{x}_i(t)$ equals $\sigma(u_-(t),u_+(t))$ and $(u_-(t),u_+(t))$ belongs to $\mathcal{G}_{\rm w}$.
\end{enumerate}
\end{definition}
It is worth to underline that the number of phase transitions performed by a weak solution $u$ of~\eqref{eq:model}, \eqref{eq:initial} is bounded from above by $\tv(u) \, [V_{\rm c} - V_{\rm f}]^{-1}$.

\begin{definition}[General concept of entropy solution of~\eqref{eq:model}, \eqref{eq:initial}]\label{def:entropyPT}
Fix an initial datum $\bar{u}$ in $\BV(\R;\Omega)$.
A function $u$ in $\L\infty(]0,+\infty[;\BV(\R;\Omega)) \cap \C0(\R_+;\Lloc1(\R;\Omega))$ is an entropy solution to~\eqref{eq:model}, \eqref{eq:initial} if it is a weak solution and satisfies the following conditions:
\begin{enumerate}[label={(E.\arabic*)},leftmargin=*]
\item\label{entropyPT1}
The estimate~\eqref{eq:entropyLWR} holds for any test function $\varphi$ in $\Cc\infty(]0,+\infty[\times\R;\R_+)$ such that $u(t,x) \in \Omega_{\rm f}$ for a.e.~$(t,x)$ in the support of $\varphi$.
\item\label{entropyPT2}
The estimate~\eqref{eq:entropyARZ} holds for any test function $\varphi$ in $\Cc\infty(]0,+\infty[\times\R;\R_+)$ such that $u(t,x) \in \Omega_{\rm c}$ for a.e.~$(t,x)$ in the support of $\varphi$.
\item\label{entropyPT3}
If $u$ performs a phase transition from $u_-$ to $u_+$, then $(u_-,u_+)$ belongs to $\mathcal{G}_{\rm e}$.
\end{enumerate}
\end{definition}

The criterion for the phase transitions given in~\ref{entropyPT3} of Definition~\ref{def:entropyPT} is introduced to select the admissible phase transitions of entropy solutions with bounded variation, where $u_-$ and $u_+$ denote the traces of the $\BV$ entropy solution $u$ on a Lipschitz curve of jump (see~\cite{Volpert} for precise formulation of the regularity of $\BV$ functions).
In the following lemma we show that this condition is satisfied by the approximate solutions constructed in  Section~\ref{sec:aas}.

\begin{lemma}\label{lem:DMB}
Fix $\bar{u}^n$ in $\PC\left(\R;\Omega^n\right)$ and let $u^n \in \C0\left(\R_+;\PC\left(\R;\Omega^n \right)\right)$ be the approximate solution constructed in Section~\ref{sec:aas}.
Then $x\mapsto u^n(t,x)$, $t>0$, satisfies the property~\ref{entropyPT3} of Definition~\ref{def:entropyPT}.
Moreover, the number of phase transitions performed by $x\mapsto u^n(t,x)$, $t\ge0$, does not increase with time and it strictly decreases if and only if two phase transitions interact.
Finally, the number of phase transitions can decrease only by an even number.
\end{lemma}
The proof of the above lemma is postponed to Section~\ref{sec:DMB}.

We conclude the section by giving the main result of this paper in the following

\begin{thm}\label{thm:1}
For any initial datum $\bar{u}$ in $\BV(\R;\Omega)$, the approximate solution for the Cauchy problem~\eqref{eq:model}, \eqref{eq:initial} with initial datum $\bar{u}$ constructed in Section~\ref{sec:aas} converges (up to a subsequence) in $\Lloc1(\R_+\times\R;\Omega)$ to a function $u$ in $\C0(\R_+;\BV(\R;\Omega))$.
Moreover, for any $t,s\ge0$ we have that
\begin{align*}
&\tv(u(t)) \le \tv(\bar{u}),&
&\norma{u(t)}_{\L\infty(\R;\Omega)} \le C,&
&\norma{u(t)-u(s)}_{\L1(\R;\Omega)} \le L \, \modulo{t-s},
\end{align*}
where 
\begin{align*}
&L \doteq \tv(\bar{u}) \, \max\{V_{\max}, p^{-1}(W_{\max}) \, p'(p^{-1}(W_{\max}))\},&
&C \doteq \max\{\modulo{W_{\max}},\modulo{W_{\min}}\} + V_{\max}.
\end{align*}
Finally, if $V_{\max} = V_{\rm f}$, then $u$ is an entropy solution of the Cauchy problem~\eqref{eq:model}, \eqref{eq:initial} in the sense of Definition~\ref{def:entropyPT}.
\end{thm}
\noindent
The proof is based on the wave-front tracking algorithm described in Section~\ref{sec:wft} and is deferred to Section~\ref{sec:1}.
It is worth to mention here that it is not easy to prove that the constructed solutions satisfy the conditions listed in Definition~\ref{def:weakPT} or Definition~\ref{def:entropyPT}.
For this reason, in Lemma~\ref{lemma:AnimalsAsLeaders} and Lemma~\ref{lemma:entropyaltPT} we rather prove that the constructed solutions satisfy integral conditions, respectively~\eqref{eq:weakPT} and~\eqref{eq:entropyPT}, from which we easily deduce the conditions given respectively in Definition~\ref{def:weakPT} and Definition~\ref{def:entropyPT}.

\section{Example}\label{sec:ex}

In this section we apply the phase transition model~\eqref{eq:model} to simulate the traffic across a traffic light.
Typically, one is interested in computing the minimal time necessary to let all these vehicles pass through the traffic light.
For this reason we will construct the solution only in the upstream of the traffic light.
The simulation is presented in \figurename~\ref{fig:02} and is obtained by explicit analysis of the wave-front interactions, with computer-assisted computation of interaction times and front slopes.
While the overall picture of the corresponding solution is rather stable, a detailed analytical study necessarily needs to consider many slightly different cases.
Below, we restrict the construction of the solution to the most representative situation.
\begin{figure}
      \centering\begin{psfrags}\tiny
      \psfrag{r}[r,b]{$\rho$}
      \psfrag{q}[c,t]{$\rho\,v$}
        \includegraphics[width=.32\textwidth]{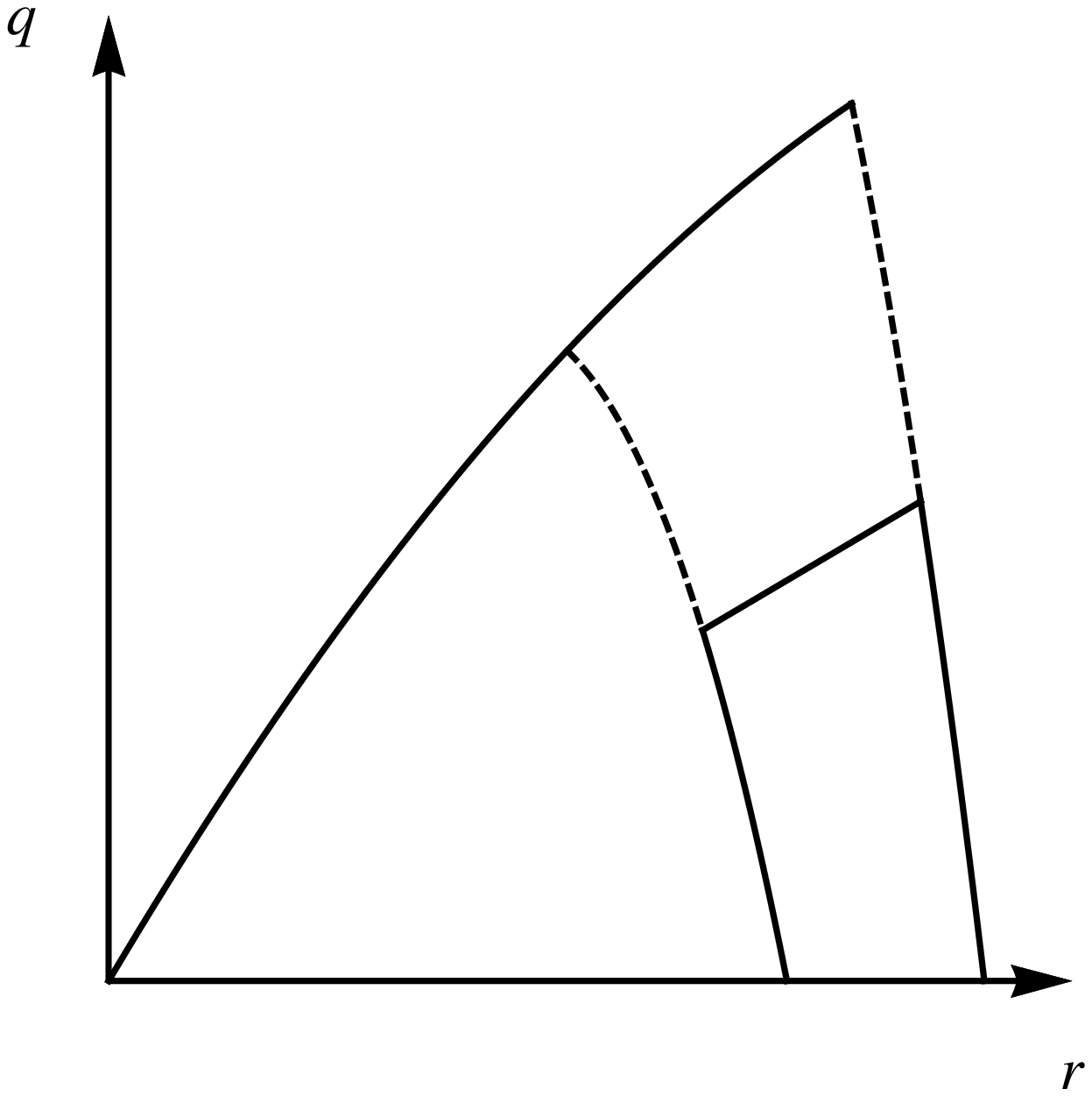}~
      \psfrag{x}[c,b]{$x$}
      \psfrag{t}[c,b]{$t$}
      \psfrag{a}[c,c]{$\mathcal{PT}_1$}
      \psfrag{b}[c,c]{$\mathcal{PT}_1'$}
      \psfrag{c}[c,c]{$\mathcal{PT}_1''$}
      \psfrag{d}[c,c]{$\mathcal{S}_1$}
      \psfrag{e}[c,c]{$\mathcal{C}_2$}
      \psfrag{f}[c,c]{$\mathcal{C}_2'$}
      \psfrag{g}[c,c]{$\mathcal{C}_2''$}
      \psfrag{h}[c,c]{$\mathcal{S}_2$}
      \psfrag{i}[c,c]{$\mathcal{R}_\ell$}
      \psfrag{l}[c,c]{$\mathcal{PT}_2$}
      \psfrag{m}[c,c]{$\mathcal{R}_\ell'$}
      \psfrag{n}[c,c]{$\mathcal{PT}_2'$}
        \includegraphics[width=.32\textwidth]{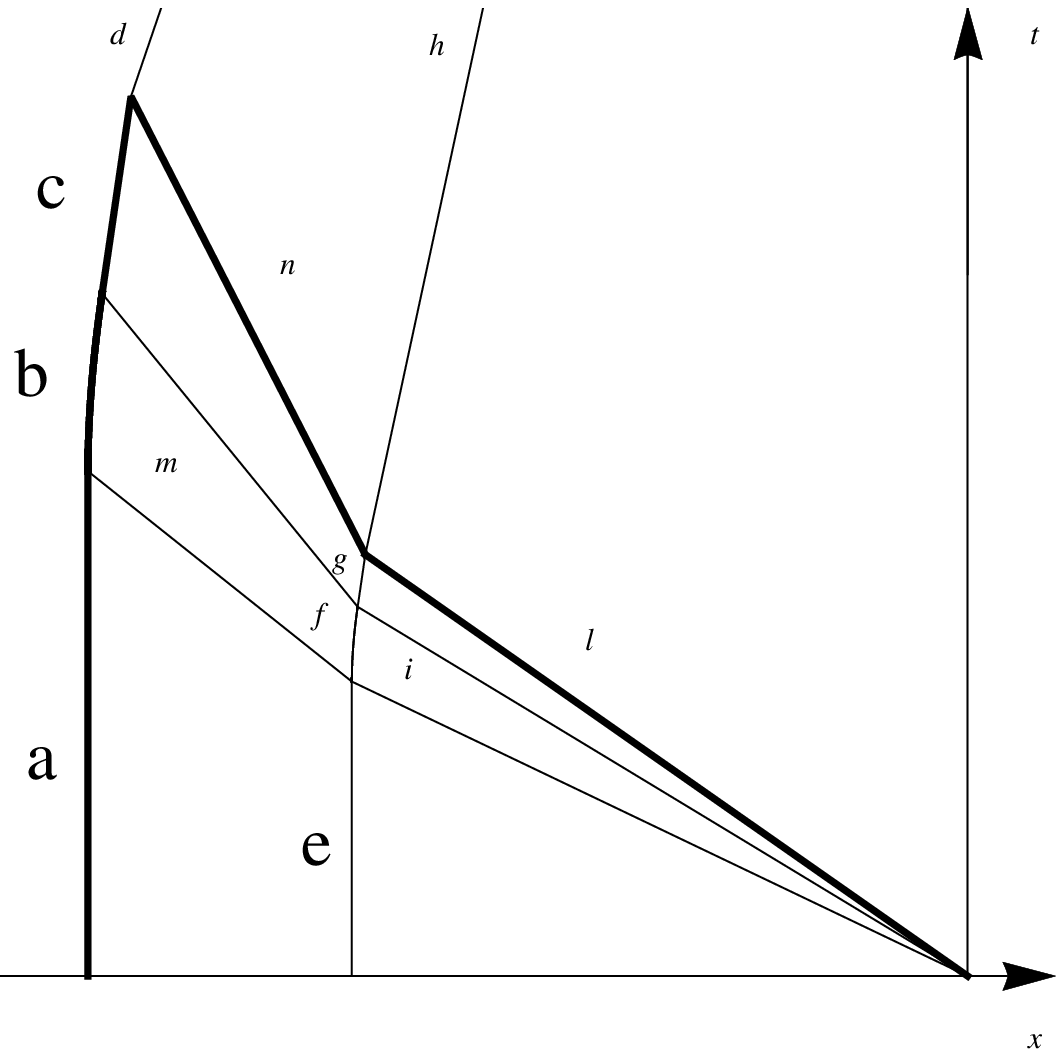}~
      \psfrag{x}[r,b]{$x$}
      \psfrag{t}[c,t]{$t$}
      \psfrag{1}[c,b]{$x_1$}
      \psfrag{2}[c,b]{$x_2$}
      \psfrag{a}[r,t]{$a_1$}
      \psfrag{b}[r,b]{$b_1$}
      \psfrag{c}[c,b]{$c_{1_{\vphantom{\displaystyle{1^1}}}}$} 
      \psfrag{d}[r,t]{$a_2$}
      \psfrag{e}[r,b]{$b_2$}
      \psfrag{f}[c,b]{$c_{2_{\vphantom{\displaystyle{1^1}}}}$}     
      \psfrag{g}[r,b]{$d_{1_{\vphantom{{1^1}}}}$}     
      \psfrag{h}[l,t]{~~$d_2$}
        \includegraphics[width=.32\textwidth]{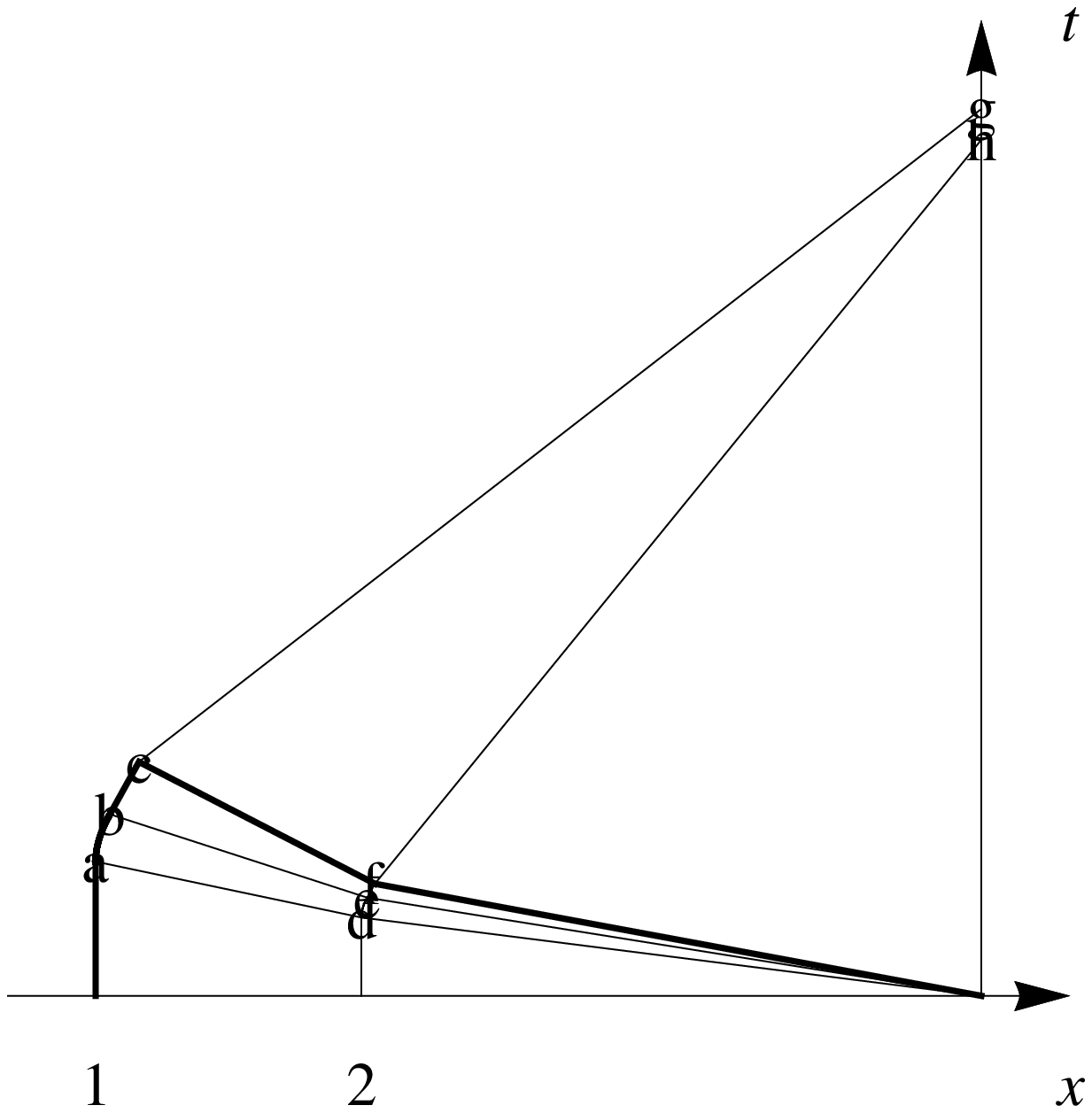}
      \end{psfrags}
      \caption{The solution constructed in Section~\ref{sec:ex} and corresponding to the numerical data~\eqref{eq:numericaldata}. The bold segments in the last two pictures correspond to phase transitions.}
\label{fig:02}
\end{figure}

More specifically, let
\begin{align*}
&v_{\rm f}(\rho) \doteq V_{\max}\left(1-\rho\right),&
&p(\rho) \doteq \rho^\gamma,~\gamma>0,
\end{align*}
and fix two constants $W_{\max}>W_{\rm c}>0$ such that the conditions~\eqref{H1}, \eqref{H2} and~\eqref{H3} are satisfied.
Consider two types of vehicles, the ``long vehicles'' characterized by the Lagrangian marker $W_{\rm c}$, and the ``short vehicles'' characterized by the Lagrangian marker $W_{\max}$.
Observe that the length of the short vehicles, $1/p^{-1}(W_{\max})$, is lower than that one of the long vehicles, $1/p^{-1}(W_{\rm c})$.

Place in $x=0$ a traffic light that turns from red to green at time $t=0$.
Assume that at time $t=0$ all the vehicles are stop in $\left[x_1,0\right[$.
More precisely, assume that the long vehicles are uniformly distributed in $\left[x_1,x_2\right[$ with density $R_{\rm c} = p^{-1}(W_{\rm c})$ and the short vehicles are uniformly distributed in $\left[x_2,0\right[$ with density $R_{\max} = p^{-1}(W_{\max})$.
The corresponding initial datum is then, see \figurename~\ref{fig:simuprophiles},
\begin{align*}
&\rho(0,x)=
\begin{cases}
R_{\rm c}&\text{if } x_1\le x<x_2,\\
R_{\max}&\text{if } x_2\le x<0,\\
0&\text{otherwise},
\end{cases}
&v(0,x)=
\begin{cases}
V_{\max}&\text{if } x<x_1,\\
0&\text{if } x_1\le x<0,\\
V_{\max}&\text{if } x\ge 0.
\end{cases}
\end{align*}

As a first step in the construction of the solution we have to consider the Riemann problems at $(t,x) \in \{(0,x_1),(0,x_2),(0,0)\}$.
We obtain that:
\begin{itemize}[leftmargin=*]
\item
The Riemann problem in $(0,x_1)$ is solved by a stationary phase transition $\mathcal{PT}_1$ from the vacuum state $(0,V_{\max})$ to $(R_{\rm c}, 0)$.
\item
The Riemann problem in $(0,x_2)$ is solved by a stationary contact discontinuity $\mathcal{C}_2$ from $(R_{\rm c}, 0)$ to $(R_{\max}, 0)$.
\item
The Riemann problem associated to $x=0$ is solved by a rarefaction $\mathcal{R}_\ell$ with support in the cone
\[
\left\{ (x,t) \in \R\times\R_+ \colon \lambda_1(R_{\max}, 0) \, t \le x \le \lambda_1(p^{-1}(W_{\max}-V_{\rm c}), V_{\rm c}) \, t \right\},
\]
and taking values $\mathcal{R}_\ell(x/t)$, where
\[
\begin{array}{@{}c@{\,}c@{\,}c@{\,}c@{\,}c@{}}
\mathcal{R}_\ell \colon&
\left[\lambda_1\left(R_{\max}, 0\right) , \lambda_1\left(p^{-1}(W_{\max}-V_{\rm c}), V_{\rm c}\right)\right]&
\to&
\Omega_{\rm c}
\\
&\xi&\mapsto&
\begin{pmatrix}
\rho_{\mathcal{R}_\ell}\left(\xi\right)
\\
v_{\mathcal{R}_\ell}\left(\xi\right)
\end{pmatrix}
&\doteq
\begin{pmatrix}
\mathfrak{R}_{\rm c}\left(W_{\max}-\xi\right),\\
W_{\max} - p\left(\rho_{\mathcal{R}_\ell}\left(\xi\right)\right)
\end{pmatrix},
\end{array}
\]
being $\mathfrak{R}_{\rm c}$ the inverse function of $\rho \mapsto p(\rho)+\rho\,p'(\rho)$, followed by a phase transition $\mathcal{PT}_2$ from $(p^{-1}(W_{\max}-V_{\rm c}),V_{\rm c})$ to $(R_{\rm f}'',V_{\rm f})$, followed by a rarefaction $\mathcal{R}_r$ with support in the cone
\[
\left\{ (x,t) \in \R\times\R_+ \colon \lambda_{\rm f}(R_{\rm f}'') \, t \le x \le \lambda_{\rm f}(0) \, t \right\},
\]
and taking values $\mathcal{R}_r(x/t)$, where
\[
\begin{array}{@{}c@{\,}c@{\,}c@{\,}c@{\,}c@{}}
\mathcal{R}_r \colon&
\left[\lambda_{\rm f}(R_{\rm f}'') , \lambda_{\rm f}(0)\right]&
\to&
\Omega_{\rm f}
\\
&\xi&\mapsto&
\begin{pmatrix}
\rho_{\mathcal{R}_r}\left(\xi\right)
\\
v_{\mathcal{R}_r}\left(\xi\right)
\end{pmatrix}
&\doteq
\begin{pmatrix}
\mathfrak{R}_{\rm f}\left(\xi\right)\\
v_{\rm f}\left(\rho_{\mathcal{R}_r}\left(\xi\right)\right)
\end{pmatrix},
\end{array}
\]
being $\mathfrak{R}_{\rm f}$ the inverse function of $\rho \mapsto v_{\rm f}(\rho)+\rho\,v_{\rm f}'(\rho)$.
\end{itemize}

\begin{figure}
      \centering
      \begin{psfrags}
      \psfrag{x}[c,c]{$t$}
      \psfrag{t}[c,c]{$x$}
      \psfrag{r}[c,c]{$\rho$}
      \psfrag{v}[c,c]{$v$}
      \includegraphics[width=.49\textwidth]{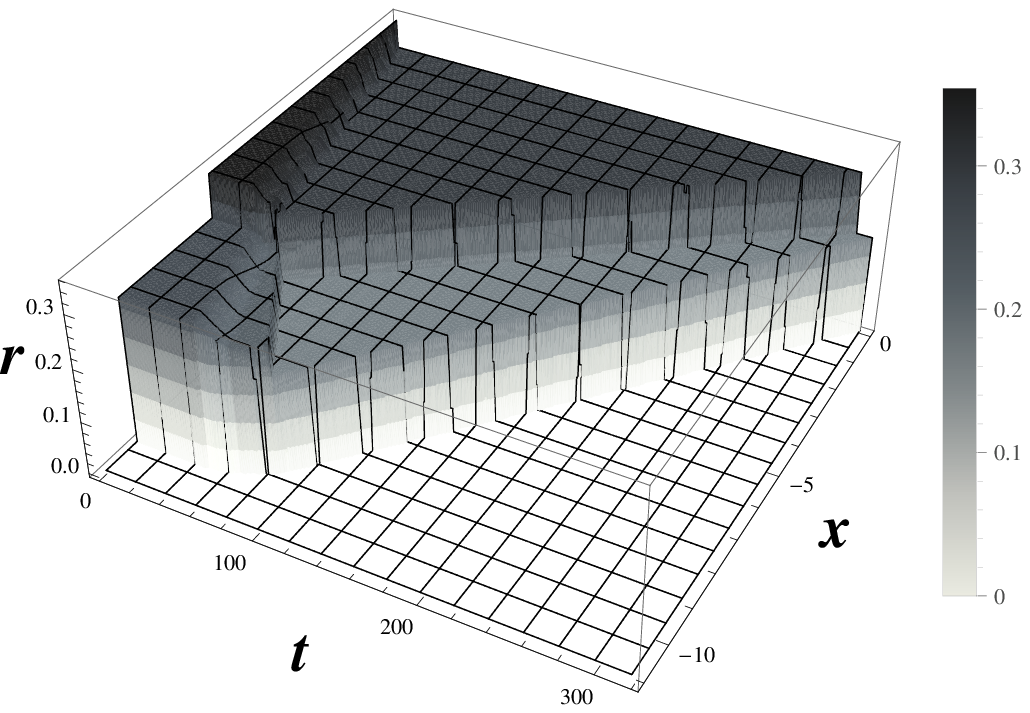}\hfill
      \psfrag{x}[c,c]{$x$}
      \psfrag{t}[c,c]{$t$}
      \includegraphics[width=.49\textwidth]{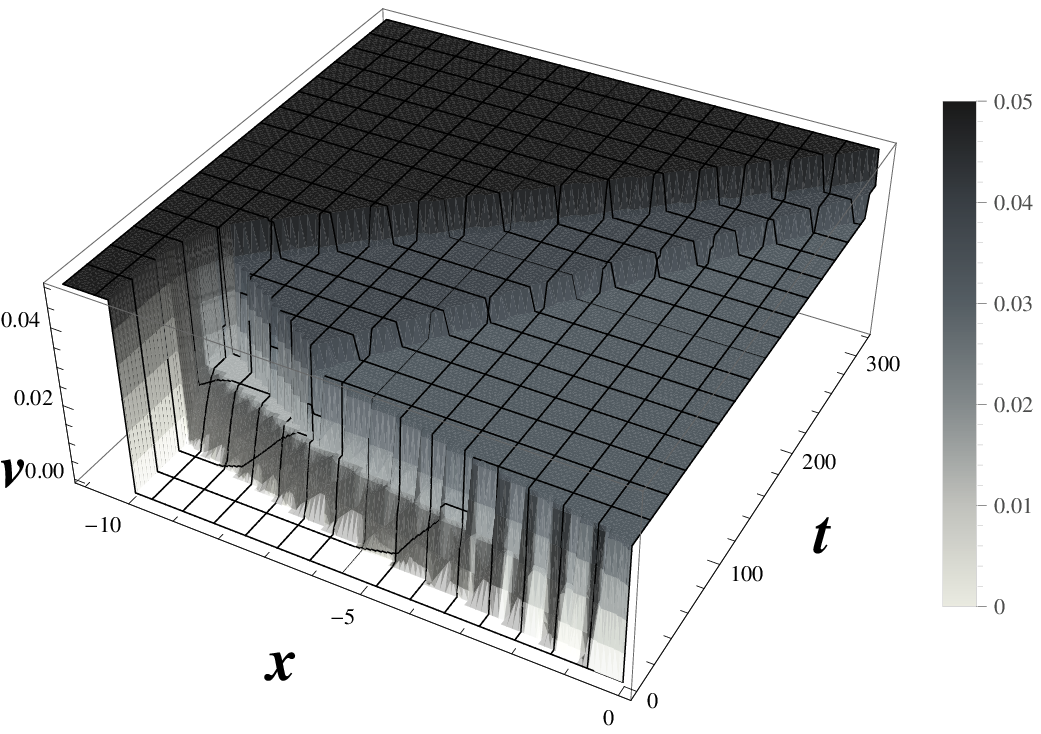}
      \end{psfrags}
      \caption{The solution $(t,x)\mapsto u(u,x) = \left(\rho(t,x),u(t,x)\right)$ contructed in Section~\ref{sec:ex} and corresponding to the numerical data~\eqref{eq:numericaldata}.}
      \label{fig:simuxtrExtv}
\end{figure}

To prolong then the solution, we have to consider the Riemann problems corresponding to each interaction as follows:
\begin{itemize}[leftmargin=*]
\item
The contact discontinuity $\mathcal{C}_2$ meets the rarefaction $\mathcal{R}_\ell$ in $a_2$ given by
\begin{align*}
&a_2\colon&
&x_{a_2}\doteq x_2,&
&t_{a_2}\doteq \frac{x_2}{\lambda_1\left(R_{\max}, 0\right)} .
\end{align*}
The result of this interaction is a contact discontinuity $\mathcal{C}_2'$, which accelerates during its interaction with $\mathcal{R}_\ell$ according to the following ordinary differential equation
\begin{align*}
&\mathcal{C}_2'\colon&
&\dot{x}_{\mathcal{C}_2'}(t) = v_{\mathcal{R}_\ell}\left(\frac{x_{\mathcal{C}_2'}(t)}{t}\right),&
&x_{\mathcal{C}_2'}(t_{a_2})=x_{a_2} .
\end{align*}
\item
The contact discontinuity $\mathcal{C}_2'$ stops to interact with the rarefaction $\mathcal{R}_\ell$ once it reaches $b_2$, implicitly given by
\begin{align*}
&b_2\colon&
&x_{b_2} = x_{\mathcal{C}_2'}(t_{b_2}),&
&x_{b_2} = t_{b_2} \, \lambda_1\left(p^{-1}(W_{\max}-V_{\rm c}),V_{\rm c}\right).
\end{align*}
Then, a contact discontinuity $\mathcal{C}_2''$ from $(p^{-1}(W_{\rm c}-V_{\rm c}),V_{\rm c})$ to $(p^{-1}(W_{\max}-V_{\rm c}),V_{\rm c})$ starts from $b_2$.

\item
The contact discontinuity $\mathcal{C}_2''$ reaches the phase transition $\mathcal{PT}_2$ in $c_2$, implicitly given by
\begin{align*}
&c_2\colon&
&x_{c_2} = t_{c_2} \, \sigma\left(p^{-1}(W_{\max}-V_{\rm c}),V_{\rm c},R_{\rm f}'',V_{\rm f}\right),&
&x_{c_2} - x_{b_2} = \left(t_{c_2} - t_{b_2}\right) V_{\rm c} .
\end{align*}
The result of this interaction is a phase transition $\mathcal{PT}_2'$ from $(p^{-1}(W_{\rm c}-V_{\rm c}),V_{\rm c})$ to $(R_{\rm f}',v_{\rm f}(R_{\rm f}'))$, followed by a shock $\mathcal{S}_2$ from $(R_{\rm f}',v_{\rm f}(R_{\rm f}'))$ to $(R_{\rm f}'',V_{\rm f})$.

\item
Each point of $\{(\mathcal{C}_2'(t),t) \colon t_{a_2}\leq t \leq t_{b_2}\}$, is the center of a rarefaction appearing on its left.
Denote by $\mathcal{R}_\ell'$ the juxtaposition of these rarefactions.
In order to compute the values attained by $\mathcal{R}_\ell'$ it is sufficient to apply the following rules:
\begin{itemize}[leftmargin=*]
\item
The velocity $v$ is continuous across the contact discontinuity $\mathcal{C}_2'$.
\item
The Lagrangian marker of the solution takes the constant value $W_{\rm c}$ in $\mathcal{R}_\ell'$.
\item
For any $(x_0,t_0) \in \{(\mathcal{C}_2'(t),t) \colon t_{a_2}\leq t \leq t_{b_2}\}$ and $t>t_0$ sufficiently small, the density $\rho$ in $\mathcal{R}_\ell'$ is constant along  
\begin{align}\label{eq:rays}
&\mathcal{P}\colon&
&x=x_0 + (t-t_0) \, \lambda_1
\left( 
p^{-1}\left(W_{\rm c} - v_{\mathcal{R}_\ell}\left(\frac{x_0}{t_0}\right)\right),
v_{\mathcal{R}_\ell}\left(\frac{x_0}{t_0}\right)
\right) .
\end{align}
\end{itemize}
As a consequence, the value of $\rho_{\mathcal{R}_\ell'}$ at any point $(t,x)$ of the rarefaction $\mathcal{R}_\ell'$ is equal to $p^{-1}(W_{\rm c}-v_{\mathcal{R}_\ell})$ computed at $(t_0 , x_{\mathcal{C}_2'}(t_0))$, with $t_0 = \mathcal{P}(t,x)$ obtained by ``projecting'' $(t,x)$ along~\eqref{eq:rays} to a point of $\mathcal{C}_2'$, namely
\begin{align*}
&\mathcal{R}_\ell'\colon&
&\rho_{\mathcal{R}_\ell'}(t,x) \doteq p^{-1}\left(W_{\rm c}-v_{\mathcal{R}_\ell}\left(\frac{x_{\mathcal{C}_2'}\left(\mathcal{P}(t,x)\right)}{\mathcal{P}(t,x)}\right)\right),&
&v_{\mathcal{R}_\ell'}(t,x) \doteq v_{\mathcal{R}_\ell}\left(\frac{x_{\mathcal{C}_2'}\left(\mathcal{P}(t,x)\right)}{\mathcal{P}(t,x)}\right).
\end{align*}
Observe that by definition $\mathcal{P}(t,x)$ belongs to $\left[t_{a_2},t_{b_2}\right]$ for all $(t,x)$ in $\mathcal{R}_\ell'$ and that, beside the density $\rho$, also the velocity $v$ is constant along~\eqref{eq:rays}.

\item
The stationary phase transition $\mathcal{PT}_1$ meets the rarefaction $\mathcal{R}_\ell'$ in $a_1$ given by
\begin{align*}
&a_1\colon&
&x_{a_1} \doteq x_1,&
&t_{a_1} \doteq  t_{a_2} + \frac{x_1-x_2}{\lambda_1(R_{\rm c},0)}.
\end{align*}
The result of this interaction is a phase transition $\mathcal{PT}_1'$, which accelerates during its interaction with $\mathcal{R}_\ell'$ according to the following ordinary differential equation
\begin{align*}
&\mathcal{PT}_1' \colon&
&\dot{x}_{\mathcal{PT}_1'}(t) = v_{\mathcal{R}_\ell'}\left(t,x_{\mathcal{PT}_1'}(t)\right),&
&x_{\mathcal{PT}_1'}(t_{a_1})=x_{a_1}.
\end{align*}

\item
The phase transition $\mathcal{PT}_1'$ stops to interact with the rarefaction $\mathcal{R}_\ell'$ once it reaches $b_1$, implicitly given by
\begin{align*}
&b_1\colon&
&x_{b_1} = x_{\mathcal{PT}_1'}(t_{b_1}),&
&x_{b_1}-x_{b_2}= \left(t_{b_2}-t_{b_1}\right) \lambda_1\left(p^{-1}(W_{\rm c}-V_{\rm c}),V_{\rm c}\right).
\end{align*}
Then, a phase transition $\mathcal{PT}_1''$ from $(0,V_{\max})$ to $(p^{-1}(W_{\rm c}-V_{\rm c}),V_{\rm c})$ starts from $b_1$.

\item
The phase transitions $\mathcal{PT}_1''$ and $\mathcal{PT}_2'$ meet in $c_1$, implicitly given by
\begin{align*}
&c_1\colon&
&\frac{x_{c_1} - x_{b_1}}{t_{c_1} - t_{b_1}} = V_{\rm c},&
&\frac{x_{c_1} - x_{c_2}}{t_{c_1} - t_{c_2}} = \sigma\left(p^{-1}(W_{\rm c}-V_{\rm c}),V_{\rm c},R_{\rm f}',v_{\rm f}(R_{\rm f}')\right).
\end{align*}
The result of this interaction is a shock $\mathcal{S}_1$ from $(0,V_{\max})$ to $(R_{\rm f}',v_{\rm f}(R_{\rm f}'))$.

\end{itemize}

\begin{figure}
\centering
\renewcommand{\arraystretch}{1.5}
\begin{tabular}{>{\centering\bfseries}m{0.025\hsize} @{}>{\centering}m{0.218\hsize} @{}>{\centering}m{0.218\hsize} @{}>{\centering\arraybackslash}m{0.218\hsize} @{}>{\centering\arraybackslash}m{0.218\hsize}}
&$t=0$&$t=t_{a_2}$&$t=\frac{1}{2}\left[t_{a_2}+t_{b_2}\right]$&$t=t_{b_2}$\vspace{2pt}
\\
\rotatebox{90}{$x\mapsto \rho(t,x)$}&
{\begin{psfrags}
\psfrag{x}[l,c]{$x$}
\psfrag{r}[l,t]{$\rho$}
\includegraphics[width=.8\hsize]{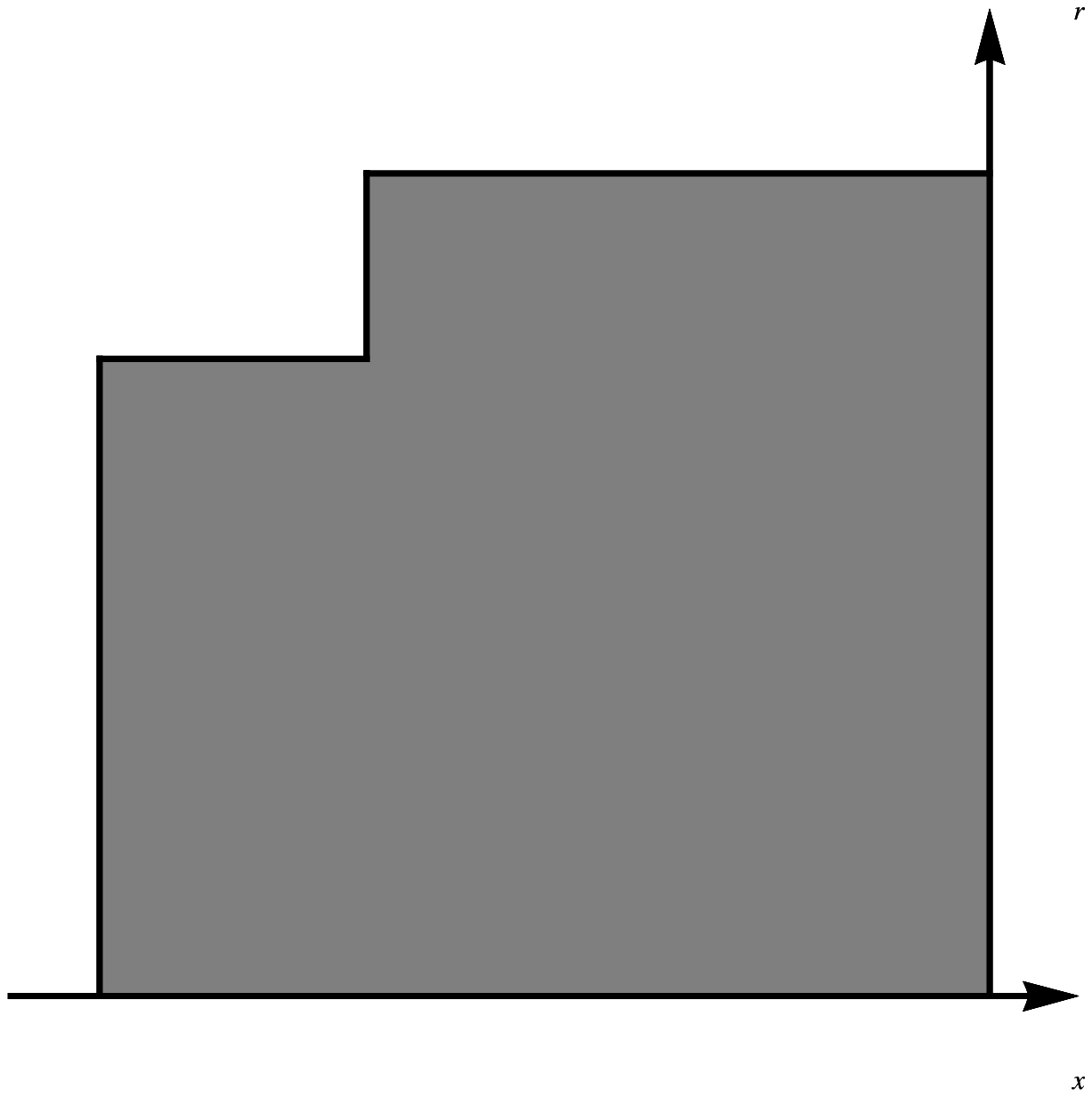}
\end{psfrags}}&
{\begin{psfrags}
\psfrag{x}[l,c]{$x$}
\psfrag{r}[l,t]{$\rho$}
\includegraphics[width=.8\hsize]{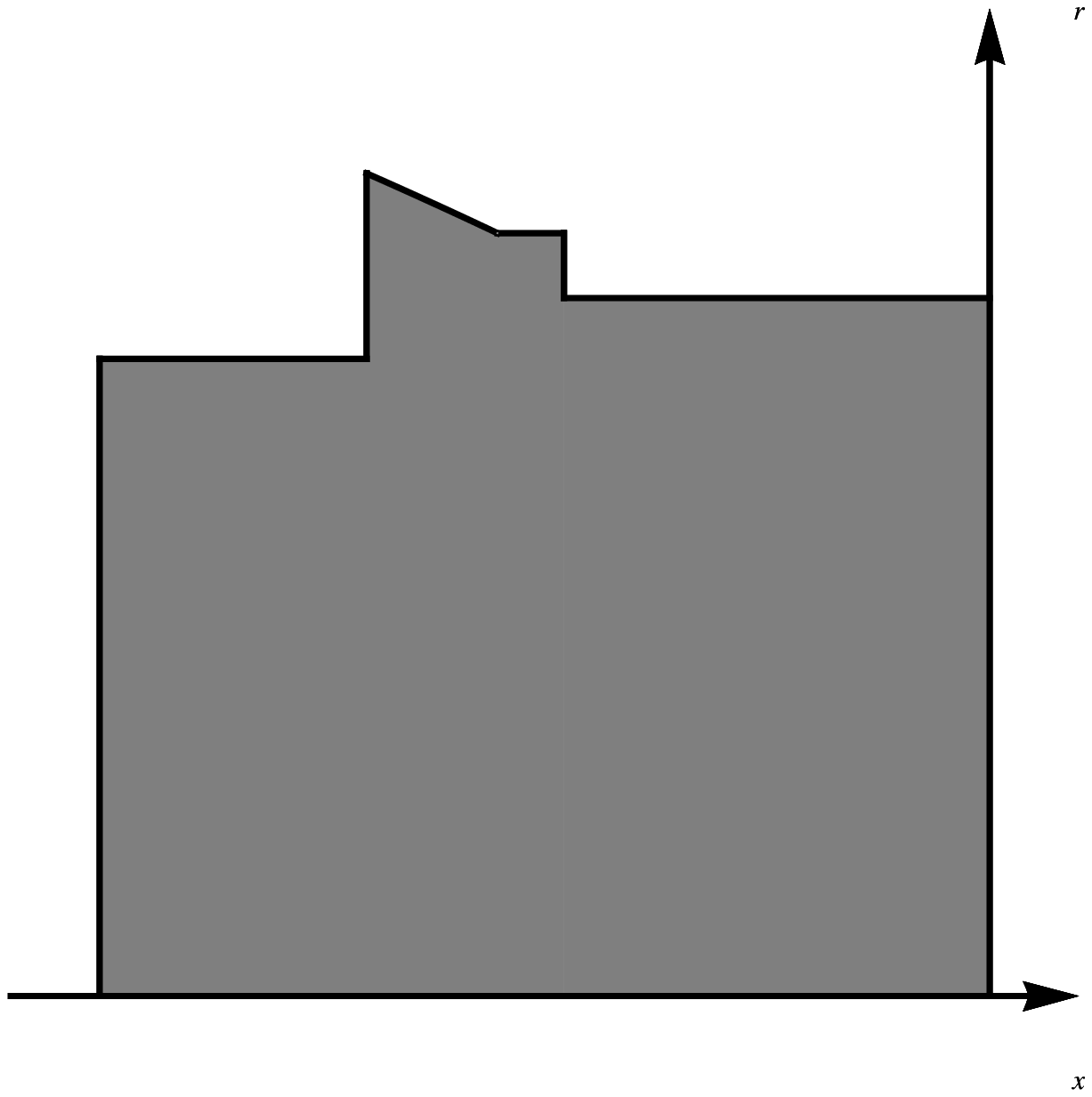}
\end{psfrags}}&
{\begin{psfrags}
\psfrag{x}[l,c]{$x$}
\psfrag{r}[l,t]{$\rho$}
\includegraphics[width=.8\hsize]{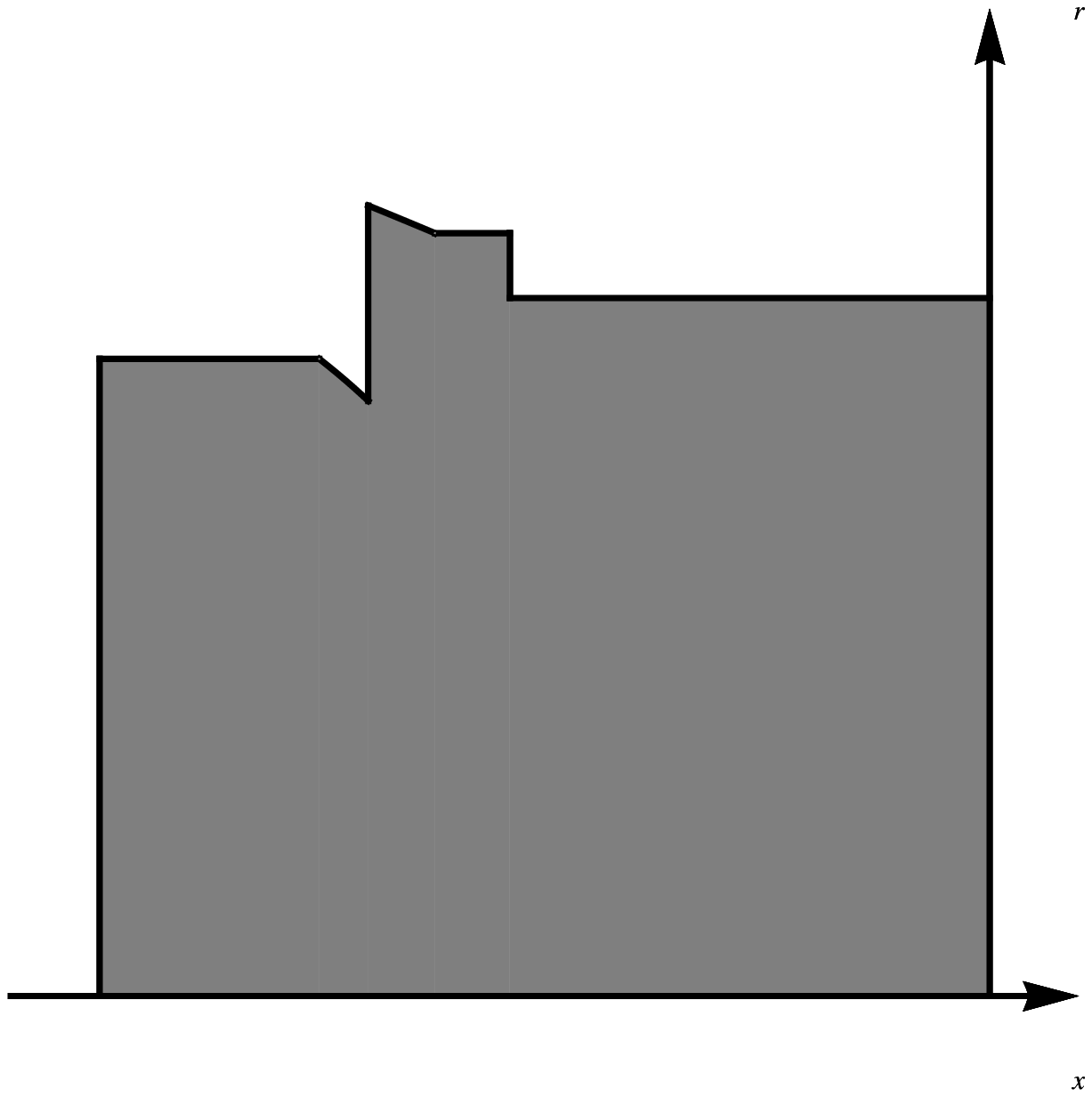}
\end{psfrags}}&
{\begin{psfrags}
\psfrag{x}[l,c]{$x$}
\psfrag{r}[l,t]{$\rho$}
\includegraphics[width=.8\hsize]{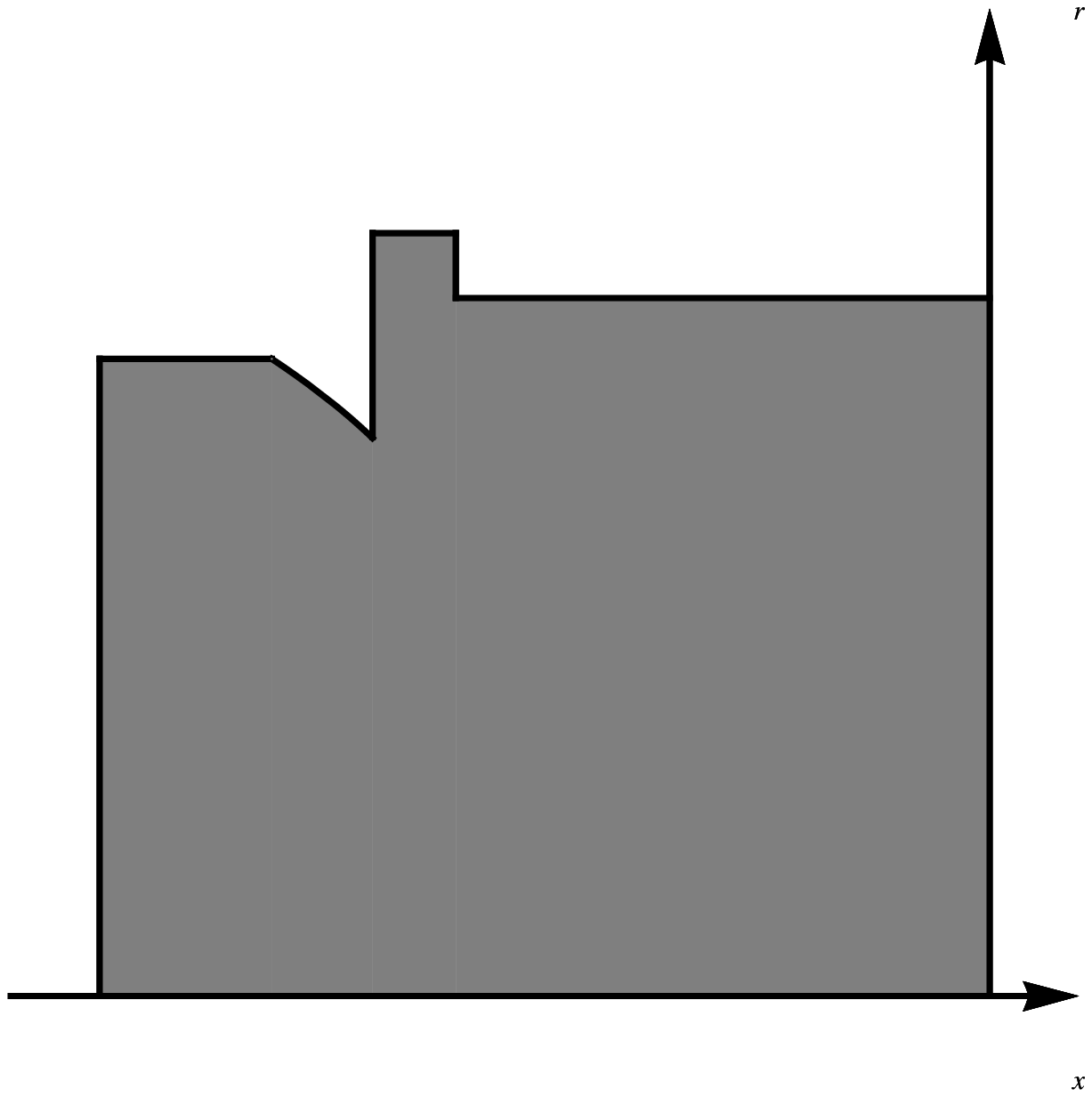}
\end{psfrags}}\vspace{2pt}
\\
\rotatebox{90}{$x\mapsto v(t,x)$}&
{\begin{psfrags}
\psfrag{x}[l,c]{$x$}
\psfrag{v}[l,t]{$v$}
\includegraphics[width=.8\hsize]{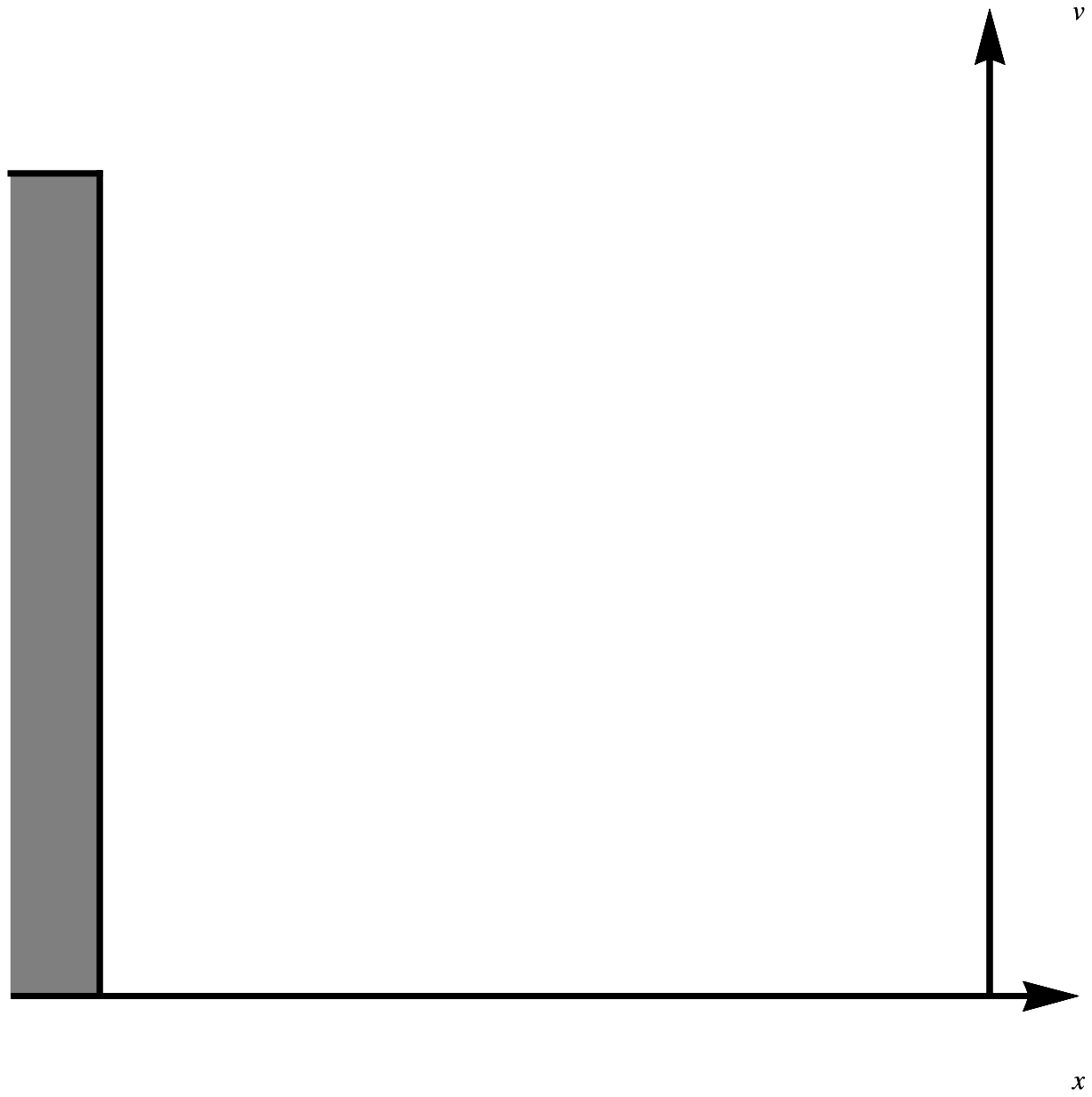}
\end{psfrags}}&
{\begin{psfrags}
\psfrag{x}[l,c]{$x$}
\psfrag{v}[l,t]{$v$}
\includegraphics[width=.8\hsize]{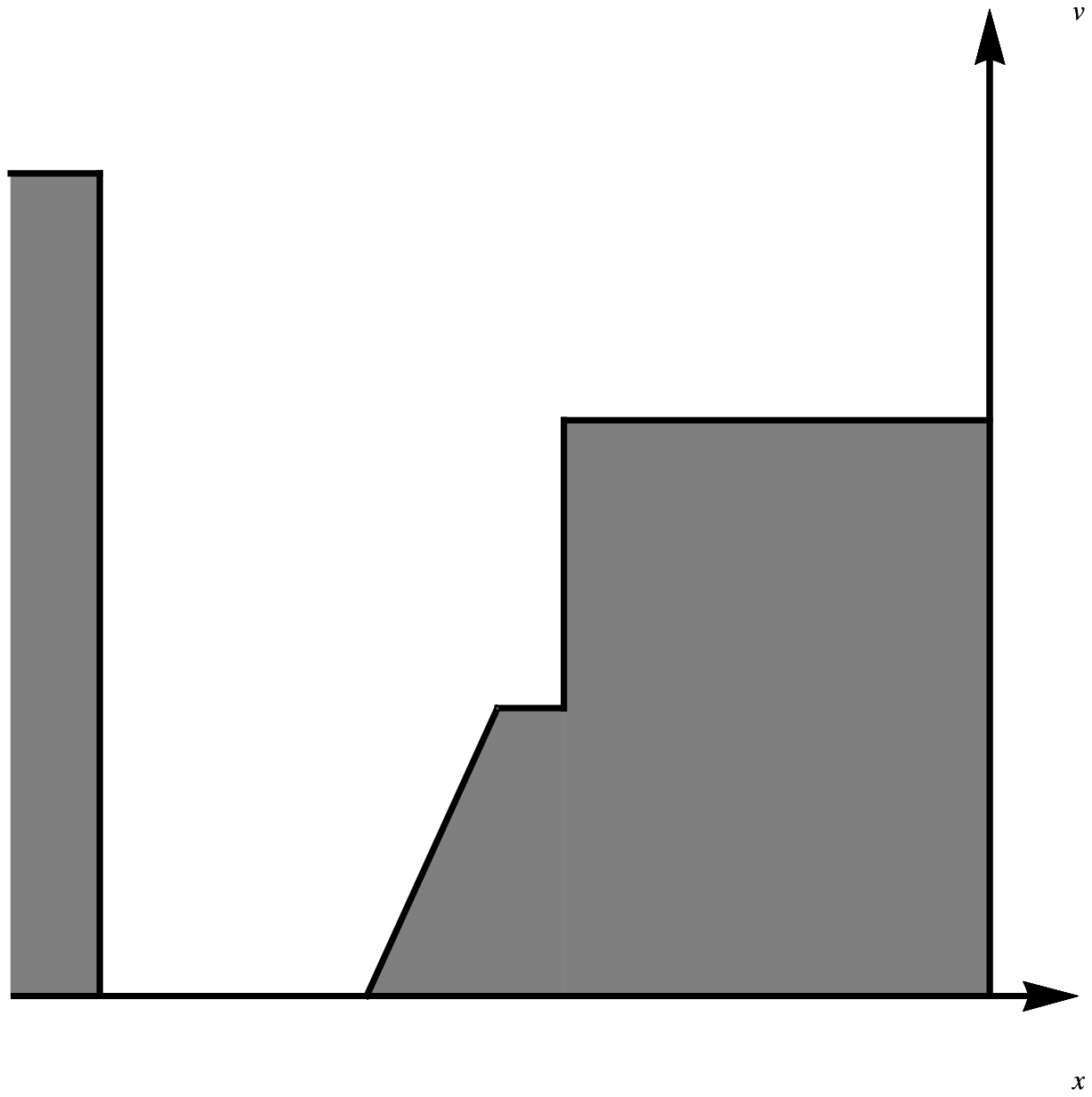}
\end{psfrags}}&
{\begin{psfrags}
\psfrag{x}[l,c]{$x$}
\psfrag{v}[l,t]{$v$}
\includegraphics[width=.8\hsize]{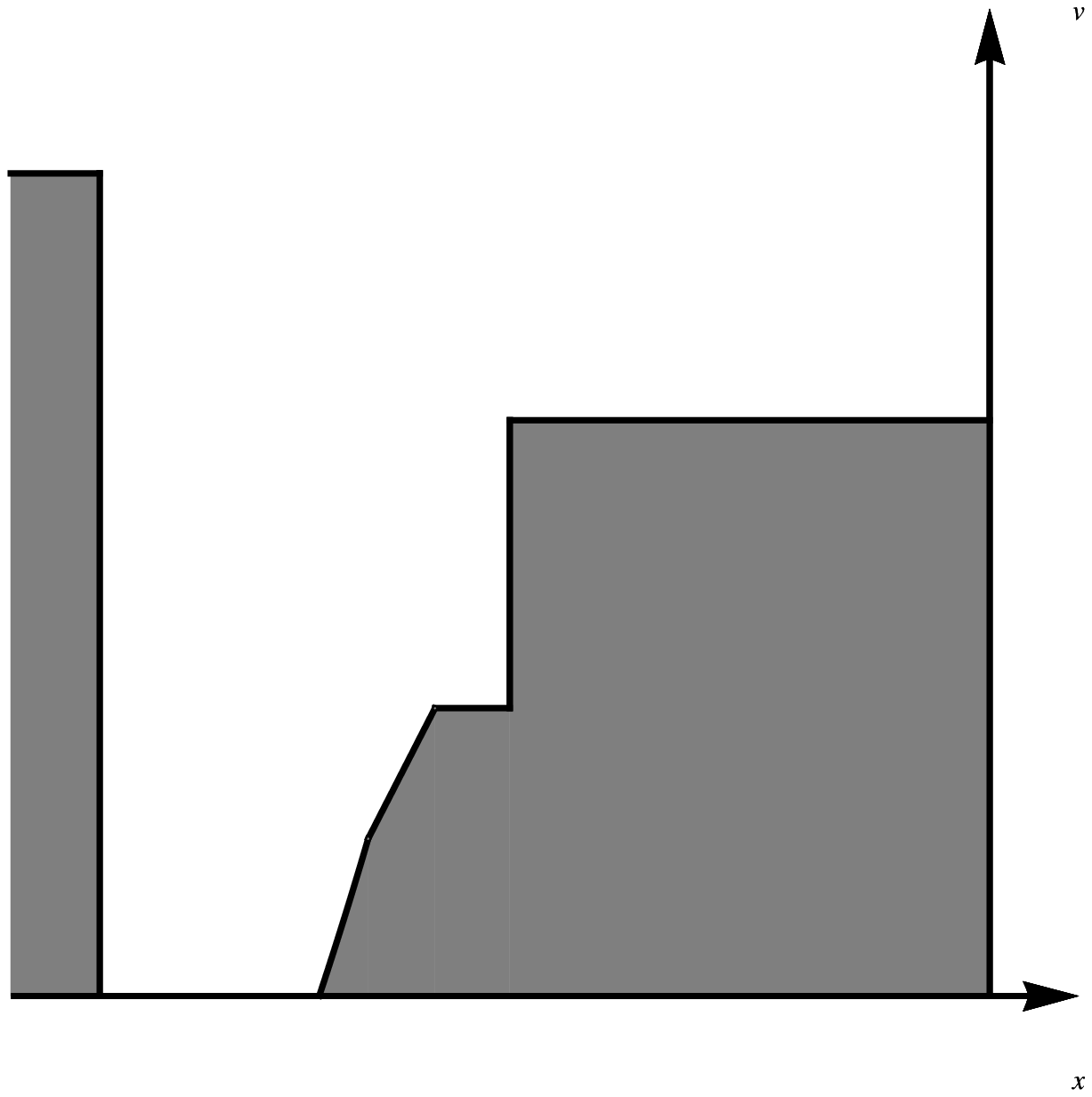}
\end{psfrags}}&
{\begin{psfrags}
\psfrag{x}[l,c]{$x$}
\psfrag{v}[l,t]{$v$}
\includegraphics[width=.8\hsize]{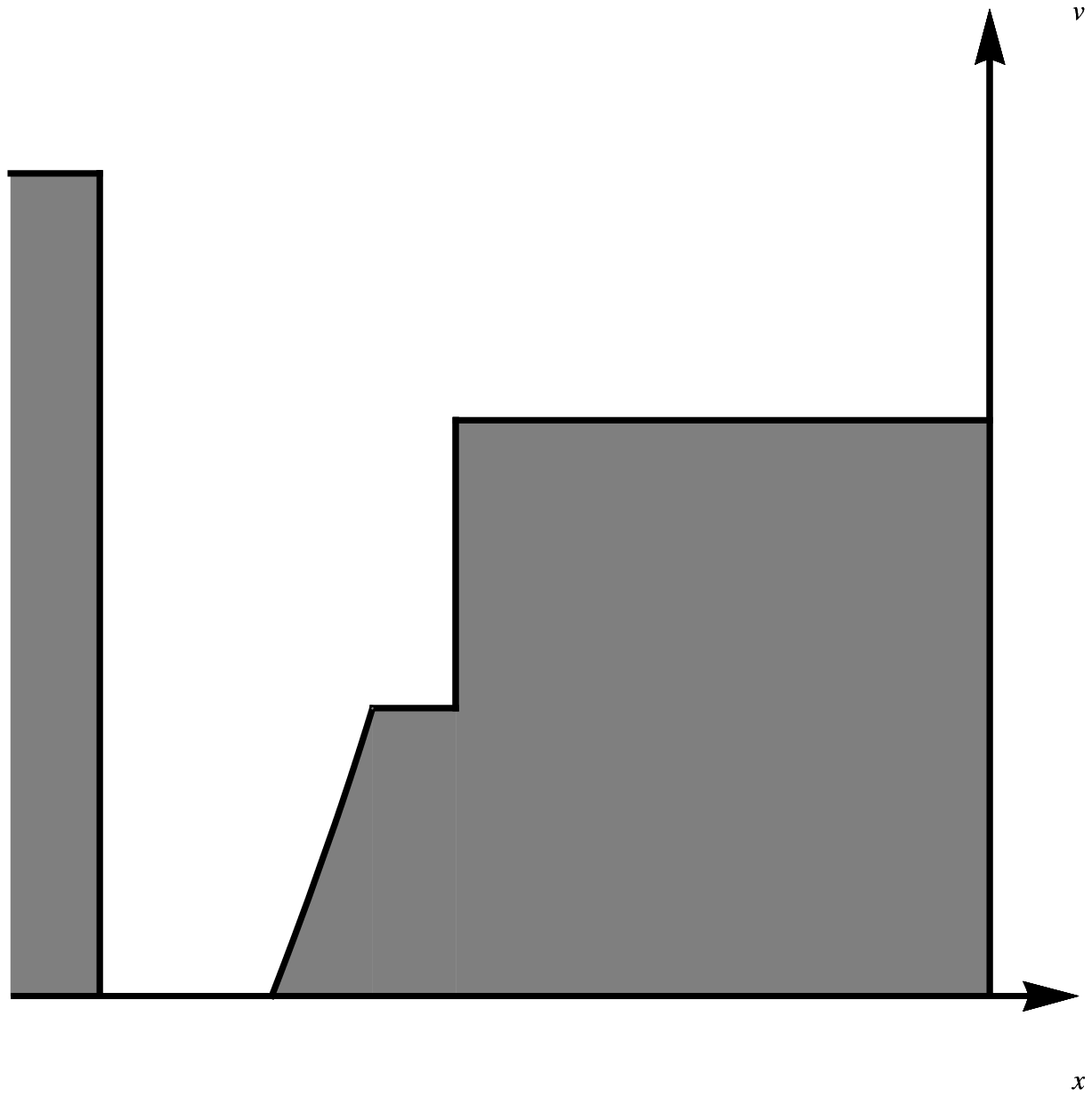}
\end{psfrags}}\vspace{2pt}
  \end{tabular}
\caption{Profiles of the solution $(t,x)\mapsto u(u,x) = \left(\rho(t,x),u(t,x)\right)$ contructed in Section~\ref{sec:ex} and corresponding to the numerical data~\eqref{eq:numericaldata}}
\label{fig:simuprophiles}
\end{figure}

Assume that the shocks $\mathcal{S}_1$ and $\mathcal{S}_2$ do not meet in $\R_-$.
Then, at times
\begin{align*}
&t_{d_2} \doteq -\frac{x_{c_2}}{\sigma\left(R_{\rm f}',v_{\rm f}(R_{\rm f}'),R_{\rm f}'',V_{\rm f}\right)}&
&\text{and}&
&t_{d_1} \doteq -\frac{x_{c_1}}{v_{\rm f}(R_{\rm f}')}
\end{align*}
the traffic light placed in $x = 0$ is reached respectively by $\mathcal{S}_1$ and $\mathcal{S}_2$.
Clearly, $t_{d_1}$ gives the time at which the last vehicle passes through $x=0$.

\begin{rem}
Once the overall picture of the solution is known, it is possible to express in a closed form the time at which the last vehicle passes through $x=0$.
Indeed, we first observe that the first long vehicle passes through $x=0$ at time
\[
t_\star \doteq \frac{R_{\max} \, |x_2|}{R_{\rm f}'' \, V_{\rm f}} ,
\]
with $t_\star < t_{d_2}$.
Then we observe that
\[
\left(x_2-x_1\right) R_{\rm c} = \left(t_{d_1} - t_{d_2}\right) R_{\rm f}' \, v_{\rm f}(R_{\rm f}') + \left(t_{d_2} - t_\star\right) R_{\rm f}'' \, V_{\rm f} .
\]
Moreover we can compute $x_{c_2}$ by solving the following system
\begin{align*}
&c_2\colon&
&x_{c_2} = t_{c_2} \, \sigma\left(p^{-1}(W_{\max}-V_{\rm c}),V_{\rm c},R_{\rm f}'',V_{\rm f}\right),&
&|x_2| \, R_{\max} = \left[|x_{c_2}| + t_{c_2} \, V_{\rm f} \right] R_{\rm f}'' ,
\end{align*}
namely
\begin{align*}
&t_{c_2} = \frac{R_{\max} \, |x_2|}{R_{\rm f}''\left[V_{\rm f} - \sigma\left(p^{-1}(W_{\max}-V_{\rm c}),V_{\rm c},R_{\rm f}'',V_{\rm f}\right)\right]} ,&
&x_{c_2} = \frac{R_{\max} \, \sigma\left(p^{-1}(W_{\max}-V_{\rm c}),V_{\rm c},R_{\rm f}'',V_{\rm f}\right) \, |x_2|}{R_{\rm f}''\left[V_{\rm f} - \sigma\left(p^{-1}(W_{\max}-V_{\rm c}),V_{\rm c},R_{\rm f}'',V_{\rm f}\right)\right]} .
\end{align*}
As a consequence we have that
\[
t_{d_2} 
= 
t_{c_2}-\frac{x_{c_2}}{\sigma\left(R_{\rm f}',v_{\rm f}(R_{\rm f}'),R_{\rm f}'',V_{\rm f}\right)} 
= 
\left[ 1 -
\frac{\sigma\left(p^{-1}(W_{\max}-V_{\rm c}),V_{\rm c},R_{\rm f}'',V_{\rm f}\right)}{\sigma\left(R_{\rm f}',v_{\rm f}(R_{\rm f}'),R_{\rm f}'',V_{\rm f}\right)}\right]
\frac{R_{\max} \, |x_2|}{R_{\rm f}''\left[V_{\rm f} - \sigma\left(p^{-1}(W_{\max}-V_{\rm c}),V_{\rm c},R_{\rm f}'',V_{\rm f}\right)\right]}  ,
\]
and therefore
\begin{align*}
t_{d_1} =&~
\frac{\left(x_2-x_1\right) R_{\rm c} - x_2 \, R_{\max}}{R_{\rm f}' \, v_{\rm f}(R_{\rm f}')}
\\
&~+
\left[ 1 -
\frac{\sigma\left(p^{-1}(W_{\max}-V_{\rm c}),V_{\rm c},R_{\rm f}'',V_{\rm f}\right)}{\sigma\left(R_{\rm f}',v_{\rm f}(R_{\rm f}'),R_{\rm f}'',V_{\rm f}\right)}\right]
\left[\vphantom{\frac{\sigma\left(p^{-1}(W_{\max}-V_{\rm c}),V_{\rm c},R_{\rm f}'',V_{\rm f}\right)}{\sigma\left(R_{\rm f}',v_{\rm f}(R_{\rm f}'),R_{\rm f}'',V_{\rm f}\right)}}
1 - \frac{R_{\rm f}'' \, V_{\rm f}}{R_{\rm f}' \, v_{\rm f}(R_{\rm f}')}\right]
\frac{R_{\max} \, |x_2|}{R_{\rm f}''\left[V_{\rm f} - \sigma\left(p^{-1}(W_{\max}-V_{\rm c}),V_{\rm c},R_{\rm f}'',V_{\rm f}\right)\right]}
.
\end{align*}
\end{rem}

The resulting solution represented in \figurename~\ref{fig:02}, \figurename~\ref{fig:simuxtrExtv} and \figurename~\ref{fig:simuprophiles} correspond to the following numerical values:
\begin{align}\label{eq:numericaldata}
&\gamma = 2 ,&
&V_{\max} = \frac{1}{20} ,&
&W_{\max} = \frac{1}{8} ,&
&W_{\rm c} = \frac{4}{30} ,&
&x_1 = -10 ,&
&x_2 = -7 .
\end{align}

We conclude the section by underlying that to construct the solution in $\R_+$ it is sufficient to apply the theory for LWR.

\section{Wave-front tracking algorithm}\label{sec:wft}

In this section we apply the wave-front tracking method to construct weak solutions to the Cauchy problem~\eqref{eq:model}, \eqref{eq:initial} that are only approximate entropic.

\subsection{Exact and approximate Riemann solvers}\label{sec:exactpprR}

We construct piecewise constant approximate solutions to the Cauchy problem~\eqref{eq:model}, \eqref{eq:initial} with initial datum $\bar{u}$ in $\mathcal{D} \doteq \{u \in \L1(\R;\Omega) \colon \tv(u) \le M\}$ by juxtaposing solutions of Riemann problems for~\eqref{eq:model}, i.e.~of Cauchy problems for~\eqref{eq:model} with the Heaviside initial data
\begin{equation}
\label{eq:Riemanndatum}
u(0,x) = \begin{cases}
u_\ell&\text{if }x<0,\\
u_r&\text{if }x>0.
\end{cases}
\end{equation}
We denote below by $\mathcal{R}_{\rm LWR}$ and $\mathcal{R}_{\rm ARZ}$ the Riemann solvers for respectively LWR and ARZ.
Then, the maps
\begin{align*}
&(t,x) \mapsto \mathcal{R}_{\rm LWR}[u_\ell,u_r](x/t)&
&\text{and}&
&(t,x) \mapsto \mathcal{R}_{\rm ARZ}[u_\ell,u_r](x/t)
\end{align*}
are the self similar Lax solutions of the Riemann problems for respectively LWR and ARZ with initial datum~\eqref{eq:Riemanndatum}, with $(u_\ell,u_r)$ being respectively in $\Omega_{\rm f} \times \Omega_{\rm f}$ and $\Omega_{\rm c} \times \Omega_{\rm c}$.
Moreover, for any $u_\ell,u_r \in \Omega$ with $\rho_\ell\neq\rho_r$, we denote by $\sigma(u_\ell,u_r)$ the speed of propagation of a discontinuity from $u_\ell$ to $u_r$ and given by~\eqref{eq:sigma}.
In the following definition we introduce a Riemann solver obtained by generalizing that one given in~\cite{goatin2006aw} to the setting presented in Section~\ref{sec:notations}.

\begin{figure}[ht]
\centering{
\begin{psfrags}
      \psfrag{a}[c,B]{$\rho$}
      \psfrag{b}[c,B]{}
      \psfrag{g}[c,B]{$\rho\,v$}
      \psfrag{r}[c,B]{$\rho_r$}
      \psfrag{s}[c,B]{$\rho_m$}
      \psfrag{t}[c,B]{$\rho_\ell$}
\includegraphics[width=.3\textwidth]{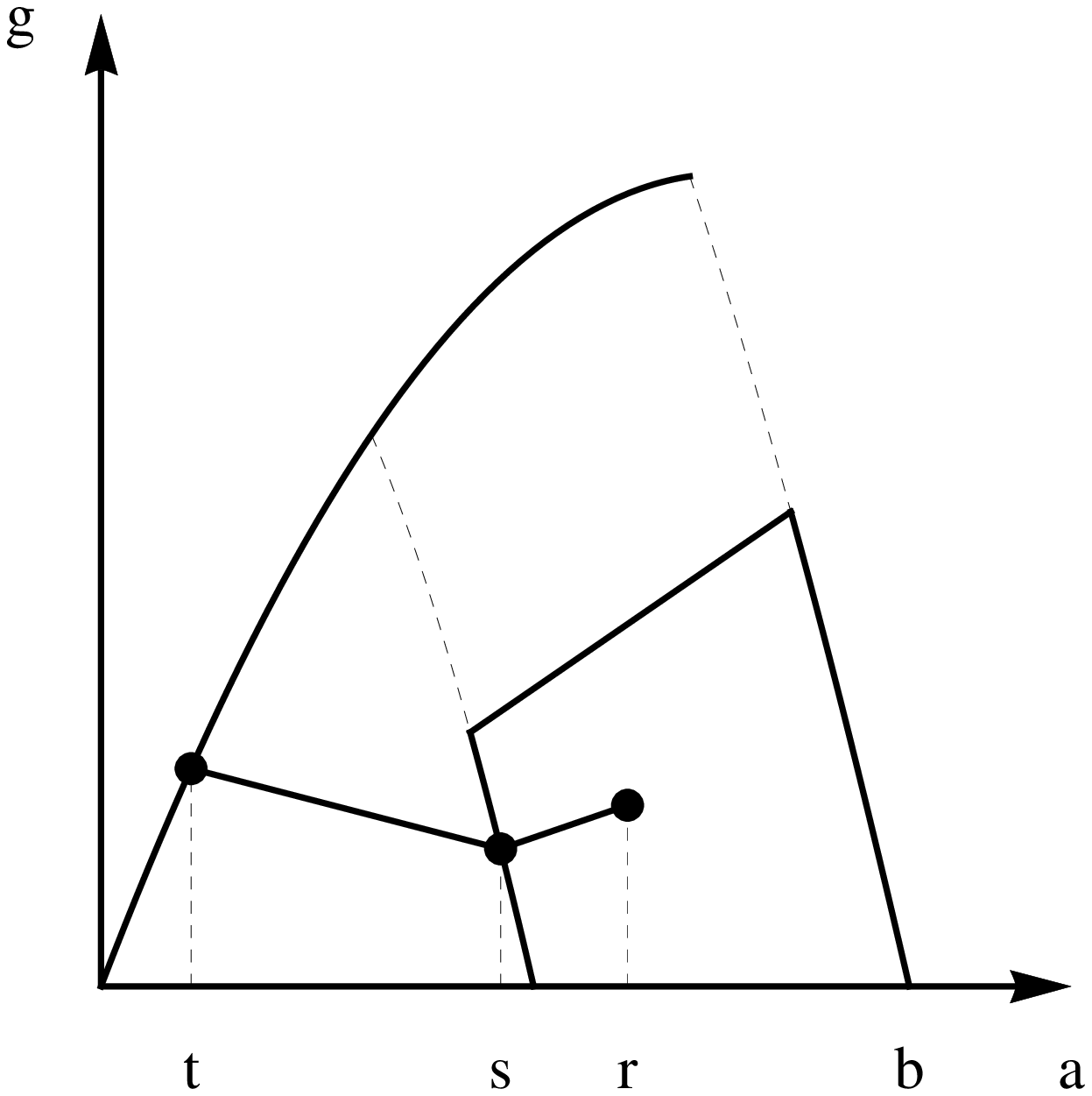}\quad
\includegraphics[width=.3\textwidth]{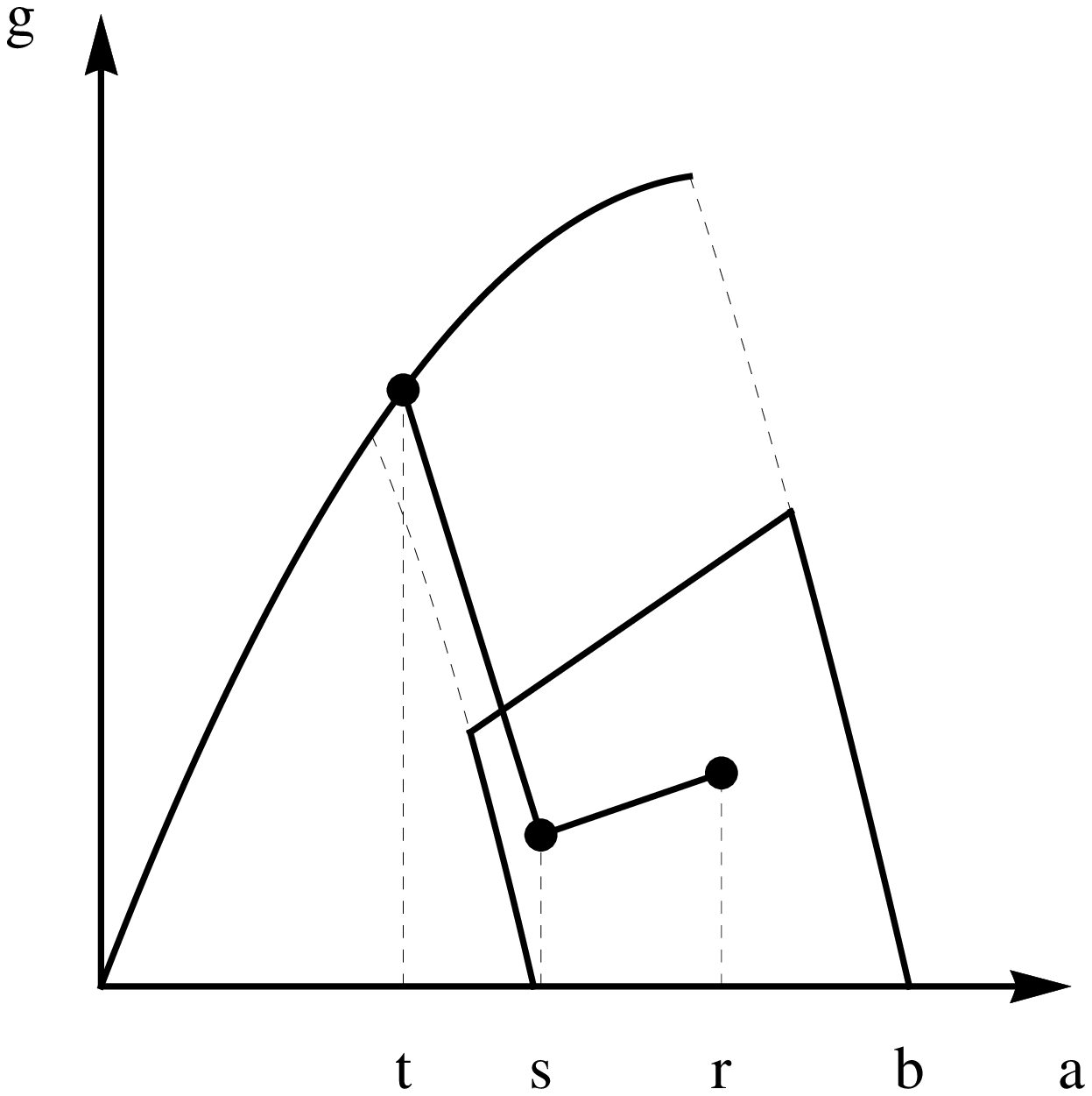}\quad
      \psfrag{s}[c,B]{$\rho_m''$}
      \psfrag{z}[c,B]{$~~\,\rho_m'$}
      \psfrag{t}[l,B]{$\rho_\ell$}
      \psfrag{r}[l,B]{$~\rho_r$}
\includegraphics[width=.3\textwidth]{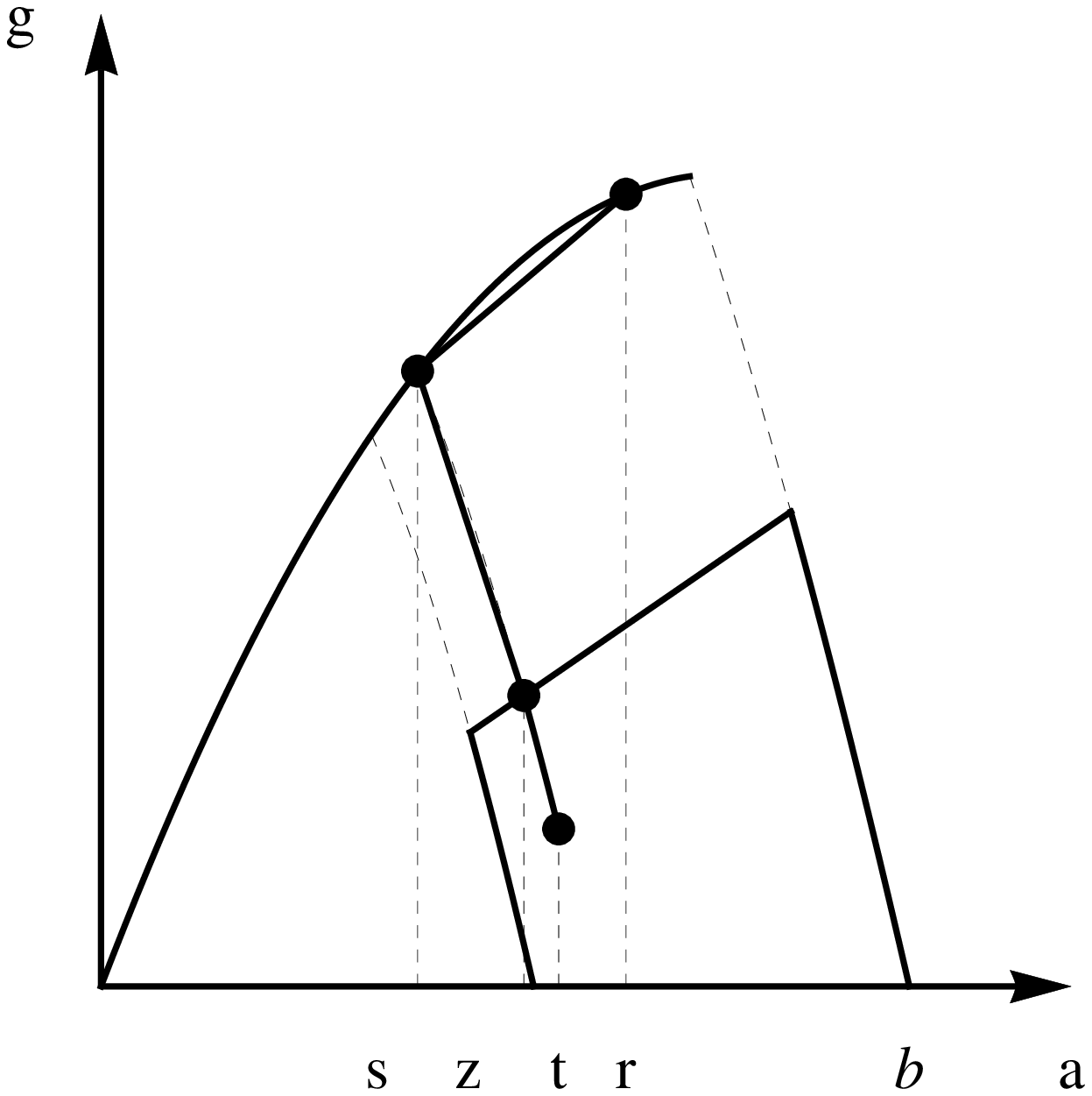}\\[5pt]
      \psfrag{v}[c,B]{$v$}
      \psfrag{w}[c,b]{$w$}
      \psfrag{0}[c,B]{$w_r$}
      \psfrag{1}[c,B]{$W_{\rm c}$}
      \psfrag{2}[c,B]{$w_\ell$}
      \psfrag{3}[l,B]{$v_\ell$}
      \psfrag{4}[c,c]{$v_r$}
      \psfrag{6}[c,B]{$~~V_{\rm f}$}
\includegraphics[width=.3\textwidth]{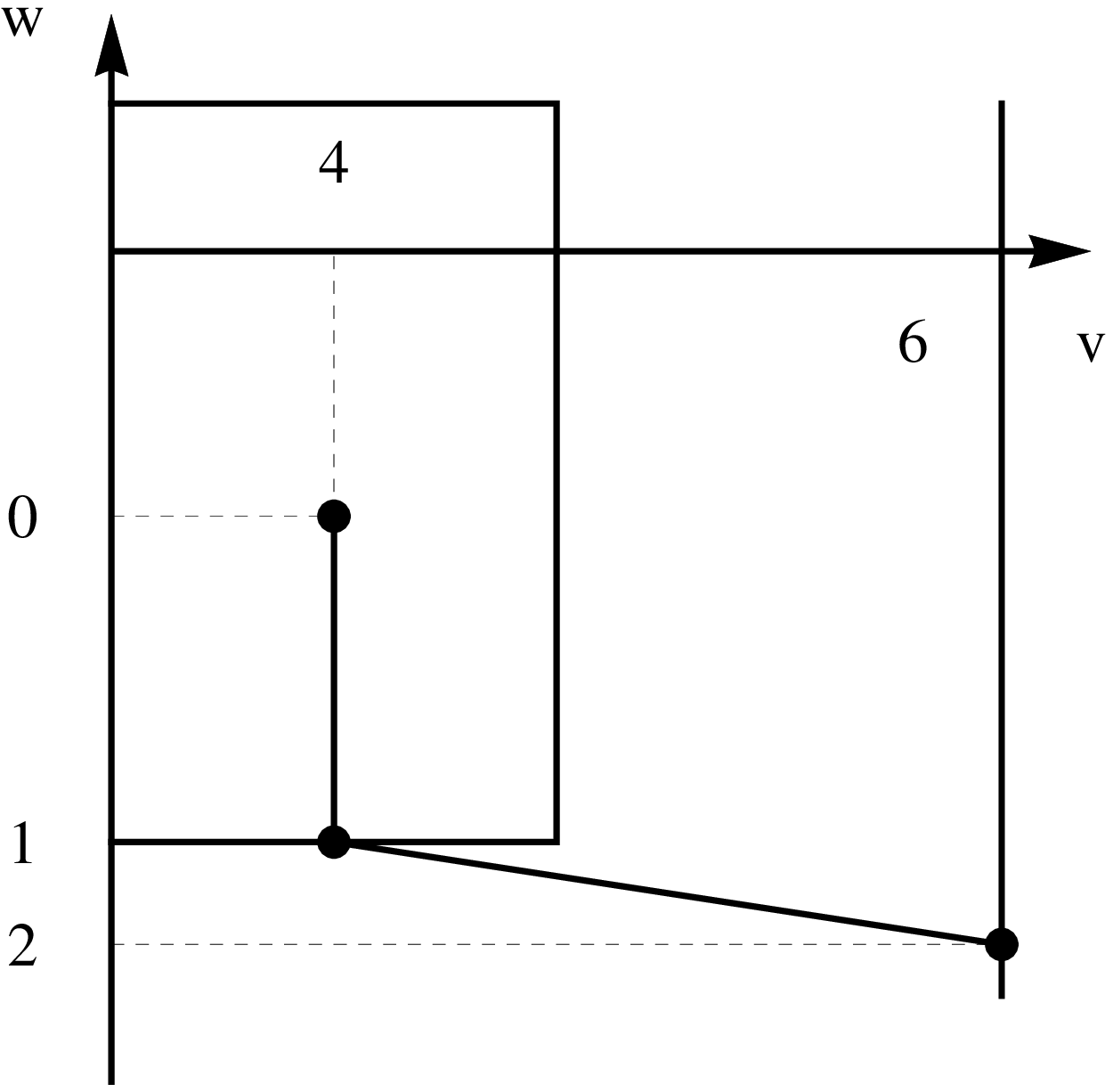}\quad
\includegraphics[width=.3\textwidth]{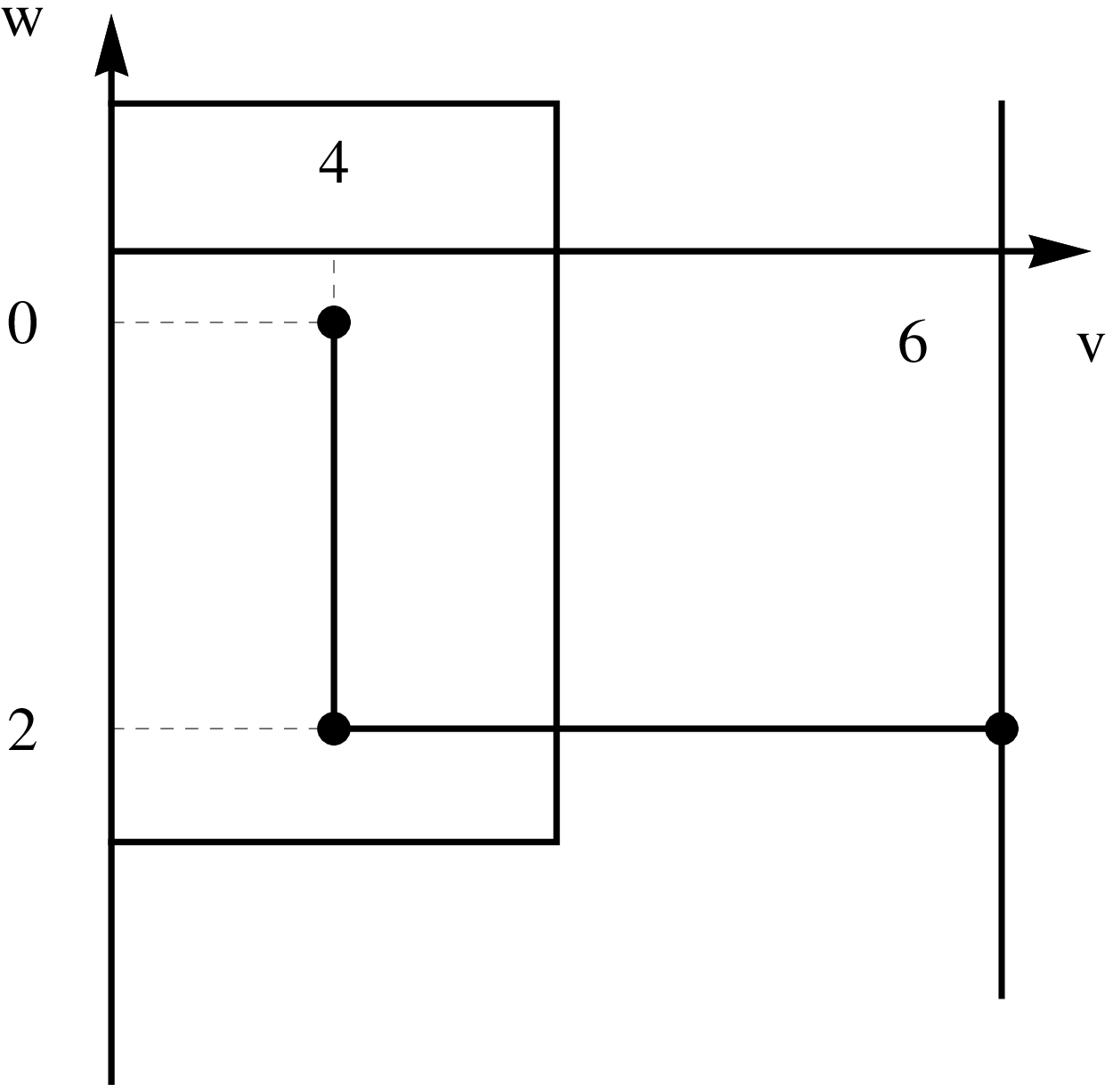}\quad
      \psfrag{1}[c,B]{$V_{\rm c}$}
\includegraphics[width=.3\textwidth]{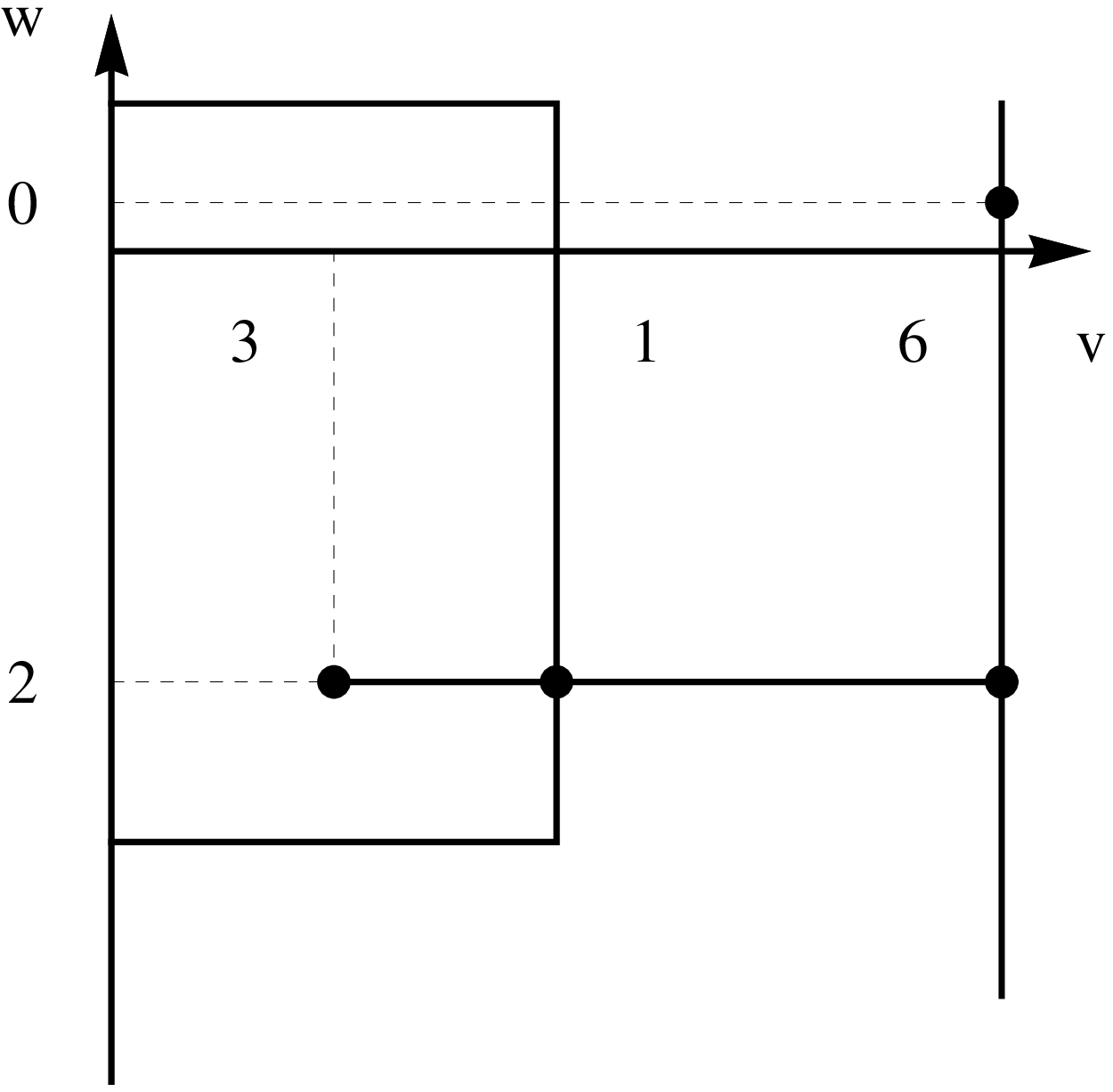}
\end{psfrags}}
\caption{Solutions of the Riemann problem~\eqref{eq:model}, \eqref{eq:Riemanndatum} given by the Riemann solver $\mathcal{R}$ introduced in Definition~\ref{def:LWR-ARZ}.}
\end{figure}

\begin{definition}\label{def:LWR-ARZ}
The Riemann solver $\mathcal{R} \colon \Omega \times \Omega \to \Lloc1(\R;\Omega)$ associated to~\eqref{eq:model} is defined as follows:

\begin{enumerate}[label={(R\arabic*)},leftmargin=*]

\item
If $u_\ell$ and $u_r$ belong to the same phase domain, then $\mathcal{R}$ coincides with the Lax Riemann solver, namely
\[
\mathcal{R}[u_\ell,u_r] \doteq 
\begin{cases}
\mathcal{R}_{\rm LWR}[u_\ell,u_r]&\text{if }u_\ell, u_r \in \Omega_{\rm f},\\
\mathcal{R}_{\rm ARZ}[u_\ell,u_r]&\text{if }u_\ell, u_r \in \Omega_{\rm c}.
\end{cases}
\]

\item\label{LWR-ARZ2} If $u_\ell \in \Omega_{\rm f}$ and $u_r \in \Omega_{\rm c}$, then
\[
\mathcal{R}[u_\ell,u_r](x) \doteq 
\begin{cases}
u_\ell&\text{if } x<\sigma(u_\ell,u_m),\\
\mathcal{R}_{\rm ARZ}[u_m,u_r](x)&\text{if } x>\sigma(u_\ell,u_m),
\end{cases}
\]
where $\rho_m \doteq p^{-1}\left(\max\left\{W_{\rm c}, w_2(u_\ell)\right\} - v_r\right)$ and $v_m \doteq v_r$.

\item\label{LWR-ARZ3} If $u_\ell \in \Omega_{\rm c}$ and $u_r \in \Omega_{\rm f}$, then
\[
\mathcal{R}[u_\ell,u_r](x) \doteq 
\begin{cases}
\mathcal{R}_{\rm ARZ}[u_\ell,u_m'](x)&\text{if } x<\sigma(u_m',u_m''),\\
\mathcal{R}_{\rm LWR}[u_m'',u_r](x)&\text{if } x>\sigma(u_m',u_m''),
\end{cases}
\]
where $\rho_m' \doteq p^{-1}(w_2(u_\ell) - V_{\rm c})$, $v_m' \doteq V_{\rm c}$, $\rho_m'' \doteq \rho_{\rm f}(w_2(u_\ell))$ and $v_m'' \doteq v_{\rm f}(\rho_m'')$.

\end{enumerate}
\end{definition}
\noindent
According to the above definition, in the case~\ref{LWR-ARZ2} we have that $\mathcal{R}[u_\ell,u_r]$ performs a phase transition from $u_\ell$ to $u_m$, followed by a possible null contact discontinuity from $u_m$ to $u_r$.
In particular, if $\rho_\ell = 0$, then $\mathcal{R}[u_\ell,u_r]$ performs a single phase transition:
\begin{align*}
&u_r \in \Omega_{\rm c}&
&\Rightarrow&
&\mathcal{R}[0,V_{\max},u_r](x) \doteq 
\begin{cases}
(0,V_{\max})&\text{if } x<v_r,\\
u_r&\text{if } x>v_r.
\end{cases}
\end{align*}
Finally, in the case~\ref{LWR-ARZ3} we have that $\mathcal{R}[u_\ell,u_r]$ performs a possible null rarefaction from $u_\ell$ to $u_m'$, a phase transition from $u_m'$ to $u_m''$, followed by a possible null Lax wave from $u_m''$ to $u_r$.

In the next proposition we collect the main properties of $\mathcal{R}$.
In particular we prove that $\mathcal{R}$ is \emph{consistent}, namely that it satisfies the following two conditions for any $u_\ell,u_m,u_r$ in $\Omega$ and $\bar x$ in $\R$:
\begin{align*}
&{\rm (I)}&
  &\mathcal{R}[u_\ell,u_r](\bar x) = u_m
  & \Rightarrow &
  &&\begin{cases}
      \mathcal{R}[u_\ell,u_m](x)= 
      \begin{cases}
          \mathcal{R}[u_\ell,u_r](x)& \hbox{ if } x\leq \bar x ,
          \\
          u_m & \hbox{ if } x > \bar x ,
      \end{cases}
      \\[10pt]
      \mathcal{R}[u_m,u_r](x) =
      \begin{cases}
          u_m & \hbox{ if } x < \bar x ,
          \\
          \mathcal{R}[u_\ell,u_r](x)& \hbox{ if } x \geq \bar x.
      \end{cases}
      \end{cases}
  \\[0pt]
&{\rm (II)}&
  &\begin{rcases}
  \mathcal{R}[u_\ell,u_m](\bar x)=u_m 
  \\
  \mathcal{R}[u_m,u_r](\bar x)=u_m
  \end{rcases}
  & \Rightarrow &
  &&\mathcal{R}[u_\ell,u_r](x)=
      \begin{cases}
      \mathcal{R}[u_\ell,u_m](x)& \hbox{ if } x < \bar x ,
      \\
      \mathcal{R}[u_m,u_r](x) & \hbox{ if } x \geq \bar x .
      \end{cases}
\end{align*}
We recall that this property is a necessary condition for the well posedness of the Cauchy problem in $\L1$.

\begin{proposition}
The Riemann solver $\mathcal{R} \colon \Omega \times \Omega \to \L\infty(\R;\Omega)$ satisfies the following conditions:
\begin{enumerate}[label={(\arabic*)},leftmargin=*]
\item $\mathcal{R}$ is consistent.
\item $\mathcal{R} \colon \Omega\times\Omega \to \Lloc1(\R;\Omega)$ is continuous.
\item For any $u_\ell, u_r \in \Omega$, we have that $(t,x) \mapsto \mathcal{R}[u_\ell,u_r](x/t)$ is an entropy solution of~\eqref{eq:model}, \eqref{eq:Riemanndatum} in the sense of Definition~\ref{def:entropyPT}.
\end{enumerate}
\end{proposition}
\begin{proof}
Since both $\mathcal{R}_{\rm LWR}$ and $\mathcal{R}_{\rm ARZ}$ satisfy the above conditions respectively in $\Omega_{\rm f} \times \Omega_{\rm f}$ and $\Omega_{\rm c} \times \Omega_{\rm c}$, we have to consider only the cases when a phase transition occurs.
\begin{enumerate}[label={(\arabic*)},leftmargin=*]
\item 
We first prove the property~(I).
Assume that $\mathcal{R}[u_\ell,u_r](\bar x) = u_m$ with $u_m$ distinct from $u_\ell$ and $u_r$.
If $u_\ell$ and $u_r$ are respectively in $\Omega_{\rm f}$ and $\Omega_{\rm c}$, then the only possible choice for $u_m$ is $u_m = (p^{-1}\left(\max\left\{W_{\rm c},w_2(u_\ell)\right\} - v_r\right), v_r) \in \Omega_{\rm c}$ and the result immediately follows.
If $u_\ell$ and $u_r$ are respectively in $\Omega_{\rm c}$ and $\Omega_{\rm f}$, then either $v_\ell < V_{\rm c}$, $u_m \in \Omega_{\rm c}$ and $w_2(u_m) = w_2(u_\ell)$ or $u_m \in \Omega_{\rm f}$ with $v_m$ between $v_{\rm f}(\rho_{\rm f}(w_2(u_\ell)))$ and $v_r$.
To conclude the proof it is sufficient in the first case to observe that $\rho_{\rm f}(w_2(u_m)) = \rho_{\rm f}(w_2(u_\ell))$, while in the latter case it is sufficient to exploit the consistency of $\mathcal{R}_{\rm LWR}$ .
\\
To prove the property~(II) it is sufficient to consider the following two cases.
First, assume that $\mathcal{R}[u_\ell,u_m]$ is given by just one phase transition and $\mathcal{R}[u_\ell,u_m](\bar x) = u_m$.
If $u_\ell \in \Omega_{\rm f}$, then $\mathcal{R}[u_m,u_r](\bar x) = u_m$ implies that $v_m = v_r$.
If $u_\ell \in \Omega_{\rm c}$, then $v_\ell = V_{\rm c}$ and $\rho_m = \rho_{\rm f}(w_2(u_\ell))$. 
Clearly therefore in both cases (II) holds true.
The second case to be considered is when $\mathcal{R}[u_m,u_r]$ is given by a phase transition and $\mathcal{R}[u_m,u_r](\bar x) = u_m$.
In this case the only possibility is to have $u_r \in \Omega_{\rm f}''$, $u_\ell \in \Omega_{\rm c}$, $v_m = V_{\rm c}$ and $w_2(u_r) = w_2(u_m) = w_2(u_\ell)$, and therefore the result is obvious.

\item
Assume that $\mathcal{R}[u_\ell,u_r]$ is given by just one phase transition.
Consider $u_\ell^\varepsilon$ and $u_r^\varepsilon$ such that $\norma{u_\ell^\varepsilon - u_\ell} \le \varepsilon$ and $\norma{u_r^\varepsilon - u_r} \le \varepsilon$.
If $\rho_\ell \ne 0$, then it is easy to prove that $\mathcal{R}[u_\ell^\varepsilon,u_r^\varepsilon]$ converges to $\mathcal{R}[u_\ell,u_r]$ in $\Lloc1$ by observing that also $\mathcal{R}[u_\ell^\varepsilon,u_r^\varepsilon]$ is given by a single phase transition and by exploiting the continuity of $\sigma$.
If $\rho_\ell = 0$, then
\[
\mathcal{R}[u_\ell^\varepsilon,u_r^\varepsilon](x) = 
\begin{cases}
u_\ell^\varepsilon&\text{if }x<\sigma(u_\ell^\varepsilon,u_m^\varepsilon),
\\
u_m^\varepsilon&\text{if }\sigma(u_\ell^\varepsilon,u_m^\varepsilon) < x<\sigma(u_m^\varepsilon,u_r^\varepsilon),
\\
u_r^\varepsilon&\text{if }x>\sigma(u_m^\varepsilon,u_r^\varepsilon),
\end{cases}
\]
where $\rho_m^\varepsilon \doteq p^{-1}(W_{\rm c}-v_r^\varepsilon)$ and $v_m^\varepsilon \doteq v_r^\varepsilon$.
Now, since both $\sigma(u_\ell^\varepsilon,u_m^\varepsilon)$ and $\sigma(u_m^\varepsilon,u_r^\varepsilon)$ converge to $\sigma(u_\ell,u_r)$ we deduce the $\Lloc1$-convergence of $\mathcal{R}[u_\ell^\varepsilon,u_r^\varepsilon]$ to $\mathcal{R}[u_\ell,u_r]$.

\item
The last property follows immediately by the Definition~\ref{def:LWR-ARZ} of $\mathcal{R}$.
Indeed, in both cases described in~\ref{LWR-ARZ2} and~\ref{LWR-ARZ3} of Definition~\ref{def:LWR-ARZ} we have that on each side of a phase transition $\mathcal{R}[u_\ell,u_r]$ can perform only Lax waves.\qedhere
\end{enumerate}
\end{proof}

\begin{figure}[ht]
\centering{
\begin{psfrags}
      \psfrag{a}[c,B]{$\rho$}
      \psfrag{g}[c,B]{$\rho\,v$}
\includegraphics[width=.3\textwidth]{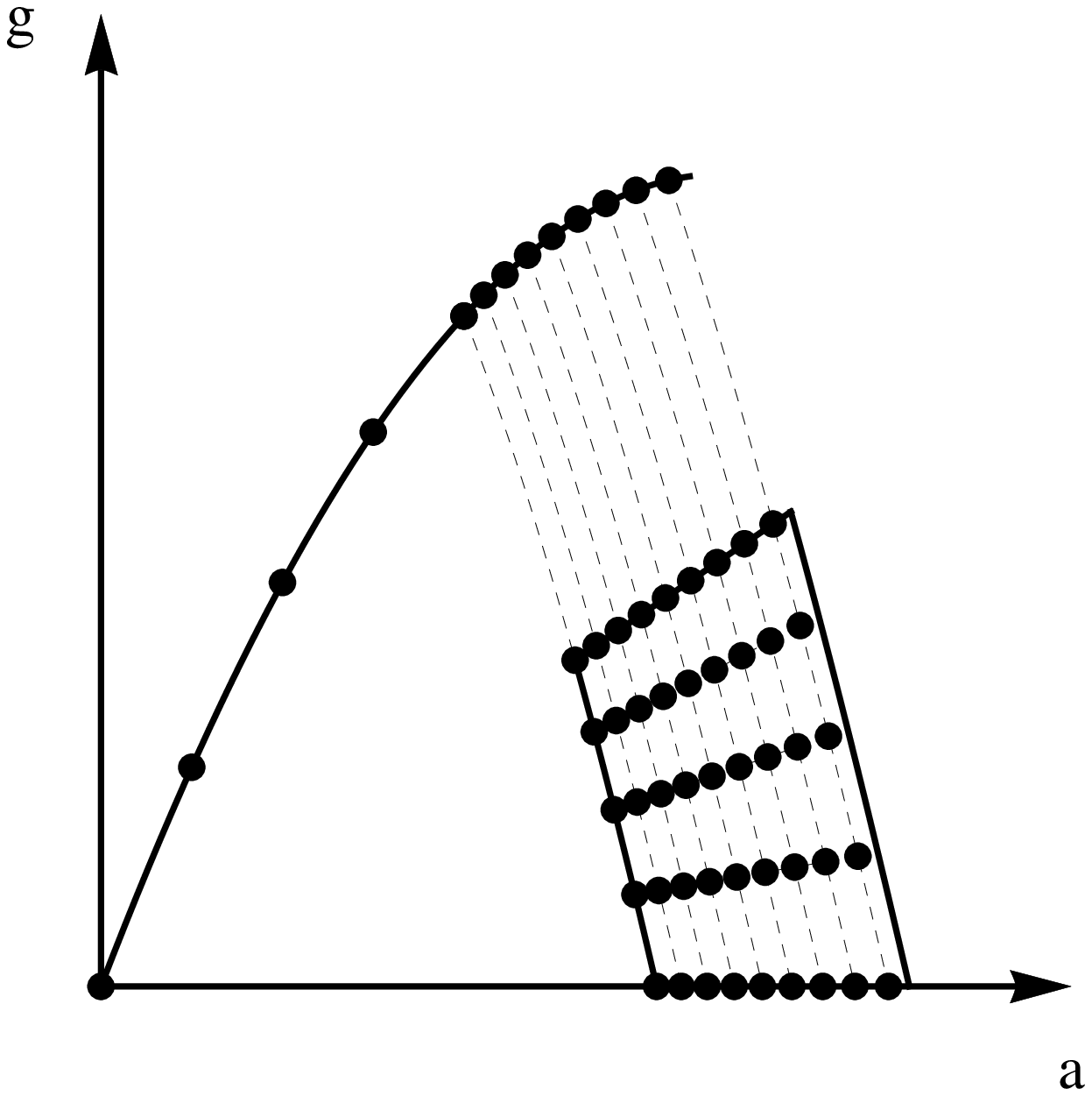}\qquad\qquad
      \psfrag{v}[c,B]{$v$}
      \psfrag{w}[c,B]{$w$}
\includegraphics[width=.3\textwidth]{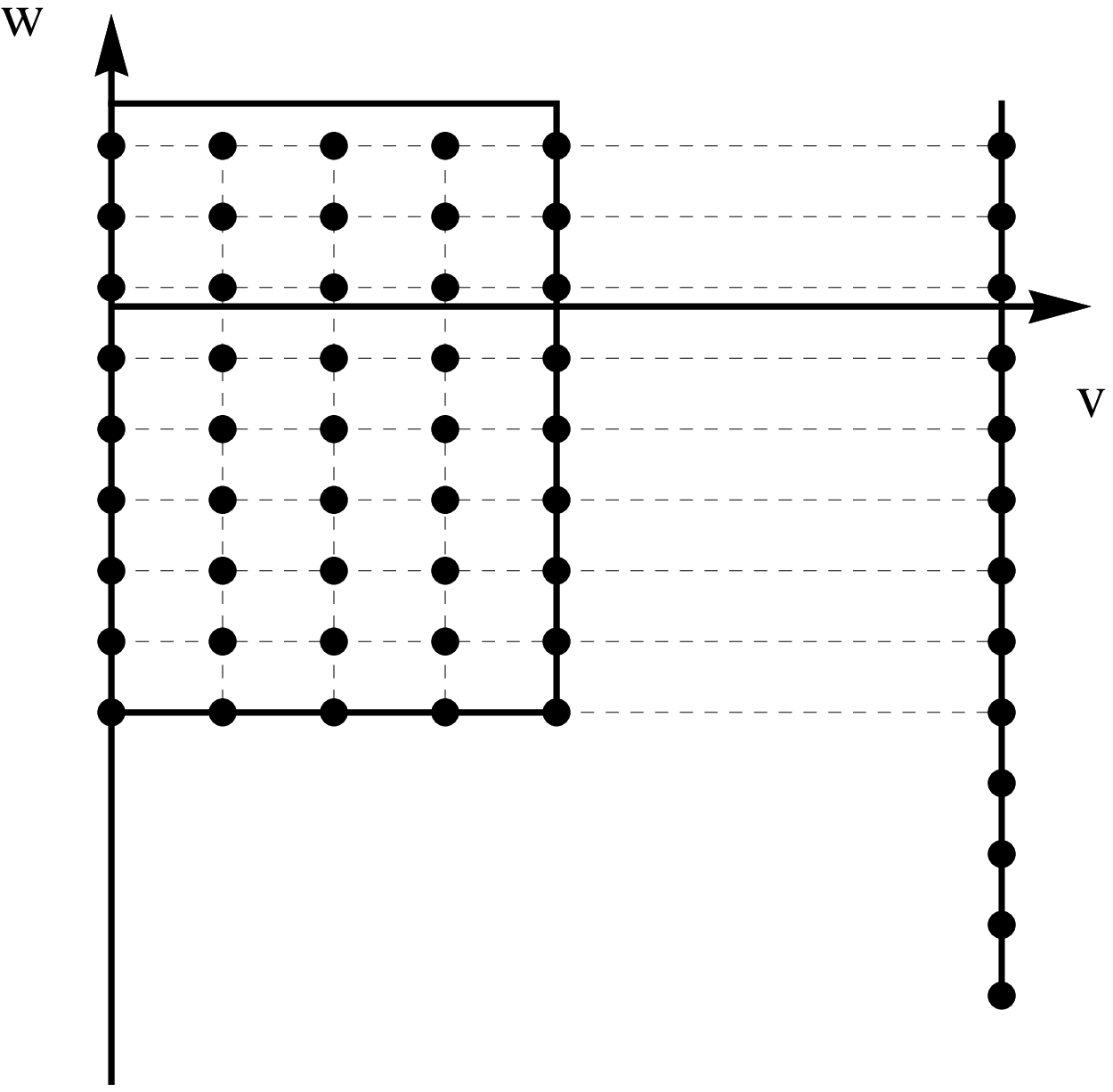}
\end{psfrags}}
\caption{The mesh $\Omega^n$ in the coordinates $(\rho,\rho\,v)$ and $(w_1,w_2)$ respectively on the left and on the right.}
\label{fig:mesh}
\end{figure}
Fix $n\in\naturals$ sufficiently large, let $\varepsilon^n_v \doteq 2^{-n}V_{\rm c}$, $\varepsilon^n_w \doteq 2^{-n}(W_{\rm c}-W_{\min})$ and introduce the grid $\Omega^n$ in $\Omega$, see \figurename~\ref{fig:mesh}, that in the Riemann coordinates writes
\[ \Omega^n \doteq \Omega \cap \left[ W^n \times V^n \right], \]
where
\begin{align*}
&W^n \doteq \left\{W_{\min}+i\,\varepsilon^n_w \colon i=0,\ldots,\left\lfloor\frac{W_{\max}-W_{\min}}{\varepsilon^n_w}\right\rfloor\right\},
&V^n \doteq \left\{i\,\varepsilon^n_v \colon i=0,\ldots,2^{n}\right\} \cup \left\{V_{\rm f}\right\}.
\end{align*}
Consider the approximate Riemann solver $\mathcal{R}^n \colon \Omega^n \times \Omega^n \to \Lloc1(\R;\Omega^n)$ obtained by discretizing the rarefaction waves of $\mathcal{R} \colon \Omega \times \Omega \to \Lloc1(\R;\Omega)$ as follows:
\begin{enumerate}[label={(Rn\arabic*)},leftmargin=*]
\item If $u_\ell,u_r \in \Omega_{\rm f} \cap \Omega^n$ and $w_2(u_r) = w_2(u_\ell) - j\,\varepsilon^n_w$, then
\[
\mathcal{R}^n[u_\ell,u_r](x) \doteq \begin{cases}
u_\ell&\text{if }x<\sigma(u_\ell,u_1),\\
u_i&\text{if } \sigma(u_{i-1}, u_i)< x < \sigma(u_i, u_{i+1}),
	\quad i=1,\ldots,j-1,\\
u_r&\text{if }x<\sigma(u_{j-1},u_r),
\end{cases}
\]
where $u_i \in \Omega_{\rm f} \cap \Omega^n$ are implicitly defined by $w_2(u_i) = w_2(u_\ell) - i\,\varepsilon^n_w$.
\item If $u_\ell,u_r \in \Omega_{\rm c} \cap \Omega^n$, $w_1(u_r) = w_1(u_\ell) + j\,\varepsilon^n_v$ and $w_2(u_r) = w_2(u_\ell)$, then
\[
\mathcal{R}^n[u_\ell,u_r](x) \doteq \begin{cases}
u_\ell&\text{if }x<\sigma(u_\ell,u_1),\\
u_i&\text{if } \sigma(u_{i-1}, u_i)< x < \sigma(u_i, u_{i+1}),
	\quad i=1,\ldots,j-1,\\
u_r&\text{if }x<\sigma(u_{j-1},u_r),
\end{cases}
\]
where $u_i \in \Omega_{\rm c} \cap \Omega^n$ are implicitly defined by $w_1(u_i) = w_1(u_\ell) + i\,\varepsilon^n_v$, $w_2(u_i) = w_2(u_\ell)$.
\end{enumerate}

\begin{proposition}\label{prop:M}
The grid $\Omega^n$ has the following properties:
\begin{enumerate}[label={(M\arabic*)},leftmargin=*]
\item For any point in $\Omega$ there is a point in $\Omega^n$ such that the distance between them is less than $\varepsilon^n_v+\varepsilon^n_w$.
\item For $n$ sufficiently big, any two distinct points in the mesh $\Omega^n$ are distant more that $2^{-1}\min\{\varepsilon^n_v, \varepsilon^n_w\}$.
\item The Riemann problem~\eqref{eq:model}, \eqref{eq:Riemanndatum} with $u_\ell, u_r \in \Omega^n$ admits a global weak solution attaining values in $\Omega^n$.\label{M3}
\end{enumerate}
\end{proposition}
\begin{proof}
The first two properties are obvious, while the weak solution satisfying the properties required in~\ref{M3} is constructed by applying the approximate Riemann solver $\mathcal{R}^n$.
Indeed, any discontinuity of $(t,x)\mapsto\mathcal{R}^n[u_\ell,u_r](x/t)$ clearly satisfies the Rankine-Hugoniot jump conditions.
\end{proof}
It is worth to note that $(t,x)\mapsto\mathcal{R}^n[u_\ell,u_r](x/t)$ may well not be an entropy solution even if $u_\ell$ and $u_r$ belong to the same phase domain.

\subsection{An approximate solution}\label{sec:aas}

An approximate solution $u^n$ to the Cauchy problem~\eqref{eq:model}, \eqref{eq:initial} with initial datum $\bar{u}$ is now constructed by applying the wave-front tracking algorithm and the approximate Riemann solver $\mathcal{R}^n$.
We first approximate $\bar{u}$ with $\bar{u}^n \in \PC(\R;\Omega^n)$ such that
\begin{align}\label{eq:apprinitial}
&\norma{\bar{u}^n}_{\L\infty(\R;\Omega)} \le \norma{\bar{u}}_{\L\infty(\R;\Omega)},&
&\lim_{n\to+\infty}\norma{\bar{u}^n-\bar{u}}_{\L1(\R;\Omega)}=0,&
&\tv(\bar{u}^n) \le \tv(\bar{u}).
\end{align}
The approximate solution $u^n$ is then obtained by gluing together the approximate solutions computed by applying $\mathcal{R}^n$ at any discontinuity of $\bar{u}^n$ at time $t=0$ and at any interaction between wave-fronts.
In order to extend the construction globally in time we have to ensure that only finitely many interactions may occur in finite time and that the range of $u^n$ remains in $\Omega^n$.

The latter requirement immediately follows from~\ref{M3} in Proposition~\ref{prop:M}.
The former requirement is obtained through suitable interaction estimates, that also ensure a bound for $\tv(u^n(t))$ uniform in $n$ and $t$.

To this aim we fix $T>0$.
If $T$ is sufficiently small, then we know that the function $u^n(T)$ is piecewise constant with jumps along a finite number of polygonal lines.
If at time $T$ an interaction between waves takes place, then some of the waves may change speed and strength, while some new ones may be created, according to the solution of the Riemann problems at each  of the points of discontinuity of $u^n(T)$.
Conventionally, we assume that the approximate solutions are left continuous in time, i.e.~$u^n(T) = u^n(T^-)$.
Then also $\tv(u^n)$ is left continuous in time and, for any $t$ in a sufficiently small left neighbourhood of $T$, we can write
\begin{equation}\label{eq:approximatedsol}
	u^n(t,x) \doteq \sum_{i \in \mathcal{J}^n} u_{i+\frac{1}{2}}^n \, \caratt{[s_i^n(t),s_{i+1}^n(t)[}(x),
\end{equation}
where $\mathcal{J}^n \subset \integers$, $u_{i+1/2}^n \doteq (\rho_{i+1/2}^n,v_{i+1/2}^n) \in \Omega^n$, $s_{i-1}^n(t) < s_i^n(t)$ and
\begin{align}\label{eq:s}
	&s_i^n(t) \doteq x_i^n + \sigma(u_{i-\frac{1}{2}}^n, u_{i+\frac{1}{2}}^n) \left(t-T\right),&
	&\rho_{i-\frac{1}{2}}^n \ne \rho_{i+\frac{1}{2}}^n,
\end{align}
$\sigma$ being defined by~\eqref{eq:sigma}.

To prove that the approximate solution is well defined and keeps the form~\eqref{eq:approximatedsol} we have to bound a priori the number of waves. 
To this aim we prove that the map $t \mapsto \mathcal{T}^n(t)$ defined by
\begin{align*}
&\mathcal{T}^n \doteq \sum_i \left[
\modulo{w_1(u^n_{i+\frac{1}{2}}) - w_1(u^n_{i-\frac{1}{2}})}
+(1+\delta_i^n) \, \modulo{w_2(u^n_{i+\frac{1}{2}}) - w_2(u^n_{i-\frac{1}{2}})}
\right],
\\
&\delta_i^n \doteq \begin{cases}
1&\text{if } w_1(u^n_{i+\frac{1}{2}}) = w_1(u^n_{i-\frac{1}{2}}) \le V_{\rm c}
\text{ and } w_2(u^n_{i-\frac{1}{2}}) - w_2(u^n_{i+\frac{1}{2}}) > \varepsilon^n_w,\\
0&\text{otherwise},
\end{cases}
\end{align*}
is non-increasing and it decreases by at least $\varepsilon^n_w$ each time the number of waves increases. 
If at time $T$ no interaction occurs, then $\Delta\tv \doteq \tv(u^n(T^+)) - \tv(u^n(T^-)) = 0$, $\Delta\mathcal{T}^n \doteq \mathcal{T}^n(T^+) -\mathcal{T}^n(T^-) = 0$ and the number of waves does not change. 
Assume that at time $T$ an interaction occurs.
For notational convenience we will omit the dependence on $n$ and write, for instance, $\Delta\mathcal{T}$ instead of $\Delta\mathcal{T}^n$.

First, if the interaction involves states in the same phase domain, then it is sufficient to apply the standard theory to obtain that the number of waves does not increase after the interaction and $\Delta\mathcal{T} \leq 2 \, \Delta\tv \leq 0$.
In particular, in the domain of congested phases, $\Omega_{\rm c}$, the result follows from the fact that away from the vacuum ARZ is a Temple system, see~\cite{FerreiraKondo2010, godvik}.

Assume now that phase transitions are involved in the interaction.
For simplicity we describe in detail the interaction between two waves, one connecting $u_\ell$ to $u_m$ and the other connecting $u_m$ to $u_r$.
We have to distinguish the following cases:
\begin{enumerate}[label={(I.\arabic*)},leftmargin=*]

\item
If $u_\ell, u_m \in \Omega_{\rm f}$ and $u_r \in \Omega_{\rm c}$, then $w_2(u_r) = \max\{w_2(u_m), W_{\rm c}\}$ and by~\ref{LWR-ARZ2} the result of the interaction is a phase transition connecting $u_\ell$ to $u_m' \in \Omega_{\rm c}$, with $w_1(u_m') = v_r$, $w_2(u_m') = \max\{w_2(u_\ell), W_{\rm c}\}$, and a possible null contact discontinuity connecting $u_m'$ to $u_r$.
Moreover $\Delta\mathcal{T} = \Delta\tv \le 0$.

\item
If $u_\ell,u_r \in \Omega_{\rm f}$ and $u_m \in \Omega_{\rm c}$, then $u_\ell \in \Omega_{\rm f}'$, $w_2(u_r) = w_2(u_m) = W_{\rm c}$, $w_1(u_m) = V_{\rm c}$ and the result of the interaction is a shock connecting $u_\ell$ to $u_r$.
Moreover $\Delta\mathcal{T} = \Delta\tv < 0$.

\item
If $u_\ell \in \Omega_{\rm f}$ and $u_m,u_r \in \Omega_{\rm c}$, then $w_2(u_m) = w_2(u_r)$ and the result of the interaction is a phase transition connecting $u_\ell$ to $u_r$.
Moreover $\Delta\mathcal{T} = \Delta\tv \le 0$.

\item
If $u_m \in \Omega_{\rm f}$ and $u_\ell,u_r \in \Omega_{\rm c}$, then $w_1(u_\ell) = V_{\rm c} > w_1(u_r)$, $w_2(u_\ell) = w_2(u_m) = w_2(u_r)$ and the result of the interaction is a shock connecting $u_\ell$ to $u_r$.
Moreover $\Delta\mathcal{T} = \Delta\tv < 0$.

\item
If $u_r \in \Omega_{\rm f}$, $u_\ell,u_m \in \Omega_{\rm c}$ and $w_2(u_\ell) < w_2(u_m)$, then $w_1(u_\ell) = w_1(u_m) = V_{\rm c}$, $w_2(u_m) = w_2(u_r)$ and the result of the interaction is a phase transition connecting $u_\ell$ to $u_m' \in \Omega_{\rm f}''$, with $w_2(u_m') = w_2(u_\ell)$, and a shock connecting $u_m'$ to $u_r$.
Moreover $\Delta\mathcal{T} = \Delta\tv = 0$.

\item\label{I6} If $u_r \in \Omega_{\rm f}$, $u_\ell,u_m \in \Omega_{\rm c}$ and $w_2(u_\ell) > w_2(u_m)$, then $w_1(u_\ell) = w_1(u_m) = V_{\rm c}$, $w_2(u_m) = w_2(u_r)$ and the result of the interaction is a phase transition connecting $u_\ell$ to $u_m' \in \Omega_{\rm f}''$, with $w_2(u_m') = w_2(u_\ell)$, and a discretized rarefaction connecting $u_m'$ to $u_r$.
Moreover $\Delta\tv=0$ and
\[\Delta\mathcal{T} = 
\begin{cases}
-\left[w_2(u_\ell) - w_2(u_m)\right]&\text{if }w_2(u_\ell)-w_2(u_m)>\varepsilon_w,\\
0&\text{otherwise}.
\end{cases}\]

\end{enumerate}

In conclusion we proved that $\Delta\tv \le 0$.
Moreover, if after the interaction the number of waves does not increase then $\Delta\mathcal{T} \le 0$, otherwise, if after the interaction the number of waves increases, namely in the case~\ref{I6} with $w_2(u_\ell)-w_2(u_m)>\varepsilon_w$, then $\Delta\mathcal{T} = -\left[w_2(u_\ell) - w_2(u_m)\right] < - \varepsilon_w$.
As a consequence, the total variation as well as the number of waves of the approximate solution is bounded for any time and it keeps the form~\eqref{eq:approximatedsol}.

\subsection{Convergence}\label{sec:conv}

In this section we first prove the main a priori estimates on the approximate solution $u^n$ and then we prove that $u^n$ converges (up to a subsequence) in $\Lloc1(\R_+ \times \R; \Omega)$ to a function $u$ satisfying the estimates stated in Theorem~\ref{thm:1}.

Since at any interaction $\Delta\tv \le 0$, by~\eqref{eq:apprinitial} we have that for any $t \ge 0$
\[
\tv(u^n(t)) \le \tv(\bar{u}^n) \le \tv(\bar{u}).
\]
Moreover, observe that $\norma{u^n(t)}_{\L\infty(\R;\Omega)} \le C \doteq \max\{\modulo{W_{\max}},\modulo{W_{\min}}\} + V_{\max}$ and
\begin{equation}\label{eq:bagno}
\norma{u^n(t)-u^n(s)}_{\L1(\R;\Omega)} \le L \, \modulo{t-s},
\end{equation}
with $L \doteq \tv(\bar{u}) \, \max\{V_{\max}, R_{\max} \, p'(R_{\max})\}$.
Indeed, if no interaction occurs for times between $t$ and $s$, then
\[
\norma{u^n(t)-u^n(s)}_{\L1(\R;\Omega)} \le 
\sum_{i \in \mathcal{J}^n} \norma{(t-s) \, \dot{s}^n_i(t) \, (u^n_{i-\frac{1}{2}} - u^n_{i+\frac{1}{2}})}
\le L \, \modulo{t-s}.
\]
The case when one or more interactions take place for times between $t$ and $s$ is similar, because by the finite speed of propagation of the waves, the map $t \mapsto u^n(t)$ is $\L1$-continuous across interaction times.

Thus, by applying Helly's Theorem in the form~\cite[Theorem~2.4]{bressanbook}, there exists a function $u \in \Lloc1(\R_+\times\R;\Omega)$ and a subsequence, still denoted $(u^n)_n$, such that $(u^n)_n$ converges to $u$ in $\Lloc1(\R_+\times\R;\Omega)$ as $n$ goes to infinity. 
Moreover, $u$ satisfies the following estimates for any $t,s>0$:
\begin{align*}
&\tv(u(t)) \le \tv(\bar{u}),&
&\norma{u(t)}_{\L\infty(\R;\Omega)} \le C,&
&\norma{u(t)-u(s)}_{\L1(\R;\Omega)} \le L \, \modulo{t-s}.
\end{align*}
In particular, the above estimates ensure that the number of phase transitions performed by $x \mapsto u(t,x)$ is uniformly bounded in $t$.

\section{Technical section}\label{sec:ts}

\subsection{Proof of Lemma~\ref{lem:DMB}}\label{sec:DMB}

The approximate solution $u^n$ is constructed by applying the wave-front tracking method and the approximate Riemann solver $\mathcal{R}^n$, obtained from $\mathcal{R}$ given in Definition~\ref{def:LWR-ARZ} by discretizing the rarefactions, see Section~\ref{sec:exactpprR}.
To prove the property~\ref{entropyPT3} of Definition~\ref{def:entropyPT} for $u^n$, it is therefore sufficient to observe that the property is satisfied by $\mathcal{R}[u_\ell,u_r]$ for any $u_\ell,u_r \in \Omega$.

To prove that the number of phase transitions performed by $x\mapsto u^n(t,x)$, $t\ge0$, does not increase with time it is sufficient to observe that the approximate solution of any Riemann problem given by $\mathcal{R}^n$ involves at most one phase transition.
However, to complete the proof we have to study how the number of phase transitions changes after a wave interaction.
First, if the states involved in an interaction are in the same phase, then no new phase transition is created after the interaction because both $\Omega_{\rm c}$ and $\Omega_{\rm f}$ are invariant domains for $\mathcal{R}^n$.
Assume now that a phase transition is involved in the interaction.
We have to distinguish then the following cases:
\begin{figure}[ht]
\centering{
\begin{psfrags}
      \psfrag{r}[c,b]{$\rho$}
      \psfrag{f}[c,t]{$\rho\,v$}
      \psfrag{a}[c,b]{$u_\ell^1~$}
      \psfrag{b}[c,b]{$u_\ell^2~$}
      \psfrag{c}[c,b]{}
      \psfrag{d}[c,b]{}
      \psfrag{e}[c,b]{$u_\ell^h~$}
      \psfrag{g}[l,c]{$~u_r^1$}
      \psfrag{h}[c,b]{}
      \psfrag{i}[l,c]{$~u_r^k$}
      \psfrag{l}[c,b]{$u_m$}
\includegraphics[width=.3\textwidth]{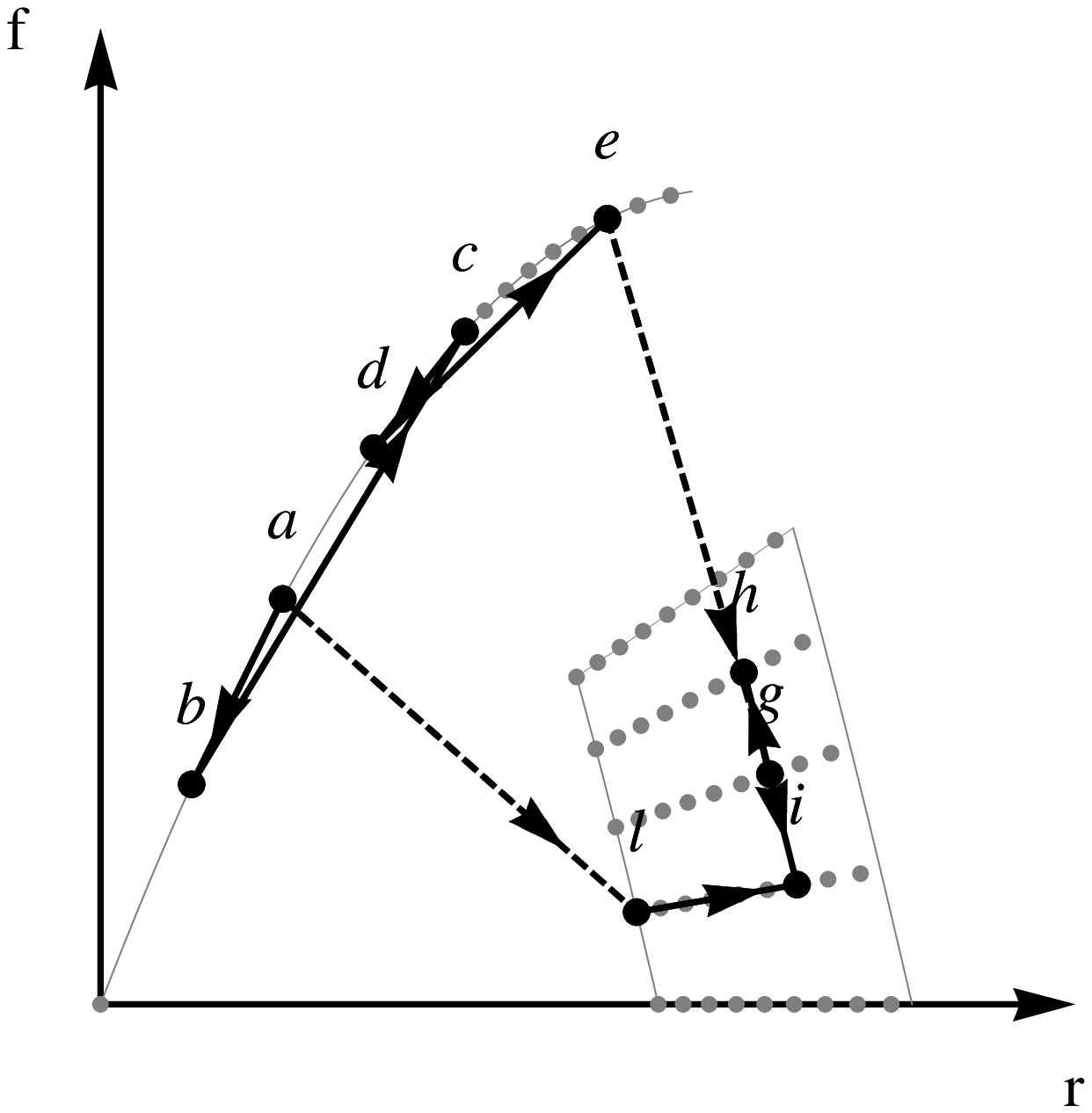}\qquad
\includegraphics[width=.3\textwidth]{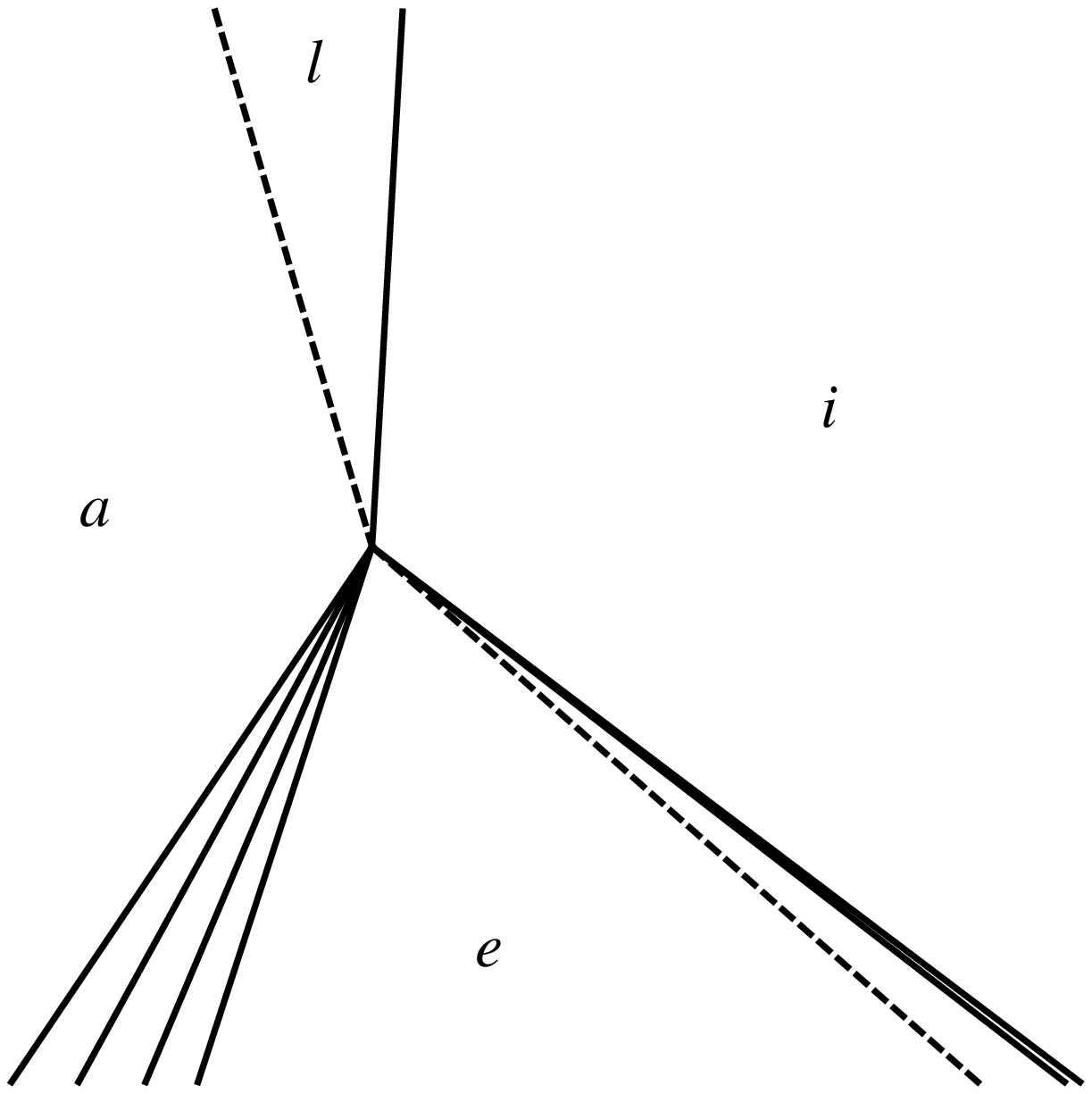}
\end{psfrags}}
\caption{An interaction involving only one phase transition from $\Omega_{\rm f}$ to $\Omega_{\rm c}$.
More precisely, the above interaction involves waves between states $u_\ell^i \in \Omega_{\rm f}$, $i=1,\ldots,h$, followed by a phase transition from $u_\ell^h$ to $u_r^1$ and represented with a dashed line, followed by waves between states $u_r^i \in \Omega_{\rm c}$, $i=1,\ldots,k$.
The result of this interaction is a phase transition from $u_\ell^1$ to $u_m = (p^{-1}(W_{\rm c}-v_r^k),v_r^k)$ and represented with a dashed line, followed by a contact discontinuity from $u_m$ to $u_r^k$.}
\label{fig:interactions01}
\end{figure}
\begin{itemize}[leftmargin=*]
\item Assume that only one phase transition is involved in the interaction and that it is from $\Omega_{\rm c}$ to $\Omega_{\rm f}$.
In this case the only possibility is that a single contact discontinuity between states in $\Omega_{\rm c}$ reaches the phase transition from the left.
The result of the interaction is then a phase transition from $\Omega_{\rm c}$ to $\Omega_{\rm f}$, followed by waves between states in $\Omega_{\rm f}$.
Therefore the number of phase transitions does not change.
\item Assume that only one phase transition is involved in the interaction and that it is from $\Omega_{\rm f}$ to $\Omega_{\rm c}$, see \figurename~\ref{fig:interactions01}.
Then the only possible waves between states in the same phase that can interact with such phase transition coming from the right must have the same second Riemann invariant coordinate $w_2$.
The result of the interaction is a phase transition from $\Omega_{\rm f}$ to $\Omega_{\rm c}$, followed by a possibly null contact discontinuity between states in $\Omega_{\rm c}$.
Therefore the number of phase transitions does not change.
\item Finally, assume that more than one phase transition is involved in the interaction.
By the previous study on the possible waves that can interact with a phase transition, we deduce that the only possible interaction involving more than one phase transition is between possible null waves in $\Omega_{\rm f}$, a phase transition from $\Omega_{\rm f}'$ to $u_{\rm c} \doteq (p^{-1}(W_{\rm c}-V_{\rm c}),V_{\rm c}) \in \Omega_{\rm c}$ and a phase transition from $u_{\rm c}$ to $(R_{\rm f}',v_{\rm f}(R_{\rm f}'))$.
The result of the interaction is then a possible null single shock.
Therefore the number of phase transitions decreases.
\end{itemize}
In conclusion, we proved that the number of phase transitions does not change as long as two phase transitions do not interact, while no phase transition results from the interaction between phase transitions.
For this reason the number of phase transitions can decrease only by an even number.

\subsection{Proof of Theorem~\ref{thm:1}}\label{sec:1}
Let $u$ be the function constructed in Section~\ref{sec:wft}.
The first part of the theorem is already proved in Section~\ref{sec:conv}.

We prove now that $u$ is a weak solution to the Cauchy problem~\eqref{eq:model}, \eqref{eq:initial} in the sense of Definition~\ref{def:weakPT} in the case the characteristic field corresponding to the free phase is linearly degenerate, namely under the assumption that
\begin{equation}\label{eq:AnimalsAsLeaders}
V_{\max} = V_{\rm f}.
\end{equation}
Clearly, the initial condition~\eqref{eq:initial} holds by~\eqref{eq:apprinitial}, \eqref{eq:bagno} and the convergence in $\Lloc1$ of $u^n$ to $u$.
Therefore, we just have to check the conditions~\ref{weakPT1}, \ref{weakPT2} and~\ref{weakPT3} of Definition~\ref{def:weakPT}.
Even if these conditions are satisfied by $u^n$, the $\Lloc1$-convergence doesn't ensure that the limit $u$ inherits these properties.
However, under the assumption~\eqref{eq:AnimalsAsLeaders} we can prove that $u$ satisfies the integral condition~\eqref{eq:weakPT}, which implies~\ref{weakPT1}, \ref{weakPT2} and~\ref{weakPT3}.

\begin{lemma}\label{lemma:AnimalsAsLeaders}
Under the assumption~\eqref{eq:AnimalsAsLeaders}, the function $u$ satisfies for any test function $\varphi$ in $\Cc\infty(]0,+\infty[\times\R;\R)$ the following identities
\begin{align}\label{eq:weakPT}
&\iint_{\R_+\times\R} \rho \left[\vphantom{\sum} \varphi_t+v \, \varphi_x\right] \d x \, \d t = 0,&
&\iint_{\R_+\times\R} \rho \, W(u) \left[\vphantom{\sum} \varphi_t+v \, \varphi_x\right] \d x \, \d t = 0,
\end{align}
where $W \in \C0(\Omega; [W_{\rm c},W_{\max}])$ is defined by
\begin{equation}\label{eq:W}
W(u) \doteq
\max\left \{w_2(u), W_{\rm c}\right \}
=
\begin{cases}
v+p(\rho)&\text{if }u\in\Omega_{\rm c} \cup \Omega_{\rm f}'',\\
W_{\rm c}&\text{if }u\in \Omega_{\rm f}'.
\end{cases}
\end{equation}
\end{lemma}
\begin{proof}
Fix a test function $\varphi$ in $\Cc\infty(]0,+\infty[\times\R;\R)$. 
Since $u^n$, $n \in \naturals$, are uniformly bounded and $W$ is uniformly continuous, it suffices to prove that
\[
\lim_{n\to+\infty}\norma{\iint_{\R_+\times\R} \rho^n \left[\varphi_t + v^n \, \varphi_x\right] \begin{pmatrix}1\\W(u^n)\end{pmatrix}\d x \, \d t} = 0.
\]
Choose $T>0$ such that $\varphi(t,x)=0$ whenever $t \not\in\left]0,T\right[$. 
With the same notation introduced in~\eqref{eq:approximatedsol}, by the Green-Gauss formula, the double integral above can be written as
\[
\int_0^T \sum_{i \in \mathcal{J}^n}
\left[
\dot{s}_i^n(t) \, \Delta Y_i(t) - \Delta F_i(t) \right]
\varphi(t,s_i^n(t)) \, \d t,
\]
where
\begin{align*}
&\Delta Y_i(t) \doteq
\rho_{i+\frac{1}{2}}^n \begin{pmatrix}1\\W(u^n_{i+\frac{1}{2}})\end{pmatrix}
-
\rho_{i-\frac{1}{2}}^n \begin{pmatrix}1\\W(u^n_{i-\frac{1}{2}})\end{pmatrix},&
&\Delta F_i(t) \doteq
\rho_{i+\frac{1}{2}}^n \, v_{i+\frac{1}{2}}^n \begin{pmatrix}1\\W(u^n_{i+\frac{1}{2}})\end{pmatrix}
-
\rho_{i-\frac{1}{2}}^n \, v_{i-\frac{1}{2}}^n \begin{pmatrix}1\\W(u^n_{i-\frac{1}{2}})\end{pmatrix}.
\end{align*}
By construction any discontinuity satisfies 
\begin{equation}\label{eq:Opeth}\tag{$\spadesuit$}
\norma{\dot{s}_i^n(t) \, \Delta Y_i(t) - \Delta F_i(t)} =0.
\end{equation}
Indeed, if $W(u^n_{i-1/2}) = W(u^n_{i+1/2})$, then it is sufficient to recall the definition of $s_i^n$ and $\sigma$ given respectively in~\eqref{eq:s} and~\eqref{eq:sigma}.
On the other hand, if $W(u^n_{i-1/2}) \neq W(u^n_{i+1/2})$, then by~\eqref{eq:AnimalsAsLeaders} we have that $v^n_{i-1/2} = v^n_{i+1/2}$,  and therefore $\dot{s}_i^n(t) = v^n_{i\pm1/2}$, which implies that~\eqref{eq:Opeth} holds true.
\end{proof}

To conclude that $u$ is a weak solution, it remains to prove that from~\eqref{eq:weakPT} we can deduce the conditions~\ref{weakPT1}, \ref{weakPT2} and~\ref{weakPT3} of Definition~\ref{def:weakPT}.
\begin{enumerate}[label={(W.\arabic*)},leftmargin=*]
\item[\ref{weakPT1}]
Fix a test function $\varphi$ in $\Cc\infty(]0,+\infty[\times\R;\R)$ with support in $u^{-1}(\Omega_{\rm f})$.
Then the equality~\eqref{eq:weakLWR} holds true because it coincides with the first condition given in~\eqref{eq:weakPT}.
\item[\ref{weakPT2}]
Fix a test function $\varphi$ in $\Cc\infty(]0,+\infty[\times\R;\R)$ with support in $u^{-1}(\Omega_{\rm c})$.
Then the equality~\eqref{eq:weakARZ} holds true because it coincides with~\eqref{eq:weakPT}, being $W \equiv w_2$ in $\Omega_{\rm c}$.
\item[\ref{weakPT3}]
Assume that $x \mapsto u(t,x)$, $t>0$, performs a phase transition from $u_-$ to $u_+$.
Then from~\eqref{eq:weakPT} we deduce that the speed of propagation $\sigma$ of the phase transition satisfies
\begin{align*}
&\left[\rho_+-\rho_-\right] \sigma = \rho_+ \, v_+-\rho_- \, v_-,&
&\left[\rho_+ \, W(u_+)-\rho_- \, W(u_-)\right] \sigma = \rho_+\, W(u_+) \, v_+-\rho_- \, W(u_-) \, v_-,
\end{align*}
or equivalently
\begin{align*}
&\rho_+\left[v_+-\sigma\right]=\rho_-\left[v_--\sigma\right],&
&\rho_-\left[v_--\sigma\right]\left[W(u_-)-W(u_+)\right]=0.
\end{align*}
By the definitions of $\Omega_{\rm f}$ and $\Omega_{\rm c}$ we have that $v_- \ne v_+$.
Hence from the above system we deduce that
\[
\rho_- \, \rho_+ \, [W(u_-)-W(u_+)] = 0.
\]
Now, to conclude the proof it is sufficient to observe that
\[
\mathcal{G}_{\rm w} =
\left\{\vphantom{\Omega_{\rm f}'} (u_-,u_+) \in \Omega\times\Omega \colon \rho_- \, \rho_+ \, [W(u_-)-W(u_+)] = 0\right \}.
\]
\end{enumerate}

Now, it only remains to prove that the function $u$ is an entropy solution to the Cauchy problem~\eqref{eq:model}, \eqref{eq:initial} in the sense of Definition~\ref{def:entropyPT}.
Clearly $u$ satisfies the condition~\ref{entropyPT1} of Definition~\ref{def:entropyPT}.
Indeed by~\eqref{eq:AnimalsAsLeaders} any discontinuity between states in $\Omega_{\rm f}$ is a contact discontinuity.
For this reason $u$ satisfies the entropy condition~\eqref{eq:entropyLWR} with the equality for any test function $\varphi$ in $\Cc\infty(]0,+\infty[\times\R;\R_+)$ with support in $u^{-1}(\Omega_{\rm f})$.
To prove the remaining conditions we need the following

\begin{lemma}\label{lemma:entropyaltPT}
Under the assumption~\eqref{eq:AnimalsAsLeaders}, the function $u$ satisfies for any non-negative test function $\varphi$ belonging to $\Cc\infty(]0,+\infty[\times\R;\R)$ and for any constant $k$ in $[0,V_{\rm f}]$ the estimate
\begin{equation}\label{eq:entropyPT}
\iint_{\R_+\times\R} \left[\mathcal{E}^k(u) \, \varphi_t + \mathcal{Q}^k(u) \, \varphi_x\right] \d x \, \d t \ge 0,
\end{equation}
where, see \figurename~\ref{fig:Rku},
\begin{align*}
&\mathcal{E}^k(u) \doteq
\begin{cases}
0&\text{if }v\le k,
\\
1-\dfrac{\rho}{R^k(W(u))}&\text{if }v>k,
\end{cases}&
&\mathcal{Q}^k(u) \doteq
\begin{cases}
0&\text{if }v\le k,
\\
k-\dfrac{\rho \, v}{R^k(W(u))}&\text{if }v>k,
\end{cases}
\end{align*}
\[
R^k(w) \doteq
\begin{cases}
p^{-1}(w-k)&\text{if }k \le V_{\rm c},
\\
\dfrac{V_{\rm f} - \sigma(p^{-1}(w-V_{\rm f}), V_{\rm f}, p^{-1}(w-V_{\rm c}),V_{\rm c})}{k - \sigma(p^{-1}(w-V_{\rm f}), V_{\rm f}, p^{-1}(w-V_{\rm c}),V_{\rm c}))} \, p^{-1}(w-V_{\rm f})&\text{if }k > V_{\rm c},
\end{cases}
\]
with $W$ given by~\eqref{eq:W}.
\end{lemma}

\begin{figure}[ht]
\centering{
\begin{psfrags}
      \psfrag{a}[c,B]{$\rho''$}
      \psfrag{b}[c,B]{$\quad f=\rho \, k''$}
      \psfrag{c}[c,B]{$\rho'$}
      \psfrag{d}[c,B]{$\quad f=\rho \, k'$}
      \psfrag{e}[c,B]{$\quad R_{\max}$}
      \psfrag{f}[c,B]{$\quad \rho$}
      \psfrag{g}[c,B]{$\rho\,v$}
      \psfrag{h}[c,B]{$\quad \rho_{\max}$}
\includegraphics[width=.3\textwidth]{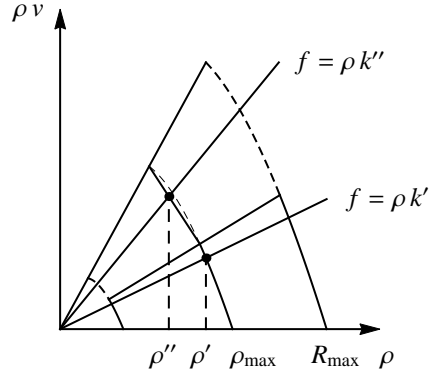}
\end{psfrags}}
\caption{The geometrical meaning of $R^k(w)$, $w\in[W_{\rm c},W_{\max}]$, given in Lemma~\ref{lemma:entropyaltPT}.
Above we let $\rho_{\max} = p^{-1}(w)$, $\rho' = R^{k'}(w)$ and $\rho'' = R^{k''}(w)$ for $0<k'<V_{\rm c}<k''<V_{\rm f}$.}
\label{fig:Rku}
\end{figure}

\begin{proof}
By the a.e.~convergence of $u^n$ to $u$ and the uniform continuity of $\mathcal{E}^k$ and $\mathcal{Q}^k$, in order to establish the above estimate, it is enough to prove that
\[
\liminf_{n\to+\infty} \iint_{\R_+\times\R} \left[\mathcal{E}^k(u^n) \, \varphi_t + \mathcal{Q}^k(u^n) \, \varphi_x\right] \d x \, \d t \ge 0.
\]
Choose $T>0$ such that $\varphi(t,x)=0$ whenever $t \not\in\left]0,T\right[$. 
With the same notation introduced in~\eqref{eq:approximatedsol}, by the Green-Gauss formula, the double integral above can be written as
\[
\int_0^T \sum_{i \in \mathcal{J}^n}
\Upsilon_i^k(t) \, \varphi(t,s_i^n(t)) \, \d t,
\]
where
\[
\Upsilon_i^k(t) \doteq \dot{s}_i^n(t) \left[\mathcal{E}^k(u_{i+\frac{1}{2}}^n) - \mathcal{E}^k(u_{i-\frac{1}{2}}^n)\right]  - \left[\mathcal{Q}^k(u_{i+\frac{1}{2}}^n) - \mathcal{Q}^k(u_{i-\frac{1}{2}}^n)\right].
\]
The case $k = V_{\rm f}$ is obvious because in this case $\Upsilon_i^k(t) = 0$.
For this reason, in the following we can assume that
\[
k < V_{\rm f}.
\]
To estimate the above integral, we have to consider separately the cases in which the $i$-th discontinuity is a phase transition, a shock, a discretized rarefaction or a contact discontinuity.
For notational simplicity we let $q^n_{i+1/2} \doteq \rho^n_{i+1/2} \, v^n_{i+1/2}$.

\begin{itemize}[leftmargin=*]
\item
Assume that the $i$-th discontinuity is a phase transition.
If $\rho^n_{i-1/2} = 0$, then $v^n_{i-1/2} = V_{\rm f} > k$, $\dot{s}_i^n(t) = v^n_{i+1/2} \le V_{\rm c}$ and
\begin{align*}
&k \ge v_{i+\frac{1}{2}}^n \Rightarrow
\Upsilon_i^k(t) =
-\dot{s}_i^n(t)  + k \ge 0,
\\
&k < v_{i+\frac{1}{2}}^n \Rightarrow
\Upsilon_i^k(t) =
-\dot{s}_i^n(t) \, \dfrac{\rho^n_{i+\frac{1}{2}}}{p^{-1}(W(u^n_{i+\frac{1}{2}})-k)}
+\dfrac{q^n_{i+\frac{1}{2}}}{p^{-1}(W(u^n_{i+\frac{1}{2}})-k)} = 0.
\end{align*} 
Assume now that $\rho^n_{i-1/2} \ne 0$.
Then $W(u^n_{i-1/2}) = W(u^n_{i+1/2})$ and we can let $\rho^k_n \doteq R^k(W(u^n_{i\pm1/2}))$.
\begin{itemize}[leftmargin=*]

\item
If $u^n_{i-1/2} \in \Omega_{\rm c}$, then $v^n_{i-1/2} = V_{\rm c} < v^n_{i+1/2} = V_{\rm f}$ and
\begin{align*}
k < V_{\rm c} \Rightarrow
\Upsilon_i^k(t) &=
\dot{s}_i^n(t) \, \dfrac{\rho^n_{i-\frac{1}{2}} - \rho^n_{i+\frac{1}{2}}}{\rho^k_n}
-\dfrac{q^n_{i-\frac{1}{2}} - q^n_{i+\frac{1}{2}}}{\rho^k_n} = 0,
\\
V_{\rm c} \le k < V_{\rm f} \Rightarrow
\Upsilon_i^k(t) &=
\dot{s}_i^n(t) \left[1-\dfrac{\rho^n_{i+\frac{1}{2}}}{\rho^k_n}\right] - \left[ k - \dfrac{q^n_{i+\frac{1}{2}}}{\rho^k_n}\right]
=
\dfrac{1}{\rho^k_n}
\left[
q^n_{i-\frac{1}{2}} 
+\dfrac{q^n_{i-\frac{1}{2}} - q^n_{i+\frac{1}{2}}}{\rho^n_{i-\frac{1}{2}} - \rho^n_{i+\frac{1}{2}}} \left( \rho^k_n - \rho^n_{i-\frac{1}{2}} \right)
-\rho^k_n \, k
\right]
=0.
\end{align*}

\item
If $u^n_{i-1/2} \in \Omega_{\rm f}$, then $v^n_{i-1/2} = V_{\rm f} > k$ and
\begin{align*}
k \ge v^n_{i+\frac{1}{2}} \Rightarrow
\Upsilon_i^k(t) &=
\dot{s}_i^n(t) \left[\dfrac{\rho^n_{i-\frac{1}{2}}}{\rho^k_n} - 1\right] - \left[ \dfrac{q^n_{i-\frac{1}{2}}}{\rho^k_n} - k\right]
=
\dfrac{1}{\rho^k_n}
\left[
\rho^k_n \, k
-q^n_{i-\frac{1}{2}} 
-\dfrac{q^n_{i-\frac{1}{2}} - q^n_{i+\frac{1}{2}}}{\rho^n_{i-\frac{1}{2}} - \rho^n_{i+\frac{1}{2}}} \left( \rho^k_n - \rho^n_{i-\frac{1}{2}} \right)
\right]
\ge 0,
\\
k < v^n_{i+\frac{1}{2}} \Rightarrow
\Upsilon_i^k(t) &=
\dot{s}_i^n(t) \, \dfrac{\rho^n_{i-\frac{1}{2}} - \rho^n_{i+\frac{1}{2}}}{\rho^k_n}
-\dfrac{q^n_{i-\frac{1}{2}} - q^n_{i+\frac{1}{2}}}{\rho^k_n} = 0.
\end{align*}

\end{itemize}

\item
If $u^n_{i-1/2}$ and $u^n_{i+1/2}$ are both in $\Omega_{\rm c}$, then it is sufficient to proceed as in the proof of~\cite[Proposition~5.2]{BCM2order} to obtain for any $k \le V_{\rm c}$ that
\begin{align*}
&v^n_{i+\frac{1}{2}} < v^n_{i-\frac{1}{2}}
\Rightarrow
\Upsilon_i^k(t) \ge 0,
\\
&v^n_{i+\frac{1}{2}} > v^n_{i-\frac{1}{2}}
\Rightarrow
\Upsilon_i^k(t) \ge 
-\max_{\rho \in \left[p^{-1}\left(W_{\rm c}-V_{\rm c}\right),R_{\max}\right]}\left[2+ \frac{\rho \, p''\left(\rho\right)}{p'\left(\rho\right)}\right]
\left[\vphantom{\frac{\rho \, p''\left(\rho\right)}{p'\left(\rho\right)}} v^n_{i+\frac{1}{2}} - v^n_{i-\frac{1}{2}}\right],
\\
&v^n_{i+\frac{1}{2}} = v^n_{i-\frac{1}{2}}
\Rightarrow
\Upsilon_i^k(t) = 0.
\end{align*}
If $k > V_{\rm c}$, then $v^n_{i-1/2}$ and $v^n_{i+1/2}$ are both less than $k$ and therefore $\Upsilon_i^k(t) = 0$.

\item
If $u^n_{i-1/2}$ and $u^n_{i+1/2}$ are both in $\Omega_{\rm f}$, then $k < V_{\rm f} = v^n_{i-1/2} = v^n_{i+1/2}$ and
\[
\Upsilon_i^k(t) =
\dot{s}_i^n(t)
\left[\frac{\rho^n_{i-\frac{1}{2}}}{R^k(W(u^n_{i-\frac{1}{2}}))} - \frac{\rho^n_{i+\frac{1}{2}}}{R^k(W(u^n_{i+\frac{1}{2}}))}\right]
-\left[\frac{q^n_{i-\frac{1}{2}}}{R^k(W(u^n_{i-\frac{1}{2}}))} - \frac{q^n_{i+\frac{1}{2}}}{R^k(W(u^n_{i+\frac{1}{2}}))}\right]
= 0
\]
because $\dot{s}_i^n(t) = V_{\rm f} = v^n_{i-1/2} = v^n_{i+1/2}$.

\end{itemize}

\noindent In conclusion we proved that
\begin{align*}
&\liminf_{n\to+\infty} \iint_{\R_+\times\R} \left[\mathcal{E}^k(u^n) \, \varphi_t + \mathcal{Q}^k(u^n) \, \varphi_x\right] \d x \, \d t
\geq
\liminf_{n\to+\infty} \int_0^T \sum_{i \in \mathcal{R}^n(t)}
\Upsilon_i^k(t) \, \varphi(t,s_i^n(t)) \, \d t,
\end{align*}
where $\mathcal{R}^n(t)$ is the set of indexes corresponding to discretized rarefactions in $\Omega_{\rm c}$.
Finally, by proceeding as in the proof of~\cite[Proposition~5.2]{BCM2order}, we have that
\[
\liminf_{n\to+\infty} \int_0^T \sum_{i \in \mathcal{R}^n(t)}
\Upsilon_i^k(t) \, \varphi(t,s_i^n(t)) \, \d t
=0,
\]
and this concludes the proof.
\end{proof}
\noindent
Condition~\ref{entropyPT2} in Definition~\ref{def:entropyPT} follows directly from the above lemma.
Indeed, $W \equiv w_2$ in $\Omega_{\rm c}$ and $(\mathcal{E}_{\rm ARZ}^k , \mathcal{Q}_{\rm ARZ}^k) \equiv (\mathcal{E}^k , \mathcal{Q}^k)$ on $\Omega_{\rm c}$ for any $k \in [0,V_{\rm c}]$.
We finally prove that $u$ satisfies also the condition~\ref{entropyPT3} of Definition~\ref{def:entropyPT}.
Assume that $x \mapsto u(t,x)$, $t>0$, performs a phase transition from $u_-$ to $u_+$.
We have to prove that $(u_-,u_+)$ belongs to $\mathcal{G}_{\rm e}$.
We already know that $(u_-,u_+)$ belongs to $\mathcal{G}_{\rm w}$.
For this reason, by comparing $\mathcal{G}_{\rm e}$ with $\mathcal{G}_{\rm w}$, it is clear that it is sufficient to prove that if $u_- \in \Omega_{\rm c}$, then $u_+ \in \Omega_{\rm f}''$ and $v_- = V_{\rm c}$.
We first recall that, by Lemma~\ref{lemma:AnimalsAsLeaders}, the speed of propagation of the phase transition is $\sigma(u_-,u_+)$ defined  by~\eqref{eq:sigma}.
Hence, by the above lemma we have that
\begin{equation}\label{eq:PorcupineTree}\tag{$\clubsuit$}
\sigma(u_-,u_+) \left[\mathcal{E}^k(u_+) - \mathcal{E}^k(u_-)\right]  - \left[\mathcal{Q}^k(u_+) - \mathcal{Q}^k(u_-)\right] \ge 0.
\end{equation}
Assume now by contradiction that $u_+ \in \Omega_{\rm f}'$.
Then $W(u_+) = W_{\rm c}$, $\rho_+ \left[w_2(u_-)-W_{\rm c}\right] = 0$ and for any $k$ in $]V_{\rm c},V_{\rm f}[$ we would have that $\rho^{\rm c}_k \doteq R^k(W_{\rm c})$ would satisfy
\[
\sigma(u_-,u_+) \left[1 - \frac{\rho_+}{\rho^{\rm c}_k}\right]  - \left[k - \frac{\rho_+ \, v_+}{\rho^{\rm c}_k}\right]
=
\frac{1}{\rho^{\rm c}_k}
\left[\vphantom{\frac{\rho_+}{\rho^{\rm c}_k}}
\rho_+ \, v_+
+
\sigma(u_-,u_+) \left(\rho^{\rm c}_k - \rho_+\right)
-
\rho^{\rm c}_k \, k
\right]
<0,
\]
which contradicts~\eqref{eq:PorcupineTree}.
Now, from the fact that $(u_-,u_+)$ belongs to $\mathcal{G}_{\rm w} \cap (\Omega_{\rm c} \times \Omega_{\rm f}'')$, we immediately deduce that $W(u_-) = W(u_+) = w_2(u_\pm)$.
Finally, if by contradiction $v_- < V_{\rm c}$, then $\rho^{\rm c} \doteq p^{-1}(w_2(u_\pm)-V_{\rm c})$ would satisfy
\[
\sigma(u_-,u_+) \left[1 - \frac{\rho_+}{\rho^{\rm c}}\right]  - \left[V_{\rm c} - \frac{\rho_+ \, v_+}{\rho^{\rm c}}\right]
=
\frac{1}{\rho^{\rm c}}
\left[\vphantom{\frac{\rho_+}{\rho^{\rm c}}}
\rho_+ \, v_+
+
\sigma(u_-,u_+) \left(\rho^{\rm c} - \rho_+\right)
-
\rho^{\rm c} \, V_{\rm c}
\right]
<0,
\]
which contradicts~\eqref{eq:PorcupineTree} with $k = V_{\rm c}$.
In conclusion we proved that $(u_-,u_+) \in \mathcal{G}_{\rm e}$ and that $u$ satisfies also the condition~\ref{entropyPT3} of Definition~\ref{def:entropyPT}.

\section*{Acknowledgment}

The first author warmly thanks the Faculty of Mathematics, Physics and Computer Science, Maria Curie-Sk{\l}odowska-University, for the hospitality during the preparation of this paper.
The last author warmly thanks Gran Sasso Science Institute for the hospitality during the preparation of this paper.
All the authors wish to thank Boris Andreianov and Carlotta Donadello for stimulating discussions.

\section*{References}

\bibliographystyle{elsarticle-harv}
\bibliography{biblio}

\begin{thebibliography}{54}
\expandafter\ifx\csname natexlab\endcsname\relax\def\natexlab#1{#1}\fi
\expandafter\ifx\csname url\endcsname\relax
  \def\url#1{\texttt{#1}}\fi
\expandafter\ifx\csname urlprefix\endcsname\relax\def\urlprefix{URL }\fi

\bibitem[{Andreianov et~al.(2015)Andreianov, Donadello, Razafison, and
  Rosini}]{AndreianovDonadelloRosiniMBE}
Andreianov, B., Donadello, C., Razafison, U., Rosini, M.~D., 2015. Riemann
  problems with non--local point constraints and capacity drop. Mathematical
  Biosciences and Engineering 12~(2), 259--278.

\bibitem[{Andreianov et~al.(0)Andreianov, Donadello, and Rosini}]{BCM2order}
Andreianov, B., Donadello, C., Rosini, M.~D., 0. A second-order model for
  vehicular traffics with local point constraints on the flow. Mathematical
  Models and Methods in Applied Sciences 0~(ja), null.

\bibitem[{Andreianov et~al.(2014)Andreianov, Donadello, and
  Rosini}]{AndreianovDonadelloRosini1}
Andreianov, B., Donadello, C., Rosini, M.~D., 2014. Crowd dynamics and
  conservation laws with nonlocal constraints and capacity drop. Math. Models
  Methods Appl. Sci. 24~(13), 2685--2722.

\bibitem[{Andreianov et~al.(2010)Andreianov, Goatin, and Seguin}]{scontrainte}
Andreianov, B., Goatin, P., Seguin, N., 2010. Finite volume schemes for locally
  constrained conservation laws. Numerische Mathematik 115, 609--645.

\bibitem[{Andreianov et~al.(2016)Andreianov, Donadello, Razafison, Rolland, and
  Rosini}]{BCJMU-order2}
Andreianov, B.~P., Donadello, C., Razafison, U., Rolland, J.~Y., Rosini, M.~D.,
  2016. Solutions of the aw-rascle-zhang system with point constraints.
  Networks and Heterogeneous Media 11~(1), 29--47.

\bibitem[{{Andreianov, Boris} et~al.(2015){Andreianov, Boris}, {Donadello,
  Carlotta}, {Razafison, Ulrich}, and {Rosini, Massimiliano
  D.}}]{BCMUqualitative}
{Andreianov, Boris}, {Donadello, Carlotta}, {Razafison, Ulrich}, {Rosini,
  Massimiliano D.}, 2015. Qualitative behaviour and numerical approximation of
  solutions to conservation laws with non-local point constraints on the flux
  and modeling of crowd dynamics at the bottlenecks. ESAIM: M2AN.

\bibitem[{Aw et~al.(2002)Aw, Klar, Materne, and Rascle}]{AKMR2002}
Aw, A., Klar, A., Materne, T., Rascle, M., 2002. Derivation of continuum
  traffic flow models from microscopic follow-the-leader models. SIAM Journal
  on Applied Mathematics 63~(1), 259--278.

\bibitem[{Aw and Rascle(2000)}]{ARZ1}
Aw, A., Rascle, M., 2000. Resurrection of ``second order'' models of traffic
  flow. SIAM Journal on Applied Mathematics 60~(3), 916--938.

\bibitem[{Bagnerini et~al.(2006)Bagnerini, Colombo, and
  Corli}]{BagneriniColomboCorli}
Bagnerini, P., Colombo, R.~M., Corli, A., 2006. On the role of source terms in
  continuum traffic flow models. Mathematical and Computer Modelling
  44~(9–10), 917 -- 930.

\bibitem[{Bellomo et~al.(2002)Bellomo, Delitala, and Coscia}]{Bellomo2002}
Bellomo, N., Delitala, M., Coscia, V., 2002. On the mathematical theory of
  vehicular traffic flow. {I}. {F}luid dynamic and kinetic modelling. Math.
  Models Methods Appl. Sci. 12~(12), 1801--1843.

\bibitem[{Bellomo and Dogbe(2011)}]{bellomo2011modeling}
Bellomo, N., Dogbe, C., 2011. On the modeling of traffic and crowds: a survey
  of models, speculations, and perspectives. SIAM Rev. 53~(3), 409--463.

\bibitem[{Berthelin et~al.(2008)Berthelin, Degond, Delitala, and
  Rascle}]{BerthelinDegondDelitalaRascle}
Berthelin, F., Degond, P., Delitala, M., Rascle, M., 2008. A model for the
  formation and evolution of traffic jams. Archive for Rational Mechanics and
  Analysis 187~(2), 185--220.

\bibitem[{Blandin et~al.(2011)Blandin, Work, Goatin, Piccoli, and
  Bayen}]{BlandinWorkGoatinPiccoliBayen}
Blandin, S., Work, D., Goatin, P., Piccoli, B., Bayen, A., 2011. A general
  phase transition model for vehicular traffic. SIAM J. Appl. Math. 71~(1),
  107--127.

\bibitem[{Borsche et~al.(2012)Borsche, Kimathi, and Klar}]{BorscheKimathiKlar}
Borsche, R., Kimathi, M., Klar, A., 2012. A class of multi-phase traffic
  theories for microscopic, kinetic and continuum traffic models. Computers \&
  Mathematics with Applications 64~(9), 2939 -- 2953.

\bibitem[{Bressan(2000)}]{bressanbook}
Bressan, A., 2000. Hyperbolic systems of conservation laws. Vol.~20 of Oxford
  Lecture Series in Mathematics and its Applications. Oxford University Press,
  Oxford, the one-dimensional Cauchy problem.

\bibitem[{Canc\`es and Seguin(2012)}]{Cances20123036}
Canc\`es, C., Seguin, N., 2012. {Error Estimate for Godunov Approximation of
  Locally Constrained Conservation Laws}. SIAM Journal on Numerical Analysis
  50~(6), 3036--3060.

\bibitem[{Chalons et~al.(2013)Chalons, Goatin, and
  Seguin}]{chalonsgoatinseguin}
Chalons, C., Goatin, P., Seguin, N., 2013. {General constrained conservation
  laws. Application to pedestrian flow modeling}. Networks and Heterogeneous
  Media 8~(2), 433--463.

\bibitem[{Colombo(2002)}]{Colombophasetransitions}
Colombo, R.~M., 2002. Hyperbolic phase transitions in traffic flow. SIAM J.
  Appl. Math. 63~(2), 708--721 (electronic).

\bibitem[{Colombo and Goatin(2007)}]{ColomboGoatinConstraint}
Colombo, R.~M., Goatin, P., 2007. A well posed conservation law with a variable
  unilateral constraint. J. Differential Equations 234~(2), 654--675.

\bibitem[{Colombo et~al.(2007)Colombo, Goatin, and
  Priuli}]{ColomboGoatinPriuli}
Colombo, R.~M., Goatin, P., Priuli, F.~S., 2007. Global well posedness of
  traffic flow models with phase transitions. Nonlinear Anal. 66~(11),
  2413--2426.

\bibitem[{Colombo et~al.(2010{\natexlab{a}})Colombo, Goatin, and
  Rosini}]{ColomboGoatinRosini1}
Colombo, R.~M., Goatin, P., Rosini, M.~D., 2010{\natexlab{a}}. Conservation
  laws with unilateral constraints in traffic modeling. In: Applied and
  industrial mathematics in {I}taly {III}. Vol.~82 of Ser. Adv. Math. Appl.
  Sci. World Sci. Publ., Hackensack, NJ, pp. 244--255.

\bibitem[{Colombo et~al.(2010{\natexlab{b}})Colombo, Goatin, and
  Rosini}]{ColomboGoatinRosini3}
Colombo, R.~M., Goatin, P., Rosini, M.~D., 2010{\natexlab{b}}. A macroscopic
  model for pedestrian flows in panic situations. In: Current advances in
  nonlinear analysis and related topics. Vol.~32 of GAKUTO Internat. Ser. Math.
  Sci. Appl. Gakk\=otosho, Tokyo, pp. 255--272.

\bibitem[{Colombo et~al.(2011)Colombo, Goatin, and
  Rosini}]{ColomboGoatinRosini2}
Colombo, R.~M., Goatin, P., Rosini, M.~D., 2011. On the modelling and
  management of traffic. ESAIM Math. Model. Numer. Anal. 45~(5), 853--872.

\bibitem[{Colombo and Marcellini(2015)}]{ColomboMarcellini2015}
Colombo, R.~M., Marcellini, F., 2015. A mixed {ODE-PDE} model for vehicular
  traffic. Mathematical Methods in the Applied Sciences 38~(7), 1292--1302.

\bibitem[{Colombo et~al.(2010{\natexlab{c}})Colombo, Marcellini, and
  Rascle}]{ColomboMarcelliniRascle}
Colombo, R.~M., Marcellini, F., Rascle, M., 2010{\natexlab{c}}. A 2-phase
  traffic model based on a speed bound. SIAM J. Appl. Math. 70~(7), 2652--2666.

\bibitem[{Colombo and Rosini(2005)}]{ColomboRosiniPedestrain1}
Colombo, R.~M., Rosini, M.~D., 2005. Pedestrian flows and non-classical shocks.
  Math. Methods Appl. Sci. 28~(13), 1553--1567.

\bibitem[{Dafermos(1972)}]{DafermosWFT}
Dafermos, C.~M., 1972. Polygonal approximations of solutions of the initial
  value problem for a conservation law. J. Math. Anal. Appl. 38, 33--41.

\bibitem[{Delle~Monache and Goatin(2014{\natexlab{a}})}]{DelleMonache2014435}
Delle~Monache, M., Goatin, P., 2014{\natexlab{a}}. A front tracking method for
  a strongly coupled pde-ode system with moving density constraints in traffic
  flow. Discrete and Continuous Dynamical Systems - Series S 7~(3), 435--447.

\bibitem[{Delle~Monache and Goatin(2014{\natexlab{b}})}]{delle2014scalar}
Delle~Monache, M.~L., Goatin, P., 2014{\natexlab{b}}. Scalar conservation laws
  with moving constraints arising in traffic flow modeling: an existence
  result. Journal of Differential equations 257~(11), 4015--4029.

\bibitem[{{Di Francesco} et~al.(2015){Di Francesco}, {Fagioli}, and
  {Rosini}}]{MarcoSimoneMax-ARZ1}
{Di Francesco}, M., {Fagioli}, S., {Rosini}, M.~D., Nov. 2015. {Many particle
  approximation of the Aw-Rascle-Zhang second order model for vehicular
  traffic}. ArXiv e-prints.

\bibitem[{Fan et~al.(2014)Fan, Herty, and Seibold}]{FanHertySeibold}
Fan, S., Herty, M., Seibold, B., 2014. Comparative model accuracy of a
  data-fitted generalized aw-rascle-zhang model. Networks and Heterogeneous
  Media 9~(2), 239--268.

\bibitem[{Ferreira and Kondo(2010)}]{FerreiraKondo2010}
Ferreira, R.~E., Kondo, C.~I., 2010. Glimm method and wave-front tracking for
  the {A}w-{R}ascle traffic flow model. Far East J. Math. Sci. (FJMS) 43~(2),
  203--223.

\bibitem[{Garavello and Piccoli(2006)}]{garavello2006traffic}
Garavello, M., Piccoli, B., 2006. Traffic flow on networks. Vol.~1 of AIMS
  Series on Applied Mathematics. American Institute of Mathematical Sciences
  (AIMS), Springfield, MO, conservation laws models.

\bibitem[{Garavello and Piccoli(2009)}]{GaravelloPiccoli2009}
Garavello, M., Piccoli, B., 2009. On fluido-dynamic models for urban traffic.
  Netw. Heterog. Media 4~(1), 107--126.

\bibitem[{Goatin(2006)}]{goatin2006aw}
Goatin, P., 2006. The {A}w--{R}ascle vehicular traffic flow model with phase
  transitions. Mathematical and computer modelling 44~(3), 287--303.

\bibitem[{Godvik and Hanche-Olsen(2008)}]{godvik}
Godvik, M., Hanche-Olsen, H., 2008. Existence of solutions for the
  {A}w-{R}ascle traffic flow model with vacuum. J. Hyperbolic Differ. Equ.
  5~(1), 45--63.

\bibitem[{Greenshields(1935)}]{Greenshields}
Greenshields, B., 1935. A study of traffic capacity. Proceedings of the Highway
  Research Board 14, 448--477.

\bibitem[{Kerner(2004)}]{kerner2004physics}
Kerner, B., 2004. The Physics of Traffic: Empirical Freeway Pattern Features,
  Engineering Applications, and Theory. Understanding Complex Systems. Springer
  Berlin Heidelberg.

\bibitem[{Kruzhkov(1970)}]{kruzkov}
Kruzhkov, S.~N., 1970. First order quasilinear equations with several
  independent variables. Mat. Sb. (N.S.) 81 (123), 228--255.

\bibitem[{Lattanzio et~al.(2011)Lattanzio, Maurizi, and
  Piccoli}]{Lattanzio201150}
Lattanzio, C., Maurizi, A., Piccoli, B., 2011. Moving bottlenecks in car
  traffic flow: a {PDE}-{ODE} coupled model. SIAM J. Math. Anal. 43~(1),
  50--67.

\bibitem[{Laval(2011)}]{hysteresis}
Laval, J.~A., 2011. Hysteresis in traffic flow revisited: An improved
  measurement method. Transportation Research Part B: Methodological 45~(2),
  385 -- 391.

\bibitem[{LeFloch(2002)}]{LeflochBook}
LeFloch, P.~G., 2002. Hyperbolic systems of conservation laws. Lectures in
  Mathematics ETH Z\"urich. Birkh\"auser Verlag, Basel, the theory of classical
  and nonclassical shock waves.

\bibitem[{Lighthill and Whitham(1955)}]{LWR1}
Lighthill, M., Whitham, G., 1955. On kinematic waves. {II.} {A} theory of
  traffic flow on long crowded roads. In: Royal Society of London. Series A,
  Mathematical and Physical Sciences. Vol. 229. pp. 317--345.

\bibitem[{Lu(2011)}]{Lu20112797}
Lu, Y.-g., 2011. Existence of global bounded weak solutions to nonsymmetric
  systems of {K}eyfitz-{K}ranzer type. J. Funct. Anal. 261~(10), 2797--2815.

\bibitem[{Marcellini(2014)}]{Marcellini2014}
Marcellini, F., 2014. Free-congested and micro-macro descriptions of traffic
  flow. Discrete Contin. Dyn. Syst. Ser. S 7~(3), 543--556.

\bibitem[{Mohan and Ramadurai(2013)}]{survey2013}
Mohan, R., Ramadurai, G., 2013. State-of-the art of macroscopic traffic flow
  modelling. Int. J. Adv. Eng. Sci. Appl. Math. 5~(2-3), 158--176.

\bibitem[{Pan and Han(2013)}]{PanHan}
Pan, L., Han, X., 2013. The {A}w-{R}ascle traffic model with {C}haplygin
  pressure. J. Math. Anal. Appl. 401~(1), 379--387.

\bibitem[{Piccoli and Tosin(2012)}]{piccolitosinreview}
Piccoli, B., Tosin, A., 2012. Vehicular traffic: a review of continuum
  mathematical models. In: Mathematics of complexity and dynamical systems.
  {V}ols. 1--3. Springer, New York, pp. 1748--1770.

\bibitem[{Richards(1956)}]{LWR2}
Richards, P.~I., 1956. Shock waves on the highway. Operations Research 4~(1),
  pp. 42--51.

\bibitem[{Rosini(2013{\natexlab{a}})}]{Rosini201363}
Rosini, M.~D., 2013{\natexlab{a}}. The initial-boundary value problem and the
  constraint. Understanding Complex Systems, 63--91.

\bibitem[{Rosini(2013{\natexlab{b}})}]{Rosinibook}
Rosini, M.~D., 2013{\natexlab{b}}. Macroscopic models for vehicular flows and
  crowd dynamics: theory and applications. Understanding Complex Systems.
  Springer, Heidelberg, classical and non-classical advanced mathematics for
  real life applications.

\bibitem[{van Wageningen-Kessels et~al.(2014)van Wageningen-Kessels, van Lint,
  Vuik, and Hoogendoorn}]{survey2014}
van Wageningen-Kessels, F., van Lint, H., Vuik, K., Hoogendoorn, S., 2014.
  Genealogy of traffic flow models. EURO Journal on Transportation and
  Logistics, 1--29.

\bibitem[{{Vol'pert}(1968)}]{Volpert}
{Vol'pert}, A., 1968. {The spaces BV and quasilinear equations}. {Math. USSR,
  Sb.} 2, 225--267.

\bibitem[{Zhang(2002)}]{ARZ2}
Zhang, H., 2002. A non-equilibrium traffic model devoid of gas-like behavior.
  Transportation Research Part B: Methodological 36~(3), 275--290.

\end{thebibliography}





\end{document}